\documentclass[12pt]{article}

\usepackage{amsmath}
\usepackage{amsthm}
\usepackage{amssymb}
\usepackage{amscd}
\usepackage{xspace}
\usepackage{verbatim}
\usepackage{geometry}
\usepackage{esint}
\usepackage{geometry}
\usepackage{fancyhdr}
\usepackage{xargs}
\usepackage{xifthen}
\newcommandx{\Set}[2][2=]{
    \ifthenelse{\isempty{#2}}
        {\left\{ {#1} \right\}}
        {\left\{  {#1}  \, \middle| \, {#2} \right\}}
}

\newtheorem{theorem}{Theorem}[section]
\newtheorem*{theorem*}{Theorem}
\newtheorem{corollary}{Corollary}[section]
\newtheorem{lemma}{Lemma}[section]
\newtheorem{proposition}{Proposition}[section]
\newtheorem*{proposition*}{Proposition}

\newtheorem{notation}{Notation}[section]
\theoremstyle{definition}
\newtheorem{definition}{Definition}[section]

\theoremstyle{remark}
\newtheorem{remark}{Remark}[section]
\numberwithin{equation}{section}
\newtheorem{question}{Question}[section]

\newcommand{\e}{\varepsilon}

\newcommand{\N}{\mathbb N}

\newcommand{\R}{\mathbb{R}}

\newcommand{\er}{\eqref}

 \DeclareMathOperator{\dist}{dist}
\DeclareMathOperator{\supp}{supp}

\DeclareMathOperator*{\esssup}{ess\,sup}
\DeclareMathOperator*{\essinf}{ess\,inf}
\DeclareMathOperator*{\esslimsup}{ess\,limsup}
\DeclareMathOperator*{\essliminf}{ess\,liminf}

\newcommandx{\norm}[1][1=\cdot]{\left|{#1}\right|}
\newcommandx{\supnorm}[2][1=\cdot,2=]{\nnorm[#1]_{\infty,#2}}
\newcommandx{\nnorm}[1][1=\cdot]{\left|\left|{#1}\right|\right|}

\newcommand{\Haus}{\mathcal H}

\newcommand{\Leb}{\mathcal L}

\date{\today}

\begin{document}
\title{Approximations in Besov Spaces and Jump Detection of Besov Functions with Bounded Variation}
\maketitle
\begin{center}
{\large Paz Hashash, Arkady Poliakovsky}
\end{center}

\begin{abstract}
In this paper, we provide a proof that functions belonging to Besov spaces $B^{r}_{q,\infty}(\mathbb{R}^N,\mathbb{R}^d)$, $q\in [1,\infty)$, $r\in(0,1)$, satisfy the following formula under a certain condition:

\begin{equation}
\label{eq:main result in abstract}
\lim_{{\epsilon}\to 0^+}\frac{1}{|\ln{\epsilon}|}\left[u_{\epsilon}\right]^q_{W^{r,q}(\mathbb{R}^N,\mathbb{R}^d)}=N\lim_{{\epsilon}\to 0^+}\int_{\mathbb{R}^N}\frac{1}{{\epsilon}^N}\int_{B_{\epsilon}(x)}\frac{|u(x)-u(y)|^q}{|x-y|^{rq}}dydx.
\end{equation}
Here, $\left[\cdot\right]_{W^{r,q}}$ represents the Gagliardo seminorm, and $u_{\epsilon}$ denotes the convolution of $u$ with a mollifier $\eta_{(\epsilon)}(x):=\frac{1}{\epsilon^N}\eta\left(\frac{x}{\epsilon}\right)$, $\eta\in W^{1,1}(\R^N),\int_{\R^N}\eta(z)dz=1$. Furthermore, we prove that every function $u$ in $BV(\mathbb{R}^N,\mathbb{R}^d)\cap B^{1/p}_{p,\infty}(\mathbb{R}^N,\mathbb{R}^d),p\in(1,\infty),$ satisfies
\begin{multline}
\lim_{\epsilon\to
0^+}\frac{1}{|\ln{\epsilon}|}\left[u_{\epsilon}\right]^q_{W^{1/q,q}(\R^N,\R^d)}=N\lim_{{\epsilon}\to 0^+}\int_{\mathbb{R}^N}\frac{1}{{\epsilon}^N}\int_{B_{\epsilon}(x)}\frac{|u(x)-u(y)|^q}{|x-y|}dydx
\\
=\left(\int_{S^{N-1}}|z_1|~d\Haus^{N-1}(z)\right)\int_{\mathcal{J}_u}
\Big|u^+(x)-u^-(x)\Big|^q d\mathcal{H}^{N-1}(x),
\end{multline}
for every $1<q<p$. Here $u^+,u^-$ are the one-sided approximate limits of $u$ along the jump set $\mathcal{J}_u$.
\tableofcontents
\end{abstract}

\section{Introduction}
The so-called 'BBM formula', as presented by Bourgain, Brezis, and
Mironescu in \cite{BBM}, provides a characterization of Sobolev
functions $W^{1,q}(\Omega)$ for $1<q<\infty$ and of functions of
bounded variation $BV(\Omega)$ using double integrals and
mollifiers, where $\Omega\subset\mathbb{R}^N$ is an open and bounded
set with a Lipschitz boundary. The full characterization for
$BV(\Omega)$ functions is attributed to D{\'a}vila \cite{Davila}.
Before describing it, let's recall some definitions.

\begin{definition}(Decreasing Support Property)
\label{def:decreasing support property}

Let $a\in(0,\infty]$ and $\rho_\e:(0,\infty)\to [0,\infty),\e\in
(0,a),$ be a family of $\mathcal{L}^1$-measurable functions. We say
that the family $\{\rho_\e\}_{\e\in (0,a)}$ has the {\it
$N$-dimensional decreasing support property} if for every $\delta\in
(0,\infty)$
\begin{equation}
\label{eq:defining property of decreasing support property}
\lim_{\e\to 0^+}\int_{\delta}^\infty\rho_\e(r)r^{N-1}dr=0.
\end{equation}
\end{definition}

Note that by using polar coordinates (see Proposition
\ref{prop:polar coordinates}), we obtain an alternative form for
\eqref{eq:defining property of decreasing support property}:
\begin{equation}
\label{eq:defining property of decreasing support property 1}
\lim_{\e\to 0^+}\int_{\R^N\setminus B_\delta(0)}\rho_\e(|z|)dz=0.
\end{equation}

\begin{definition}(Kernel)
\label{def:kernel}

Let $a\in(0,\infty]$. Let $\rho_\e:(0,\infty)\to
[0,\infty),\e\in (0,a),$ be a family of $\mathcal{L}^1$-measurable
functions. We say that the family $\{\rho_\e\}_{\e\in (0,a)}$ is a
{\it kernel} if $\int_{\R^N}\rho_{\e}(|z|)dz=1,\forall \e\in (0,a)$,
and it has the decreasing support property as defined in Definition
\ref{def:decreasing support property}.
\end{definition}

The BBM formula states that for an open and bounded set $\Omega \subset \R^N$ with a Lipschitz boundary, $1<q<\infty$, and $u\in W^{1,q}(\Omega)$, for every kernel $\{\rho_\e\}_{\e\in (0,a)}$ (as defined in Definition \ref{def:kernel}), we have
\begin{equation}
\label{jjjkjk}
\lim_{\e\to
0^+}\int_{\Omega}\left(\int_{\Omega}\rho_{\e}(|x-y|)\frac{|u(x)-u(y)|^q}{|x-y|^q}dy\right)dx=\hat
C_{q,N}\|\nabla u\|^q_{L^q(\Omega)}.
\end{equation}
Similarly, for $u\in BV(\Omega)$, we have
\begin{equation}
\label{eq:ineqality61}
\lim_{\e\to
0^+}\int_{\Omega}\left(\int_{\Omega}\rho_{\e}(|x-y|)\frac{|u(x)-u(y)|}{|x-y|}dy\right)dx=\hat
C_{1,N}\|D u\|(\Omega),
\end{equation}
where $\hat C_{q,N}:=\fint_{S^{N-1}}|z_1|^qd\mathcal{H}^{N-1}(z)$ for every $q\geq 1$.

In \cite{P}, the following question was investigated: 
\begin{question}
What does happen if we replace the left-hand side of equation \eqref{jjjkjk}, where $q>1$, by the following expression:
\begin{equation}
\label{GJGGGGJ}
\lim_{\e\to
0^+}\int_{\Omega}\left(\int_{\Omega}\rho_{\e}(|x-y|)\frac{|u(x)-u(y)|^q}{|x-y|}dy\right)dx\,
\end{equation}
? 

Here the limit $\eqref{GJGGGGJ}$ is obtained by replacing
$\frac{|u(x)-u(y)|^q}{|x-y|^q}$ in \er{jjjkjk} by
$\frac{|u(x)-u(y)|^q}{|x-y|}$.
\end{question}
Then, the following limit was studied
\begin{equation}
\label{eq:inequality7} \lim_{\e\to
0^+}\int_{\Omega}\left(\int_{\Omega\cap
B_{\e}(x)}\frac{1}{\mathcal{L}^{N}(B_1(0))\e^N}\frac{|u(x)-u(y)|^q}{|x-y|}dy\right)dx,
\end{equation}
for $1<q<\infty$, $\Omega \subset \R^N$ is an open set with a bounded Lipschitz boundary, and $u\in BV(\Omega,\R^d)\cap L^\infty(\Omega,\mathbb{R}^d)$. This is a particular case of the expression \eqref{GJGGGGJ} with the specific choice of the kernel $\tilde{\rho}_\e(r)$ given by
\begin{equation}
\tilde{\rho}_\e(r):=
\begin{cases}
\frac{1}{\e^N \Leb^{N}\left(B_1(0)\right)}& if \quad 0<r<\e\\
0 & if \quad r\geq \e
\end{cases},\quad \e\in (0,\infty).
\end{equation}

Here, we refer to such a specific kernel as the 'trivial kernel' (see Definition \ref{def:definition of trivial kernel}). The space $BV^q(\Omega,\mathbb{R}^d)$ was also considered in \cite{P}: we define $u\in BV^q(\Omega,\mathbb{R}^d)$ if and only if $u\in L^q(\Omega,\mathbb{R}^d)$ and
\begin{equation}
\limsup_{\e\to 0^+}\int_{\Omega}\left(\int_{\Omega\cap
B_{\e}(x)}\frac{1}{\e^N}\frac{|u(x)-u(y)|^q}{|x-y|}dy\right)dx<\infty
\end{equation}
holds. In \cite{P}, it was proved that the limit in \eqref{eq:inequality7} is determined solely by the jump part of the distributional derivative of $u$, without involving the absolutely continuous and Cantor parts:
\begin{theorem*}(Theorem 1.1 in \cite{P})

Let $\Omega\subset\R^N$
be an open set with bounded Lipschitz boundary and let $u\in
BV(\Omega,\R^d)\cap L^\infty(\Omega,\R^d).$ Then for every
$1<q<\infty$ we have $u\in BV^q(\Omega,\R^d)$ and
\begin{equation}
\label{eq:equality8}
C_N\int_{\mathcal{J}_u}|u^+(x)-u^-(x)|^q~d\Haus^{N-1}(x)=\lim_{\e\to
0^+}\int_{\Omega}\left(\int_{\Omega\cap
B_{\e}(x)}\frac{1}{\e^N}\frac{|u(x)-u(y)|^q}{|x-y|}dy\right)dx,
\end{equation}
where
\begin{equation}
\label{eq:definition of dimensional constant C_N}
C_N:=\frac{1}{N}\int_{S^{N-1}}|z_1|~d\Haus^{N-1}(z),\quad
z:=(z_1,...,z_N).
\end{equation}
Here $\mathcal{J}_u$ is the jump set of the function $u$, and $u^+$
and $u^-$ are the one-sided approximate limits of $u$ on $\mathcal{J}_u$.
\end{theorem*}

Recall the definition of Besov space $B^r_{q,\infty}(\mathbb{R}^N,\mathbb{R}^d)$:
\begin{definition}(Besov spaces)\\
\label{def:definition of Besov space}
Let $1\leq q<\infty$ and $r\in (0,1)$. Define
\begin{equation}
B^r_{q,\infty}(\R^N,\R^d):=\bigg\{u\in L^q(\R^N,\R^d)
:\sup_{h\in\R^N\setminus\{0\}}\int_{\R^N}\frac{|u(x+h)-u(x)|^q}{|h|^{rq}}dx<\infty\bigg\}.
\end{equation}
For an open set $\Omega\subset\R^N$, the local space $\left(B^{r}_{q,\infty}\right)_{\text{loc}}(\Omega,\R^d)$ is defined to be the set of all functions $u\in L^q_{\text{loc}}(\Omega,\R^d)$ such that for every compact $K\subset\Omega$ there exists a function $u_K\in B^r_{q,\infty}(\R^N,\R^d)$ such that $u_K(x)=u(x)$ for $\mathcal{L}^N$-almost every $x\in K$.
\end{definition}

The following proposition gives us a connection between Besov
functions in $B^{1/q}_{q,\infty}$ and $BV^q-$functions.
\begin{proposition*}(Proposition 1.1 in \cite{P})
For $1<q<\infty$ we have:
\begin{equation}
BV^q(\R^N,\R^d)=B^{1/q}_{q,\infty}(\R^N,\R^d).
\end{equation}
Moreover, for every open set $\Omega\subset\R^N$ we have
\begin{equation}
BV^q_{\text{loc}}(\Omega,\R^d)=\left(B^{1/q}_{q,\infty}\right)_{\text{loc}}(\Omega,\R^d),
\end{equation}
where the local space $BV^q_{\text{loc}}(\Omega,\R^d)$ is defined in a
usual way.
\end{proposition*}
A more general result than the proposition above was independently obtained by Brasseur in \cite{Br}.
For a comprehensive introduction to Besov spaces, a recommended reference is \cite{Giovanni}.

Next, recall the notion of Gagliardo seminorm:
\begin{definition}
(Gagliardo Seminorm)
\label{def:Gagliardo seminorm}

Let $1\leq
q<\infty$, $E\subset \R^N$ be an $\mathcal{L}^N$-measurable set,
$u\in L^q(E,\R^d)$ and $r\in (0,1)$. The {\it Gagliardo seminorm} of
$u$ in $E$ is defined by
\begin{equation}
\left[u\right]_{W^{r,q}(E,\R^d)}:=\left(\int_{E}\int_{E}\frac{|u(x)-u(y)|^q}{|x-y|^{N+rq}}dxdy\right)^{\frac{1}{q}}.
\end{equation}
\end{definition}

In \cite{PAsymptotic}, the following result was proved: for a
Lipschitz domain $\Omega$, $q\in (1,\infty)$, $u\in
BV(\Omega,\R^d)\cap L^\infty(\Omega,\R^d)$, and $\eta\in
W^{1,1}(\R^N)$ such that $\int_{\R^N}\eta(z)dz=1$, if we mollify $u$
by setting for every $\e>0$ and $x\in\R^N$
\begin{equation}
\label{eq:mollified u}
u_\e(x):=\frac{1}{\e^N}\int_{\R^N}\eta\Big(\frac{y-x}{\e}\Big)u(y)dy,
\end{equation}
then
\begin{equation}
\label{eq:connection between Gagliardo constant and jump variation in bounded and BV case}
\lim_{\e\to
0^+}\frac{1}{|\ln{\e}|}\left[u_\e\right]^q_{W^{1/q,q}(\Omega,\R^d)}=
\tilde C_N
\int_{\mathcal{J}_u} \Big|u^+(x)-u^-(x)\Big|^q d\mathcal{H}^{N-1}(x),
\end{equation}
with an appropriate dimensional constant $\tilde C_N>0$ (where $u$ in \eqref{eq:mollified u} is assumed to be continued from $\Omega$ to $\R^N$ such that $u\in BV(\R^N,\R^d)\cap L^\infty(\R^N,\R^d)$ and $\|Du\|(\partial\Omega)=0$). It is worth noting that the particular case of \eqref{eq:connection between Gagliardo constant and jump variation in bounded and BV case} with $\eta$ as the Gaussian, $q=2$, and $\Omega=\R^N$ was previously proved by Figalli and Jerison in \cite{AFDJ} for the characteristic function of a set, and by Hern\'{a}ndez in \cite{FHHalp} for a general function $u$. Combining \eqref{eq:equality8} and \eqref{eq:connection between Gagliardo constant and jump variation in bounded and BV case}, we deduce that
\begin{equation}
\label{eq:connection between Gagliardo constant and Besov constant in bounded and BV case}
\lim_{\e\to
0^+}\frac{1}{|\ln{\e}|}\left[u_\e\right]^q_{W^{1/q,q}(\Omega,\R^d)}=
\frac{\tilde C_N}{C_N}
\lim_{\e\to 0^+}\int_{\Omega}\left(\int_{\Omega\cap
B_{\e}(x)}\frac{1}{\e^N}\frac{|u(x)-u(y)|^q}{|x-y|}dy\right)dx.
\end{equation}
This naturally leads us to pose the following interesting question: does \eqref{eq:connection between Gagliardo constant and Besov constant in bounded and BV case} hold also for $u\in BV^q\setminus(BV\cap L^\infty)$?

Our first two main results are related to this question:
\begin{theorem}
\label{thm:Sandwich inequality for Gagliardo constants, intro}
Let $q \in [1,\infty)$ and $r \in (0,1)$. Suppose $u \in B^r_{q,\infty}(\mathbb{R}^N,\mathbb{R}^d)$, $E \subset \mathbb{R}^N$ be an $\mathcal{L}^N$-measurable set and $\eta \in W^{1,1}(\mathbb{R}^N)$. For each $\e\in(0,\infty)$ we denote
\begin{equation}
\label{eq:definition of convolution}
u_\e(x):=\int_{\R^N}\eta(z)u(x-\e z)dz.
\end{equation}
Then,
\begin{multline}
\label{eq:Sandwich inequality for Gagliardo constants,intro}
\left|\int_{\R^N}\eta(z)dz\right|^q\liminf_{\e\to 0^+}\int_{S^{N-1}}\int_E\chi_{E}(x+\e n)\frac{|u(x+\e n)-u(x)|^q}{\e^{rq}}dxd\mathcal{H}^{N-1}(n)
\\
\leq\liminf_{\e\to
0^+}\frac{1}{|\ln{\e}|}\left[u_\e\right]^q_{W^{r,q}(E,\R^d)}
\leq\limsup_{\e\to
0^+}\frac{1}{|\ln{\e}|}\left[u_\e\right]^q_{W^{r,q}(E,\R^d)}
\\
\leq  \left|\int_{\R^N}\eta(z)dz\right|^q\limsup_{\e\to 0^+}\int_{S^{N-1}}\int_E\chi_{E}(x+\e n)\frac{|u(x+\e n)-u(x)|^q}{\e^{rq}}dxd\mathcal{H}^{N-1}(n).
\end{multline}
\end{theorem}

\begin{theorem}
\label{them:connection between Gagliardo constant and Besov constant dependent on arbitrary kernel, introduction}
Let $q\in[1,\infty)$, $r\in(0,1)$. Let $u\in B^r_{q,\infty}(\R^N,\R^d)$, $E\subset\R^N$ be an $\mathcal{L}^N$-measurable set and $\eta\in  W^{1,1}\left(\R^N\right)$. For each $\e\in(0,\infty)$ we denote $u_\e(x):=\int_{\R^N}\eta(z)u(x-\e z)dz$. Assume that the following limit exists:
\begin{equation}
\label{eq:equation200000}
\lim_{\e\to 0^+}\int_{S^{N-1}}\int_E\chi_{E}(x+\e n)\frac{|u(x+\e n)-u(x)|^q}{\e^{rq}}dxd\mathcal{H}^{N-1}(n).
\end{equation}
Then, for every kernel $\rho_\e$ we get
\begin{multline}
\lim_{\e\to
0^+}\frac{1}{|\ln{\e}|}\left[u_\e\right]^q_{W^{r,q}(E,\R^d)}
\\
=\left|\int_{\R^N}\eta(z)dz\right|^q\mathcal{H}^{N-1}\left(S^{N-1}\right)\lim_{\e\to 0^+}\int_{E}\int_{E}\rho_{\e}(|x-y|)\frac{|u(x)-u(y)|^q}{|x-y|^{rq}}dydx
\\
=\left|\int_{\R^N}\eta(z)dz\right|^q\lim_{\e\to 0^+}\int_{S^{N-1}}\int_E\chi_{E}(x+\e n)\frac{|u(x+\e n)-u(x)|^q}{\e^{rq}}dxd\mathcal{H}^{N-1}(n).
\end{multline}
\end{theorem}

Our next result refers to jumps of functions in Besov spaces
$B^{1/p}_{p,\infty}$, which are also functions of bounded variation.
This result generalizes \er{eq:equality8} (The main improvement is
that we don't assume that $u\in L^\infty$):
\begin{theorem}
\label{lem: limit for Besov integral in terms of jumps,introduction}
Let $1<p<\infty$, $u\in BV(\R^N,\R^d)\cap B^{1/p}_{p,\infty}(\R^N,\R^d)$ and $1<q<p$. Then, for every $n\in \R^N$ and every Borel set $B\subset\R^N$ such that $\mathcal{H}^{N-1}(\partial B\cap \mathcal{J}_u)=0$, we have
\begin{equation}
\label{eq:equation277}
\lim_{\e\to 0^+}\int_B\chi_{B}(x+\e n)\frac{|u(x+\e n)-u(x)|^q}{\e}dx
=\int_{B\cap \mathcal{J}_u}\left|u^+(x)-u^-(x)\right|^q|\nu_u(x)\cdot n|d\mathcal{H}^{N-1}(x),
\end{equation}
and for every kernel $\rho_\e$, we have

\begin{multline}
\lim_{\e\to 0^+}\int_{B}\int_{B}\rho_{\e}(|x-y|)\frac{|u(x)-u(y)|^q}{|x-y|}dydx
\\
=\lim_{\e\to 0^+}\fint_{S^{N-1}}\int_B\chi_{B}(x+\e n)\frac{|u(x+\e n)-u(x)|^q}{\e}dxd\mathcal{H}^{N-1}(n)
\\
=\left(\fint_{S^{N-1}}|z_1|~d\Haus^{N-1}(z)\right)\int_{\mathcal{J}_u\cap B}
\Big|u^+(x)-u^-(x)\Big|^q d\mathcal{H}^{N-1}(x).
\end{multline}
Here $u^+,u^-$ are the one-sided approximate limits of $u$, $\nu_u$ is a unit normal and $\mathcal{J}_u$ is the jump set of $u$ (see Definition \ref{def:approximate jump points}).
\end{theorem}

\begin{corollary}
\label{cor: limit for Besov integral in terms of jumps,introduction}
Let $1<q<p<\infty$, and $u\in BV(\R^N,\R^d)\cap B^{1/p}_{p,\infty}(\R^N,\R^d)$. Let $\eta\in  W^{1,1}\left(\R^N\right)$, and define, for each $\e\in(0,\infty)$ and $x\in\R^N$, the mollification $u_\e(x):=\int_{\R^N}\eta(z)u(x-\e z)dz$. Let $B\subset\R^N$ be a Borel set such that $\mathcal{H}^{N-1}(\partial B\cap \mathcal{J}_u)=0$. Let $\rho_\e$ be a kernel.

Then we have the following equalities:
\begin{multline}
\frac{1}{\mathcal{H}^{N-1}\left(S^{N-1}\right)}\lim_{\e\to
0^+}\frac{1}{|\ln{\e}|}\left[u_{\e}\right]^q_{W^{1/q,q}(B,\R^d)}
\\
=\left|\int_{\R^N}\eta(z)dz\right|^q\lim_{\e\to 0^+}\int_{B}\int_{B}\rho_{\e}(|x-y|)\frac{|u(x)-u(y)|^q}{|x-y|}dydx
\\
=\left|\int_{\R^N}\eta(z)dz\right|^q\lim_{\e\to 0^+}\fint_{S^{N-1}}\int_B\chi_{B}(x+\e n)\frac{|u(x+\e n)-u(x)|^q}{\e}dxd\mathcal{H}^{N-1}(n)
\\
=\left|\int_{\R^N}\eta(z)dz\right|^q \left(\fint_{S^{N-1}}|z_1|~d\Haus^{N-1}(z)\right)\int_{\mathcal{J}_u\cap B}
\Big|u^+(x)-u^-(x)\Big|^q d\mathcal{H}^{N-1}(x).
\end{multline}
\end{corollary}

\begin{notation}
Throughout the paper, we adopt the following notation: $N$ and $d$ are natural numbers ($N, d \in \mathbb{N}$). We denote $S^{N-1}$ as the $(N-1)$-dimensional sphere in $\mathbb{R}^N$. The $N$-dimensional Lebesgue measure is denoted as $\mathcal{L}^N$, while $\mathcal{H}^{N-1}$ represents the $(N-1)$-dimensional Hausdorff measure. For an open ball in $\mathbb{R}^N$ centered at $x$ with a radius of $r$, we use the notation $B_r(x)$. The characteristic function of a set $E$ is denoted as $\chi_E$. Furthermore, we use the notation $A \subset\subset B$ to indicate that $\bar{A}$ is a compact set and $\bar{A}\subset B$, where $\bar{A}$ represents the topological closure of $A$.
\end{notation}

\section{Estimates for Gagliardo Seminorm of Mollified Besov Functions in Terms of Besov Seminorm}
In this section we establish estimates for the Gagliardo seminorm of mollified Besov functions in relation to the Besov seminorm of the functions themselves, without mollification (refer to Corollary \ref{cor:Estimates for Gagliardo seminorm of mollified Besov functions - part 2}). These estimates will enable us to establish a continuity property for the upper and lower $G$-functionals in the next section (refer to Definition \ref{def:upper and lower functionals} and Lemma \ref{lem:continuity of upper and lower $G$-functionals}).

\begin{definition}(Besov Seminorm)\\
\label{def:Besov seminorm}
Let $1\leq q<\infty$, $r\in (0,1)$ and $E\subset \R^N$ be an $\mathcal{L}^N$-measurable set. Let $u:E\to\R^d$ be an $\mathcal{L}^N$-measurable function. The {\it Besov seminorm} of $u$ with parameters $r,q$ in $E$ is defined by
\begin{align}
[u]_{B^{r}_{q,\infty}(E,\R^d)}:=\sup_{h\in \R^N\setminus\{0\}}\left(\int_{\R^N}\frac{|u(x+h)-u(x)|^q}{|h|^{rq}}\chi_{E}(x+h)\chi_{E}(x)dx\right)^{1/q}.
\end{align}
\end{definition}

\begin{definition}(Mollification and Mollifier)
\label{def:Mollification}

Let $\eta:\R^N\to \R$ be a function. For each $\e\in (0,\infty)$ we denote $\eta_{(\e)}(x):=\frac{1}{\e^N}\eta\left(\frac{x}{\e}\right),x\in \R^N$. The function $\eta_{(\e)}$ is called an {\it $\e$-mollifier} obtained by $\eta$. We call $\{\eta_{(\e)}\}_{\e\in(0,\infty)}$ a {\it family of mollifiers}. For $\eta\in L^1(\R^N)$, $1\leq q\leq \infty$, and $u\in L^q(\R^N,\R^d)$, let us define
\begin{equation}
\label{eq:convolution}
u_\e(x):=u*\eta_{(\e)}(x)=\int_{\R^N}\eta_{(\e)}(x-z)u(z)dz=\int_{\R^N}\eta(z)u(x-\e z)dz.
\end{equation}
The convolution $u_\e$ is called {\it mollification} of $u$ by the family of mollifiers $\{\eta_{(\e)}\}_{\e\in(0,\infty)}$.
\end{definition}

\begin{lemma}(Boundedness of Mollified Functions in Besov and Gagliardo Seminorms)
\label{lem:Besov seminorm and Gagliardo seminorm with convolution}

Let $1\leq q<\infty$, $u\in L^q(\R^N,\R^d)$ and $\eta\in L^1(\R^N)$. Then, for every $z\in \R^N$ and $\e\in (0,\infty)$
\begin{equation}
\label{eq:equation24}
\int_{\R^N}|u_\e(x)-u_\e(x+z)|^qdx\leq \left(\int_{\R^N}|\eta(v)|dv\right)^{q} \int_{\R^N}|u(x)-u(x+z)|^qdx.
\end{equation}
In particular, for every $r\in (0,1)$
\begin{equation}
\label{eq:equation21}
\sup_{\e\in (0,\infty)}\left([u_\e]_{B^{r}_{q,\infty}(\R^N,\R^d)}\right)\leq \left(\int_{\R^N}|\eta(v)|dv\right)[u]_{B^{r}_{q,\infty}(\R^N,\R^d)};
\end{equation}
\begin{equation}
\label{eq:equation22}
\sup_{\e\in (0,\infty)}\left(\left[u_\e\right]_{W^{r,q}(\R^N,\R^d)}\right)\leq \left(\int_{\R^N}|\eta(v)|dv\right)\left[u\right]_{W^{r,q}(\R^N,\R^d)}.
\end{equation}
\end{lemma}

\begin{proof}
By \eqref{eq:convolution}, H{\"o}lder's inequality, Fubini's theorem and change of variable formula
\begin{multline}
\label{eq:equation23}
\int_{\R^N}|u_\e(x)-u_\e(x+z)|^qdx
=\int_{\R^N}\left|\int_{\R^N}\eta(v)\left(u(x-\e v)-u(x+z-\e v)\right)dv\right|^qdx
\\
\leq\int_{\R^N}\left(\int_{\R^N}|\eta(v)|\left|u(x-\e v)-u(x+z-\e v)\right|dv\right)^qdx
\\
=\int_{\R^N}\left(\int_{\R^N}|\eta(v)|^{\frac{q-1}{q}}\left(|\eta(v)|^{\frac{1}{q}}\left|u(x-\e v)-u(x+z-\e v)\right|\right)dv\right)^qdx
\\
\leq\left(\int_{\R^N}|\eta(v)|dv\right)^{q-1}\int_{\R^N}\int_{\R^N}|\eta(v)|\left|u(x-\e v)-u(x+z-\e v)\right|^q dvdx
\\
=\left(\int_{\R^N}|\eta(v)|dv\right)^{q-1}\int_{\R^N}|\eta(v)|\left(\int_{\R^N}\left|u(x-\e v)-u(x+z-\e v)\right|^q dx\right)dv
\\
=\left(\int_{\R^N}|\eta(v)|dv\right)^{q}\int_{\R^N}\left|u(y)-u(y+z)\right|^qdy.
\end{multline}
Let $r\in (0,1)$. Dividing the inequality \eqref{eq:equation23} by $|z|^{rq},z\neq 0$, taking the supremum over $z\in \R^N\setminus\{0\}$ and then the supremum over $\e\in (0,\infty)$, we obtain \eqref{eq:equation21}.
By Definition \ref{def:Gagliardo seminorm} (Gagliardo seminorm), change of variable formula, Fubini's theorem and  \eqref{eq:equation24} we get
\begin{multline}
\label{eq:equation34}
\left[u_\e\right]^q_{W^{r,q}(\R^N,\R^d)}=\int_{\R^N}\left(\int_{\R^N}\frac{|u_\e(x)-u_\e(y)|^q}{|x-y|^{N+rq}}dx\right)dy
=\int_{\R^N}\left(\int_{\R^N}\frac{|u_\e(x+y)-u_\e(y)|^q}{|x|^{N+rq}}dx\right)dy
\\
=\int_{\R^N}\frac{1}{|x|^{N+rq}}\left(\int_{\R^N}|u_\e(x+y)-u_\e(y)|^qdy\right)dx
\\
\leq \|\eta\|^q_{L^1(\R^N)}\int_{\R^N}\frac{1}{|x|^{N+rq}}\left(\int_{\R^N}|u(x+y)-u(y)|^qdy\right)dx
=\|\eta\|^q_{L^1(\R^N)}\int_{\R^N}\left(\int_{\R^N}\frac{|u(x)-u(y)|^q}{|x-y|^{N+rq}}dx\right)dy.
\end{multline}
Inequality \eqref{eq:equation22} follows from \eqref{eq:equation34}.
\end{proof}

\begin{lemma}(Estimates for Gagliardo Seminorm of Mollified Besov Functions - part 1)\\
\label{lem:Estimates for Gagliardo seminorm of mollified Besov functions - part 1}
Let
$1\leq q<\infty$, $r\in (0,1)$, $u\in B^r_{q,\infty}(\R^N,\R^d)$ and $\eta\in W^{1,1}(\R^N)$. For every $\e\in (0,\infty)$ and $z\in \R^N\setminus\{0\}$ we denote
\begin{equation}
\label{eq:equation69}
g^\e(z):=\int_{\R^N}\frac{|u_\e(x)-u_\e(x+z)|^q}{|z|^{N+rq}}dx.
\end{equation}
Then, for every $0<\beta<\gamma<\infty$ it follows that
\begin{align}
\label{eq:equation68}
\int_{\R^N\setminus B_\gamma(0)}g^\e(z)dz\leq \|\eta\|^q_{L^1(\R^N)}2^q\|u\|^q_{L^q(\R^N,\R^d)}\frac{\mathcal{H}^{N-1}\left(S^{N-1}\right)}{rq\gamma^{rq}};
\end{align}
\begin{align}
\label{eq:equation70}
\int_{B_\gamma(0)\setminus B_\beta(0)}g^\e(z)dz\leq \|\eta\|^q_{L^1(\R^N)} [u]^q_{B^{r}_{q,\infty}(\R^N,\R^d)}\mathcal{H}^{N-1}\left(S^{N-1}\right)\left(\ln(\gamma)-\ln(\beta)\right);
\end{align}
\begin{align}
\label{eq:equation71}
\int_{B_\beta(0)}g^\e(z)dz\leq \left(\int_{\R^N}|\nabla\eta(v)|dv\right)^{q}2^q\|u\|^q_{L^q(\R^N,\R^d)}\frac{\mathcal{H}^{N-1}\left(S^{N-1}\right)}{q-rq}\frac{\beta^{q-rq}}{\e^{q}}.
\end{align}
If $\e=\beta$, then we have the following alternative to \eqref{eq:equation71} estimate:
\begin{equation}
\label{eq:equation72}
\int_{B_\e(0)}g^\e(z)dz\leq \|\nabla\eta\|^{q-1}_{L^1(\R^N,\R^N)}[u]^q_{B^{r}_{q,\infty}(\R^N,\R^d)}\left(\int_{\R^N}|\nabla\eta(v)|(|v|+2)^{rq}dv\right)\frac{\mathcal{H}^{N-1}\left(S^{N-1}\right)}{q-rq}.
\end{equation}
The right hand side of \eqref{eq:equation72} can be infinite.
\end{lemma}

\begin{proof}
By Lemma \ref{lem:Besov seminorm and Gagliardo seminorm with convolution} and the convexity of the function $r\longmapsto r^q,r\in [0,\infty)$, we have
\begin{align}
\label{eq:equation35}
g^\e(z)\leq \frac{\|\eta\|^q_{L^1(\R^N)}}{|z|^{N+rq}}\int_{\R^N}|u(x)-u(x+z)|^qdx\leq \|\eta\|^q_{L^1(\R^N)}2^q\|u\|^q_{L^q(\R^N,\R^d)} \frac{1}{|z|^{N+rq}}.
\end{align}
Thus, by polar coordinates (refer to Proposition \ref{prop:polar coordinates})
\begin{multline}
\int_{\R^N\setminus B_\gamma(0)}g^\e(z)dz\leq \|\eta\|^q_{L^1(\R^N)}2^q\|u\|^q_{L^q(\R^N,\R^d)}\int_{\R^N\setminus B_\gamma(0)} \frac{1}{|z|^{N+rq}}dz
\\
=\|\eta\|^q_{L^1(\R^N)}2^q\|u\|^q_{L^q(\R^N,\R^d)}\frac{\mathcal{H}^{N-1}\left(S^{N-1}\right)}{rq\gamma^{rq}}.
\end{multline}
It proves \eqref{eq:equation68}. By \eqref{eq:equation69}, \eqref{eq:equation21} and polar coordinates
\begin{multline}
\int_{B_\gamma(0)\setminus B_\beta(0)}g^\e(z)dz=\int_{B_\gamma(0)\setminus B_\beta(0)}\frac{1}{|z|^N}\left(\int_{\R^N}\frac{|u_\e(x)-u_\e(x+z)|^q}{|z|^{rq}}dx\right)dz
\\
\leq \left(\int_{\R^N}|\eta(v)|dv\right)^{q} [u]^q_{B^{r}_{q,\infty}(\R^N,\R^d)}\int_{B_\gamma(0)\setminus B_\beta(0)}\frac{1}{|z|^N}dz
\\
=\left(\int_{\R^N}|\eta(v)|dv\right)^{q} [u]^q_{B^{r}_{q,\infty}(\R^N,\R^d)}\mathcal{H}^{N-1}\left(S^{N-1}\right)\left(\ln(\gamma)-\ln(\beta)\right).
\end{multline}
It proves \eqref{eq:equation70}. We now prove \eqref{eq:equation71}. By \eqref{eq:equation69} and Fubini's theorem
\begin{align}
\int_{B_\beta(0)}g^\e(z)dz
=\int_{\R^N}\left(\int_{B_\beta(0)}\frac{|u_\e(x)-u_\e(x+z)|^q}{|z|^{N+rq}}dz\right)dx.
\end{align}
Assume for a moment that $\eta\in C^1(\R^N)\cap W^{1,1}(\R^N)$. By \eqref{eq:convolution}, change of variable formula, the fundamental theorem of calculus, Fubini's theorem and Jensen's inequality we obtain for every $x\in\R^N$
\begin{multline}
\label{eq:equation77}
\int_{B_\beta(0)}\frac{|u_\e(x)-u_\e(x+z)|^q}{|z|^{N+rq}}dz
\\
=\int_{B_\beta(0)}\frac{1}{|z|^{N+rq}}\left|\frac{1}{\e^N}\int_{\R^N}\left(\eta\left(\frac{x-y}{\e}\right)-\eta\left(\frac{x+z-y}{\e}\right)\right)u(y)dy\right|^qdz,\quad\quad [z=\e w]
\\
=\frac{1}{\e^{rq}}\int_{B_{\beta/\e}(0)}\frac{1}{|w|^{N+rq}}\left|\frac{1}{\e^N}\int_{\R^N}\left(\eta\left(\frac{x-y}{\e}\right)-\eta\left(\frac{x-y}{\e}+w\right)\right)u(y)dy\right|^qdw,\quad\quad [y=x-\e v]
\\
=\frac{1}{\e^{rq}}\int_{B_{\beta/\e}(0)}\frac{1}{|w|^{N+rq}}\left|\int_{\R^N}\left(\eta(v)-\eta(v+w)\right)u(x-\e v)dv\right|^qdw
\\
=\frac{1}{\e^{rq}}\int_{B_{\beta/\e}(0)}\frac{1}{|w|^{N+rq}}\left|\int_{\R^N}\left(\eta(v+w)-\eta(v)\right)\left(u(x-\e v)-u(x)\right)dv\right|^qdw
\\
=\frac{1}{\e^{rq}}\int_{B_{\beta/\e}(0)}\frac{1}{|w|^{N+rq}}\left|\int_{\R^N}\left(w\cdot\int_0^1\nabla\eta(v+tw)dt\right)\left(u(x-\e v)-u(x)\right)dv\right|^qdw
\\
\leq\frac{1}{\e^{rq}}\int_{B_{\beta/\e}(0)}\frac{1}{|w|^{N+rq-q}}\left(\int_0^1\int_{\R^N}|\nabla\eta(v+tw)|\left|u(x-\e v)-u(x)\right|dvdt\right)^qdw
\\
\leq\frac{1}{\e^{rq}}\int_0^1\int_{B_{\beta/\e}(0)}\frac{1}{|w|^{N+rq-q}}\left(\int_{\R^N}|\nabla\eta(v+tw)|\left|u(x-\e v)-u(x)\right|dv\right)^qdwdt
\\
=\frac{1}{\e^{rq}}\int_0^1\int_{B_{\beta/\e}(0)}\frac{1}{|w|^{N+rq-q}}\left(\int_{\R^N}|\nabla\eta(v)|\left|u(x-\e (v-tw))-u(x)\right|dv\right)^qdwdt.
\end{multline}
In the forth equality we use $\int_{\R^N}\left(\eta(v+w)-\eta(v)\right)dv=0,w\in \R^N$.
By H{\"o}lder's inequality
\begin{multline}
\label{eq:equation75}
\left(\int_{\R^N}|\nabla\eta(v)|\left|u(x-\e(v-tw))-u(x)\right|dv\right)^q
\\
=\left(\int_{\R^N}|\nabla\eta(v)|^{\frac{q-1}{q}}|\nabla\eta(v)|^{\frac{1}{q}}\left|u(x-\e(v-tw))-u(x)\right|dv\right)^q
\\
\leq \left(\int_{\R^N}|\nabla\eta(v)|dv\right)^{q-1}\int_{\R^N}|\nabla\eta(v)|\left|u(x-\e (v-tw))-u(x)\right|^qdv.
\end{multline}
By \eqref{eq:equation75} and Fubini's theorem
\begin{multline}
\label{eq:equation76}
\int_{\R^N}\left(\int_{\R^N}|\nabla\eta(v)|\left|u(x-\e(v-tw))-u(x)\right|dv\right)^qdx
\\
\leq\left(\int_{\R^N}|\nabla\eta(v)|dv\right)^{q-1}\int_{\R^N}|\nabla\eta(v)|\left(\int_{\R^N}\left|u(x-\e (v-tw))-u(x)\right|^qdx\right)dv
\\
\leq \left(\int_{\R^N}|\nabla\eta(v)|dv\right)^{q}2^q\|u\|^q_{L^q(\R^N,\R^d)}.
\end{multline}
By \eqref{eq:equation77} and \eqref{eq:equation76} we get
\begin{multline}
\label{eq:equation86}
\int_{\R^N}\int_{B_\beta(0)}\frac{|u_\e(x)-u_\e(x+z)|^q}{|z|^{N+rq}}dzdx
\\
\leq\left(\int_{\R^N}|\nabla\eta(v)|dv\right)^{q}2^q\|u\|^q_{L^q(\R^N,\R^d)}\frac{1}{\e^{rq}}\int_{B_{\beta/\e}(0)}\frac{1}{|w|^{N+rq-q}}dw
\\
=\left(\int_{\R^N}|\nabla\eta(v)|dv\right)^{q}2^q\|u\|^q_{L^q(\R^N,\R^d)}\frac{1}{\e^{rq}}\frac{\mathcal{H}^{N-1}\left(S^{N-1}\right)}{q-rq}\left(\frac{\beta}{\e}\right)^{q-rq}
\\
=\left(\int_{\R^N}|\nabla\eta(v)|dv\right)^{q}2^q\|u\|^q_{L^q(\R^N,\R^d)}\frac{\mathcal{H}^{N-1}\left(S^{N-1}\right)}{q-rq}\frac{\beta^{q-rq}}{\e^{q}}.
\end{multline}
It proves \eqref{eq:equation71} in case $\eta\in C^1(\R^N)\cap W^{1,1}(\R^N)$. We now prove \eqref{eq:equation72} in case $\eta\in C^1(\R^N)\cap W^{1,1}(\R^N)$.
By \eqref{eq:equation75} and Definition \ref{def:Besov seminorm} (definition of Besov seminorm)
\begin{multline}
\label{eq:equation79}
\int_{\R^N}\left(\int_{\R^N}|\nabla\eta(v)|\left|u(x-\e(v-tw))-u(x)\right|dv\right)^qdx
\\
\leq\left(\int_{\R^N}|\nabla\eta(v)|dv\right)^{q-1}\int_{\R^N}|\nabla\eta(v)|\left(\int_{\R^N}\left|u(x-\e (v-tw))-u(x)\right|^qdx\right)dv
\\
\leq \e^{rq}\left(\int_{\R^N}|\nabla\eta(v)|dv\right)^{q-1}[u]^q_{B^{r}_{q,\infty}(\R^N,\R^d)}\int_{\R^N}|\nabla\eta(v)||v-tw|^{rq}dv.
\end{multline}
By \eqref{eq:equation77} with $\e=\beta$ and \eqref{eq:equation79}
\begin{multline}
\label{eq:equation80}
\int_{\R^N}\int_{B_\e(0)}\frac{|u_\e(x)-u_\e(x+z)|^q}{|z|^{N+rq}}dzdx
\\
\leq \|\nabla\eta\|^{q-1}_{L^1(\R^N,\R^N)}[u]^q_{B^{r}_{q,\infty}(\R^N,\R^d)}\int_0^1\int_{B_1(0)}\frac{1}{|w|^{N+rq-q}}\left(\int_{\R^N}|\nabla\eta(v)||v-tw|^{rq}dv\right)dwdt
\\
\leq \|\nabla\eta\|^{q-1}_{L^1(\R^N,\R^N)}[u]^q_{B^{r}_{q,\infty}(\R^N,\R^d)}\int_0^1\int_{B_1(0)}\frac{1}{|w|^{N+rq-q}}\left(\int_{\R^N}|\nabla\eta(v)|(|v|+1)^{rq}dv\right)dwdt
\\
=\|\nabla\eta\|^{q-1}_{L^1(\R^N,\R^N)}[u]^q_{B^{r}_{q,\infty}(\R^N,\R^d)}\left(\int_{\R^N}|\nabla\eta(v)|(|v|+1)^{rq}dv\right)\int_{B_1(0)}\frac{1}{|w|^{N+rq-q}}dw
\\
=\|\nabla\eta\|^{q-1}_{L^1(\R^N,\R^N)}[u]^q_{B^{r}_{q,\infty}(\R^N,\R^d)}\left(\int_{\R^N}|\nabla\eta(v)|(|v|+1)^{rq}dv\right)\frac{\mathcal{H}^{N-1}\left(S^{N-1}\right)}{q-rq}.
\end{multline}
It proves \eqref{eq:equation72} in case $\eta\in C^1(\R^N)\cap W^{1,1}(\R^N)$. We now generalize \eqref{eq:equation71} and \eqref{eq:equation72} for $\eta\in W^{1,1}(\R^N)$.
For $\eta\in W^{1,1}(\R^N)$, let $\eta_\delta:=\eta*\gamma_{(\delta)},\gamma_{(\delta)}(v):=\frac{1}{\delta^N}\gamma\left(\frac{v}{\delta}\right)$, where $\gamma\in C^1(\R^N)$, $\supp(\gamma)\subset B_1(0)$, $\gamma\geq 0$ and $\|\gamma\|_{L^1(\R^N)}=1$. Here $\supp(\gamma)$ stands for the support of $\gamma$.
By \eqref{eq:equation80} we get for every $0<\delta<1$
\begin{multline}
\label{eq:equation81}
\int_{\R^N}\int_{B_\e(0)}\frac{|u*\left(\eta_\delta\right)_{(\e)}(x)-u*\left(\eta_\delta\right)_{(\e)}(x+z)|^q}{|z|^{N+rq}}dzdx
\\
\leq\|\nabla\eta_\delta\|^{q-1}_{L^1(\R^N,\R^N)}[u]^q_{B^{r}_{q,\infty}(\R^N,\R^d)}\left(\int_{\R^N}|\nabla\eta_\delta(v)|(|v|+1)^{rq}dv\right)\frac{\mathcal{H}^{N-1}\left(S^{N-1}\right)}{q-rq}
\\
\leq \|\nabla\eta\|^{q-1}_{L^1(\R^N,\R^N)}[u]^q_{B^{r}_{q,\infty}(\R^N,\R^d)}\left(\int_{\R^N}|\nabla\eta(v)|(|v|+2)^{rq}dv\right)\frac{\mathcal{H}^{N-1}\left(S^{N-1}\right)}{q-rq}.
\end{multline}
Let us explain the last inequality of \eqref{eq:equation81}: Since
\begin{multline}
|\nabla\eta_\delta(v)|=|\nabla\eta*\gamma_{(\delta)}(v)|=\left|\int_{\R^N}\nabla \eta(z)\gamma_{(\delta)}(v-z)dz\right|=\left|\int_{\R^N}\nabla \eta(v-\delta y)\gamma(y)dy\right|
\\
\leq \int_{\R^N}|\nabla \eta(v-\delta y)|\gamma(y)dy,
\end{multline}
then we get by Fubini's theorem, change of variable formula and properties of $\gamma$
\begin{align}
\label{eq:equation85}
\int_{\R^N}|\nabla\eta_\delta(v)|dv
\leq \int_{\R^N}\left(\int_{\R^N}|\nabla \eta(v-\delta y)|dv\right)\gamma(y)dy=\int_{\R^N}|\nabla \eta(v)|dv,
\end{align}
and
\begin{multline}
\label{eq:equation111}
\int_{\R^N}|\nabla\eta_\delta(v)|(|v|+1)^{rq}dv
\leq \int_{\R^N}\left(\int_{\R^N}|\nabla \eta(v-\delta y)|(|v|+1)^{rq}dv\right)\gamma(y)dy
\\
=\int_{B_1(0)}\left(\int_{\R^N}|\nabla \eta(v)|(|v+\delta y|+1)^{rq}dv\right)\gamma(y)dy
\leq\int_{\R^N}|\nabla \eta(v)|(|v|+2)^{rq}dv.
\end{multline}
The second inequality in \eqref{eq:equation81} follows from \eqref{eq:equation85} and \eqref{eq:equation111}.  Note that
\begin{align}
u*\left(\eta_\delta\right)_{(\e)}=u*\left(\eta*\gamma_{(\delta)}\right)_{(\e)}
=u*\left(\eta_{(\e)}*\left(\gamma_{(\delta)}\right)_{(\e)}\right)=u*\left(\eta_{(\e)}*\left(\gamma_{(\e)}\right)_{(\delta)}\right)
=\left(u*\eta_{(\e)}\right)*\left(\gamma_{(\e)}\right)_{(\delta)}.
\end{align}
Since $u*\eta_{(\e)}\in L^q(\R^N,\R^d)$, $\gamma_{(\e)}\in C^1_c(\R^N)$, $\gamma_{(\e)}\geq 0$ and $\|\gamma_{(\e)}\|_{L^1(\R^N)}=1$, then the family of functions $\{u*\left(\eta_\delta\right)_{(\e)}\}_{\{0<\delta<1\}}$ converges in $L^q(\R^N,\R^d)$ to the function $u*\eta_{(\e)}$ as $\delta\to 0^+$, and hence has a subsequence converging almost everywhere. Thus, by \eqref{eq:equation81} and Fatou's Lemma we get \eqref{eq:equation72} for $\eta\in W^{1,1}(\R^N)$.

Using the same technique we get also \eqref{eq:equation71} for $\eta\in W^{1,1}(\R^N)$: Let $\{u*\left(\eta_{\delta_{n}}\right)_{(\e)}\}_{n\in\N}$ be a sequence converging $\mathcal{L}^N$-almost everywhere to the function $u*\eta_{(\e)}$. By \eqref{eq:equation86} and \eqref{eq:equation85} we have for every $n\in\N$
\begin{multline}
\int_{\R^N}\int_{B_\beta(0)}\frac{|u*\left(\eta_{\delta_{n}}\right)_{(\e)}(x)-u*\left(\eta_{\delta_{n}}\right)_{(\e)}(x+z)|^q}{|z|^{N+rq}}dzdx
\\
\leq \left(\int_{\R^N}|\nabla\eta_{\delta_n}(v)|dv\right)^{q}2^q\|u\|^q_{L^q(\R^N,\R^d)}\frac{\mathcal{H}^{N-1}\left(S^{N-1}\right)}{q-rq}\frac{\beta^{q-rq}}{\e^{q}}
\\
\leq \left(\int_{\R^N}|\nabla\eta(v)|dv\right)^{q}2^q\|u\|^q_{L^q(\R^N,\R^d)}\frac{\mathcal{H}^{N-1}\left(S^{N-1}\right)}{q-rq}\frac{\beta^{q-rq}}{\e^{q}}.
\end{multline}
Taking the limit as $n$ goes to $\infty$ and using Fatou's lemma we get \eqref{eq:equation71} for $\eta\in W^{1,1}(\R^N)$.
\end{proof}

\begin{corollary}
\label{cor:Estimates for Gagliardo seminorm of mollified Besov functions - part 2}
(Estimates for Gagliardo Seminorm of Mollified Besov Functions - part 2)\\
Let
$1\leq q<\infty$, $r\in(0,1)$, $u\in B^r_{q,\infty}(\R^N,\R^d)$ and $\eta\in W^{1,1}(\R^N)$. For every $\e\in (0,\infty)$ and $0<\beta<\gamma<\infty$ it follows that
\begin{multline}
\label{eq:equation73}
\left[u_{\e}\right]^q_{W^{r,q}(\R^N,\R^d)}\leq \|\eta\|^q_{L^1(\R^N)}2^q\|u\|^q_{L^q(\R^N,\R^d)}\frac{\mathcal{H}^{N-1}\left(S^{N-1}\right)}{rq\gamma^{rq}}
\\
+\|\eta\|^q_{L^1(\R^N)} [u]^q_{B^{r}_{q,\infty}(\R^N,\R^d)}\mathcal{H}^{N-1}\left(S^{N-1}\right)\left(\ln(\gamma)-\ln(\beta)\right)
\\
+\left(\int_{\R^N}|\nabla\eta(v)|dv\right)^{q}2^q\|u\|^q_{L^q(\R^N,\R^d)}\frac{\mathcal{H}^{N-1}\left(S^{N-1}\right)}{q-rq}\frac{\beta^{q-rq}}{\e^{q}}.
\end{multline}
In particular,
\begin{multline}
\label{eq:equation74}
\sup_{\e\in (0,1/e)}\frac{1}{|\ln\e|}\left[u_{\e}\right]^q_{W^{r,q}(\R^N,\R^d)}\leq \|\eta\|^q_{L^1(\R^N)}2^q\|u\|^q_{L^q(\R^N,\R^d)}\frac{\mathcal{H}^{N-1}\left(S^{N-1}\right)}{rq}
\\
+\|\eta\|^q_{L^1(\R^N)}[u]^q_{B^{r}_{q,\infty}(\R^N,\R^d)}\mathcal{H}^{N-1}\left(S^{N-1}\right)\frac{q}{q-rq}
\\
+\left(\int_{\R^N}|\nabla\eta(v)|dv\right)^{q}2^q\|u\|^q_{L^q(\R^N,\R^d)}\frac{\mathcal{H}^{N-1}\left(S^{N-1}\right)}{q-rq}.
\end{multline}
\end{corollary}

\begin{proof}
By definition of Gagliardo seminorm (Definition \ref{def:Gagliardo seminorm}), change of variable formula, Fubini's theorem and additivity of integral we get
\begin{multline}
\label{eq:equation83}
\left[u_{\e}\right]^q_{W^{r,q}(\R^N,\R^d)}=\int_{\R^N}\left(\int_{\R^N}\frac{|u_\e(x)-u_\e(y)|^q}{|x-y|^{N+rq}}dy\right)dx
=\int_{\R^N}\left(\int_{\R^N}\frac{|u_\e(x)-u_\e(x+z)|^q}{|z|^{N+rq}}dz\right)dx
\\
=\int_{\R^N}\left(\int_{\R^N}\frac{|u_\e(x)-u_\e(x+z)|^q}{|z|^{N+rq}}dx\right)dz
=\int_{\R^N\setminus B_\gamma(0)}g^\e(z)dz
+\int_{B_\gamma(0)\setminus B_\beta(0)}g^\e(z)dz
+\int_{B_\beta(0)}g^\e(z)dz,
\end{multline}
where we denote
\begin{equation}
g^\e(z):=\int_{\R^N}\frac{|u_\e(x)-u_\e(x+z)|^q}{|z|^{N+rq}}dx,\quad z\in \R^N\setminus\{0\}.
\end{equation}
Therefore, we get \eqref{eq:equation73} by \eqref{eq:equation68},\eqref{eq:equation70},\eqref{eq:equation71} and \eqref{eq:equation83}. Inequality \eqref{eq:equation74} follows from \eqref{eq:equation73} choosing $\gamma=1$, $\beta=\e^{\frac{q}{q-rq}}$, and using that $\frac{1}{|\ln\e|}<1$ for every $\e\in (0,1/e)$.
\end{proof}

\section{Continuity of $G$-Functionals}
In this section, we define the upper and lower $G$-functionals (see Definition \ref{def:upper and lower functionals}). We prove continuity properties for these functionals (see Lemma \ref{lem:continuity of upper and lower $G$-functionals}). These continuity properties, in particular, allow us to generalize results involving $\eta\in C^1_c(\R^N)$ to cases where $\eta\in W^{1,1}(\R^N)$ (refer to the proof of Corollary \ref{cor: the upper(lower) limit of J equals to the upper(lower) limit of Besov approximation of u}). Additionally, we introduce the Gagliardo constants, which are specific instances of the $G$-functionals where the function $\eta$ is fixed (see Definition \ref{def:Gagliardo constants}).

\begin{definition}(The Upper and Lower $G$-Functionals)
\label{def:upper and lower functionals}

Let us define for $q\in [1,\infty)$, $r\in(0,1)$ and an $\mathcal{L}^N$-measurable set $E\subset\R^N$ the {\it upper $G$-functional} and the {\it lower $G$-functional}, respectively, to be
\begin{multline}
\overline{G}_{E},\underline{G}_{E}: B^r_{q,\infty}(\R^N,\R^d)\times W^{1,1}(\R^N)\to [0,\infty),
\\
\overline{G}_{E}(u,\eta):=\limsup_{\e\to
0^+}\frac{1}{|\ln{\e}|}\left[u*\eta_{(\e)}\right]^q_{W^{r,q}(E,\R^d)},\quad\underline{G}_{E}(u,\eta):=\liminf_{\e\to
0^+}\frac{1}{|\ln{\e}|}\left[u*\eta_{(\e)}\right]^q_{W^{r,q}(E,\R^d)}.
\end{multline}
\end{definition}

\begin{remark}(Well-definedness of the Upper and Lower $G$-Functionals)
\label{rem:Well-definedness of the upper and lower $G$-functionals}

The well-definedness of the upper and lower $G$-functionals follows immediately from \eqref{eq:equation74}. Note that  $\left[u*\eta_{(\e)}\right]_{W^{r,q}(\R^N,\R^d)}\in [0,\infty)$ for every $\e\in (0,\infty)$ assuming only that $u\in L^q(\R^N,\R^d)$ and $\eta\in W^{1,1}(\R^N)$: One can show by H{\"o}lder's inequality that the convolution $u*\eta$ lies in  $L^q(\R^N,\R^d)$, $1\leq q\leq \infty$, whenever $u\in L^q(\R^N,\R^d)$ and $\eta\in L^1(\R^N)$.
Therefore, if $1\leq q<\infty$, $u\in L^q(\R^N,\R^d)$ and $\eta\in W^{1,1}(\R^N)$, then $u_\e\in L^q(\R^N,\R^d)$, and it has weak derivatives $\frac{\partial}{\partial x_i}u_\e=u*\frac{\partial}{\partial x_i}\eta_{(\e)}\in L^q(\R^N,\R^d)$ for each $1\leq i\leq N$. Therefore, $u_\e\in W^{1,q}(\R^N,\R^d)\subset W^{r,q}(\R^N,\R^d)$, for every $r\in(0,1)$. Thus, $\left[u*\eta_{(\e)}\right]_{W^{r,q}(\R^N,\R^d)}\in [0,\infty)$ for every $\e\in (0,\infty)$.
\end{remark}

\begin{lemma}(Continuity of the Upper and Lower $G$-Functionals)
\label{lem:continuity of upper and lower $G$-functionals}

Let $q\in[1,\infty)$, $r\in (0,1)$
and $E\subset\R^N$ be an $\mathcal{L}^N$-measurable set.
\\
1. If $u\in B^r_{q,\infty}(\R^N,\R^d)$ and $\{\eta_n\}_{n=1}^\infty\subset W^{1,1}(\R^N)$ is a sequence such that $\eta_n$ converges to $\eta$ in $W^{1,1}(\R^N)$, then
\begin{equation}
\label{eq:equation88}
\lim_{n\to \infty}\overline{G}_{E}(u,\eta_n)=\overline{G}_{E}(u,\eta),\quad \lim_{n\to \infty}\underline{G}_{E}(u,\eta_n)=\underline{G}_{E}(u,\eta).
\end{equation}
2. If $\eta\in W^{1,1}(\R^N)$ and $\{u_n\}_{n=1}^\infty\subset B^r_{q,\infty}(\R^N,\R^d)$ is a sequence such that $u_n$ converges to $u$ in $B^r_{q,\infty}(\R^N,\R^d)$, which means that $\lim_{n\to \infty}\left(\|u-u_n\|_{L^q(\R^N,\R^d)}+[u-u_n]_{B^{r}_{q,\infty}(\R^N,\R^d)}\right)=0$, then
\begin{equation}
\label{eq:equation92}
\lim_{n\to \infty}\overline{G}_{E}(u_n,\eta)=\overline{G}_{E}(u,\eta),\quad \lim_{n\to \infty}\underline{G}_{E}(u_n,\eta)=\underline{G}_{E}(u,\eta).
\end{equation}
\end{lemma}
\begin{proof}
1. For every $n\in \N$ we get by \eqref{eq:equation74}
\begin{align}
\sup_{\e\in (0,1/e)}\frac{1}{|\ln{\e}|}\left[u*\left(\eta_n\right)_{(\e)}\right]^q_{W^{r,q}(E,\R^d)}<\infty,\sup_{\e\in (0,1/e)}\frac{1}{|\ln{\e}|}\left[u*\eta_{(\e)}\right]^q_{W^{r,q}(E,\R^d)}<\infty.
\end{align}
Therefore, by Lemma \ref{lem:liminfsup lemma} we get
\begin{multline}
\label{eq:equation90}
\left|\liminf_{\e\to
0^+}\frac{1}{|\ln{\e}|^{1/q}}\left[u*\left(\eta_n\right)_{(\e)}\right]_{W^{r,q}(E,\R^d)}-\liminf_{\e\to
0^+}\frac{1}{|\ln{\e}|^{1/q}}\left[u*\eta_{(\e)}\right]_{W^{r,q}(E,\R^d)}\right|
\\
\leq \limsup_{\e\to
0^+}\frac{1}{|\ln{\e}|^{1/q}}\left|\left[u*\left(\eta_n\right)_{(\e)}\right]_{W^{r,q}(E,\R^d)}-\left[u*\eta_{(\e)}\right]_{W^{r,q}(E,\R^d)}\right|,
\end{multline}
and
\begin{multline}
\label{eq:equation91}
\left|\limsup_{\e\to
0^+}\frac{1}{|\ln{\e}|^{1/q}}\left[u*\left(\eta_n\right)_{(\e)}\right]_{W^{r,q}(E,\R^d)}-\limsup_{\e\to
0^+}\frac{1}{|\ln{\e}|^{1/q}}\left[u*\eta_{(\e)}\right]_{W^{r,q}(E,\R^d)}\right|
\\
\leq \limsup_{\e\to
0^+}\frac{1}{|\ln{\e}|^{1/q}}\left|\left[u*\left(\eta_n\right)_{(\e)}\right]_{W^{r,q}(E,\R^d)}-\left[u*\eta_{(\e)}\right]_{W^{r,q}(E,\R^d)}\right|.
\end{multline}
By the triangle inequality for Gagliardo seminorm we get
\begin{multline}
\limsup_{\e\to
0^+}\frac{1}{|\ln{\e}|^{1/q}}\left|\left[u*\left(\eta_n\right)_{(\e)}\right]_{W^{r,q}(E,\R^d)}-\left[u*\eta_{(\e)}\right]_{W^{r,q}(E,\R^d)}\right|
\\
\leq \limsup_{\e\to
0^+}\frac{1}{|\ln{\e}|^{1/q}}\left[u*\left(\eta_n\right)_{(\e)}-u*\eta_{(\e)}\right]_{W^{r,q}(E,\R^d)}
=\limsup_{\e\to
0^+}\frac{1}{|\ln{\e}|^{1/q}}\left[u*\left(\left(\eta_n\right)_{(\e)}-\eta_{(\e)}\right)\right]_{W^{r,q}(E,\R^d)}
\\
=\limsup_{\e\to
0^+}\frac{1}{|\ln{\e}|^{1/q}}\left[u*\left(\eta_n-\eta\right)_{(\e)}\right]_{W^{r,q}(E,\R^d)}
\leq\limsup_{\e\to
0^+}\frac{1}{|\ln{\e}|^{1/q}}\left[u*\left(\eta_n-\eta\right)_{(\e)}\right]_{W^{r,q}(\R^N,\R^d)}.
\end{multline}
Therefore, by \eqref{eq:equation74}
\begin{multline}
\label{eq:equation87}
\limsup_{\e\to
0^+}\frac{1}{|\ln{\e}|}\left|\left[u*\left(\eta_n\right)_{(\e)}\right]_{W^{r,q}(E,\R^d)}-\left[u*\eta_{(\e)}\right]_{W^{r,q}(E,\R^d)}\right|^q
\leq\limsup_{\e\to
0^+}\frac{1}{|\ln{\e}|}\left[u*\left(\eta_n-\eta\right)_{(\e)}\right]^q_{W^{r,q}(\R^N,\R^d)}
\\
\leq \sup_{\e\in (0,1/e)}\frac{1}{|\ln\e|}\left[u*\left(\eta_n-\eta\right)_{(\e)}\right]^q_{W^{r,q}(\R^N,\R^d)}
\leq \|\eta_n-\eta\|^q_{L^1(\R^N)}2^q\|u\|^q_{L^q(\R^N,\R^d)}\frac{\mathcal{H}^{N-1}\left(S^{N-1}\right)}{rq}
\\
+\|\eta_n-\eta\|^q_{L^1(\R^N)}[u]^q_{B^{r}_{q,\infty}(\R^N,\R^d)}\mathcal{H}^{N-1}\left(S^{N-1}\right)\frac{q}{q-rq}
\\
+\left(\int_{\R^N}|\nabla\left(\eta_n-\eta\right)(v)|dv\right)^{q}2^q\|u\|^q_{L^q(\R^N,\R^d)}\frac{\mathcal{H}^{N-1}\left(S^{N-1}\right)}{q-rq}.
\end{multline}
Taking the limit as $n\to \infty$ in \eqref{eq:equation87} we get \eqref{eq:equation88} from \eqref{eq:equation90} and \eqref{eq:equation91}.
\\
2. Replacing $\eta_n$ with $\eta$ and $u$ with $u_n$, we get in the same way
\begin{multline}
\limsup_{\e\to
0^+}\frac{1}{|\ln{\e}|}\left|\left[u_n*\eta_{(\e)}\right]_{W^{r,q}(E,\R^d)}-\left[u*\eta_{(\e)}\right]_{W^{r,q}(E,\R^d)}\right|^q
\leq \|\eta\|^q_{L^1(\R^N)}2^q\|u-u_n\|^q_{L^q(\R^N,\R^d)}\frac{\mathcal{H}^{N-1}\left(S^{N-1}\right)}{rq}
\\
+\|\eta\|^q_{L^1(\R^N)}[u-u_n]^q_{B^{r}_{q,\infty}(\R^N,\R^d)}\mathcal{H}^{N-1}\left(S^{N-1}\right)\frac{q}{q-rq}
\\
+\left(\int_{\R^N}|\nabla\eta(v)|dv\right)^{q}2^q\|u-u_n\|^q_{L^q(\R^N,\R^d)}\frac{\mathcal{H}^{N-1}\left(S^{N-1}\right)}{q-rq}.
\end{multline}
Taking the limit as $n\to\infty$ we get \eqref{eq:equation92}.
\end{proof}

\begin{definition}(Gagliardo constants)
\label{def:Gagliardo constants}

Let $q\in [1,\infty)$, $r\in(0,1)$, let $E\subset\R^N$ be an $\mathcal{L}^N$-measurable set, and $\eta\in W^{1,1}(\R^N)$. We define the {\it $(r,q)$ upper Gagliardo constant of $u$ in $E$ with respect to $\eta$} as the quantity:
\begin{equation}
\limsup_{\e\to
0^+}\frac{1}{|\ln{\e}|}\left[u*\eta_{(\e)}\right]^q_{W^{r,q}(E,\R^d)}.
\end{equation}
Similarly, replacing the $\limsup$ by the $\liminf$, we define the {\it $(r,q)$ lower Gagliardo constant of $u$ in $E$ with respect to $\eta$}. If the limit exists, we refer to it as the {\it $(r,q)$ Gagliardo constant of $u$ in $E$ with respect to $\eta$}.
\end{definition}

\section{$B^{r,q}$-Functions}
In this section, we introduce the space of functions $B^{r,q}(E,\mathbb{R}^d)$ (see Definition \ref{def:functions in BVr,q}). We establish several properties of these functions, as detailed in Propositions \ref{prop:properties of Brq seminorms} and \ref{prop:Negligibility of the upper infinitesimal $B^{r,q}$-seminorm for Sobolev functions}, as well as Corollary \ref{cor:Non-equivalence of the Seminorms}. Additionally, we prove the equivalence between the space $B^{r,q}(\mathbb{R}^N,\mathbb{R}^d)$ and the Besov space $B^r_{q,\infty}(\mathbb{R}^N,\mathbb{R}^d)$ (refer to Theorem \ref{thm:characterization of Besov functions via double integral}).

\begin{definition}($B^{r,q}$-Seminorms)
\label{def:Brq seminorms}

Let us define for $r\in(0,1)$, $q\in[1,\infty)$, an $\mathcal{L}^N$-measurable set $E\subset\R^N$ and $\mathcal{L}^N$-measurable function $u:E\to \R^d$ the following two quantities:\\
The {\it $B^{r,q}$-seminorm} is defined by
\begin{equation}
|u|_{B^{r,q}(E,\R^d)}:=\sup_{\e\in(0,1)}\left(\int_{E}\frac{1}{\e^N}\int_{E\cap B_{\e}(x)}\frac{|u(x)-u(y)|^q}{|x-y|^{rq}}dydx\right)^{\frac{1}{q}};
\end{equation}
the {\it upper infinitesimal $B^{r,q}$-seminorm} is defined by
\begin{equation}
[u]_{B^{r,q}(E,\R^d)}:=\limsup_{\e\to 0^+}\left(\int_{E}\frac{1}{\e^N}\int_{E\cap B_{\e}(x)}\frac{|u(x)-u(y)|^q}{|x-y|^{rq}}dydx\right)^{\frac{1}{q}}.
\end{equation}
\end{definition}

\begin{definition}(The Space $B^{r,q}$)
\label{def:functions in BVr,q}

Let $r\in(0,1)$, $q\in[1,\infty)$ and an $\mathcal{L}^N$-measurable set $E\subset\R^N$. We define a set
\begin{equation}
B^{r,q}(E,\R^d):=\bigg\{u\in L^q(E,\R^d):|u|_{B^{r,q}(E,\R^d)}<\infty\bigg\}.
\end{equation}
We define the local space $B^{r,q}_{\text{loc}}(E,\R^d)$ as follows: $u\in B^{r,q}_{\text{loc}}(E,\R^d)$ if and only if $u\in L^q_{\text{loc}}(E,\R^d)$ and $u\in B^{r,q}(K,\R^d)$ for every compact set $K\subset E$.
\end{definition}

\begin{proposition}(Properties of $B^{r,q}$-Seminorms)
\label{prop:properties of Brq seminorms}

Let $r\in(0,1)$, $q\in[1,\infty)$ and $E\subset\R^N$ be an $\mathcal{L}^N$-measurable set. Then,
\\
1. The $B^{r,q}$-seminorm and the upper infinitesimal $B^{r,q}$-seminorm are seminorms on $B^{r,q}(E,\R^d)$;
\\
2. For $u\in L^q(E,\R^d)$, $|u|_{B^{r,q}(E,\R^d)}<\infty$ if and only if $[u]_{B^{r,q}(E,\R^d)}<\infty$;
\\
3. For an open set $\Omega\subset\R^N$,  $u\in B^{r,q}_{\text{loc}}(\Omega,\R^d)$ if and only if for every compact set $K\subset \Omega$ we have
\begin{equation}
\label{eq:equation108}
\limsup_{\e\to 0^+}\int_{K}\frac{1}{\e^N}\int_{B_{\e}(x)}\frac{|u(x)-u(y)|^q}{|x-y|^{rq}}dydx<\infty.
\end{equation}
\\
4. Let us denote:
\begin{equation}
\label{eq:the norms 1 and 2}
\|u\|_1:=[u]_{B^{r,q}(E,\R^d)}+\|u\|_{L^q(E,\R^d)},\quad \|u\|_2:=|u|_{B^{r,q}(E,\R^d)}+\|u\|_{L^q(E,\R^d)}.
\end{equation}
Then, $\|\cdot\|_1,\|\cdot\|_2$ are norms on the space $B^{r,q}(E,\R^d)$\footnote{As usual, on the space of equivalent classes obtained by equality $\mathcal{L}^N$-almost everywhere.} and $\left(B^{r,q}(E,\R^d),\|\cdot\|_2\right)$ is a Banach space.
\end{proposition}

\begin{proof}
1. Let $u,v\in B^{r,q}(E,\R^d)$ and $a\in\R$. It follows immediately from definitions that $|u|_{B^{r,q}(E,\R^d)},[u]_{B^{r,q}(E,\R^d)}$ are non-negative and homogeneous,  which means that $|au|_{B^{r,q}(E,\R^d)}=|a||u|_{B^{r,q}(E,\R^d)}$ and $[au]_{B^{r,q}(E,\R^d)}=|a|[u]_{B^{r,q}(E,\R^d)}$. We have by Minkowski's inequality
\begin{multline}
\label{eq:calculation for proving seminorm property}
\left(\int_{E}\int_{E\cap B_\e(x)}\frac{|(u+v)(x)-(u+v)(y)|^q}{|x-y|^{rq}}dydx\right)^{\frac{1}{q}}
\\
\leq \left(\int_{E}\int_{E}\left[\chi_{B_\e(x)}(y)\frac{|u(x)-u(y)|}{|x-y|^{r}}+\chi_{B_\e(x)}(y)\frac{|v(x)-v(y)|}{|x-y|^{r}}\right]^qdydx\right)^{\frac{1}{q}}
\\
\leq \left(\int_{E}\int_{E}\left[\chi_{B_\e(x)}(y)\frac{|u(x)-u(y)|}{|x-y|^{r}}\right]^qdydx\right)^{\frac{1}{q}}
+\left(\int_{E}\int_{E}\left[\chi_{B_\e(x)}(y)\frac{|v(x)-v(y)|}{|x-y|^{r}}\right]^qdydx\right)^{\frac{1}{q}}.
\end{multline}
The triangle inequality for $[\cdot]_{B^{r,q}(E,\R^d)},|\cdot|_{B^{r,q}(E,\R^d)}$ follows from \eqref{eq:calculation for proving seminorm property}.
\\
2. Since for every $\mathcal{L}^N$-measurable function $u:E\to \R^d$ we have $[u]_{B^{r,q}(E,\R^d)}\leq |u|_{B^{r,q}(E,\R^d)}$, then the finiteness of $|u|_{B^{r,q}(E,\R^d)}$ implies the finiteness of $[u]_{B^{r,q}(E,\R^d)}$. Assume $[u]_{B^{r,q}(E,\R^d)}<\infty$. Then, there exists a number
$0<\e_0<1$
such that
\begin{equation}
\sup_{\e\in (0,\e_0]}\int_{E}\left(\int_{E\cap B_\e(x)}\frac{1}{\e^N}\frac{|u(x)-u(y)|^q}{|x-y|^{rq}}dy\right)dx<\infty.
\end{equation}
We have
\begin{multline}
\sup_{\e\in [\e_0,1)}\int_{E}\left(\int_{E\cap B_\e(x)}\frac{1}{\e^N}\frac{|u(x)-u(y)|^q}{|x-y|^{rq}}dy\right)dx
\leq \int_{E}\left(\int_{E\cap B_{\e_0}(x)}\frac{1}{\e_0^N}\frac{|u(x)-u(y)|^q}{|x-y|^{rq}}dy\right)dx
\\
+\sup_{\e\in [\e_0,1)}\int_{E}\left(\int_{E\cap \left(B_\e(x)\setminus B_{\e_0}(x)\right)}\frac{1}{\e^N}\frac{|u(x)-u(y)|^q}{|x-y|^{rq}}dy\right)dx,
\end{multline}
and
\begin{multline}
\sup_{\e\in [\e_0,1)}\int_{E}\left(\int_{E\cap \left(B_\e(x)\setminus B_{\e_0}(x)\right)}\frac{1}{\e^N}\frac{|u(x)-u(y)|^q}{|x-y|^{rq}}dy\right)dx
\\
\leq 2^{q-1}\frac{1}{\e_0^{N+rq}}\sup_{\e\in [\e_0,1)}\int_{E}\left(\int_{E\cap \left(B_\e(x)\setminus B_{\e_0}(x)\right)}|u(x)|^qdy\right)dx
\\
+2^{q-1}\frac{1}{\e_0^{N+rq}}\sup_{\e\in [\e_0,1)}\int_{E}\left(\int_{E\cap \left(B_{\e}(x)\setminus B_{\e_0}(x)\right)}|u(y)|^qdy\right)dx
\leq2^q\frac{1}{\e_0^{N+rq}}\Leb^N(B_1(0))\|u\|^q_{L^q(E,\mathbb{R}^d)}<\infty.
\end{multline}
Thus, $|u|_{B^{r,q}(E,\R^d)}<\infty$.
\\
3. If for a compact set $K\subset\R^N$ we have \eqref{eq:equation108}, then $[u]_{B^{r,q}(K,\R^d)}<\infty$ and by item 2 we have also $|u|_{B^{r,q}(K,\R^d)}<\infty$, hence $u\in B^{r,q}(K,\R^d)$. For the opposite implication, let $u\in B^{r,q}_{\text{loc}}(\Omega,\R^d)$ and $K\subset\Omega$ be a compact set. Let $\Omega_0\subset\subset\Omega$ be an open set containing $K$. Since $\Omega_0$ is open, then we have for every small enough $\e\in(0,\infty)$ that $K+B_\e(0)\subset\Omega_0$. Hence, by item 2 we have
\begin{equation}
\limsup_{\e\to 0^+}\int_{K}\frac{1}{\e^N}\int_{B_{\e}(x)}\frac{|u(x)-u(y)|^q}{|x-y|^{rq}}dydx\leq \limsup_{\e\to 0^+}\int_{\Omega_0}\frac{1}{\e^N}\int_{B_{\e}(x)\cap \Omega_0}\frac{|u(x)-u(y)|^q}{|x-y|^{rq}}dydx<\infty.
\end{equation}
\\
4.  Since by assertion 1, $[\cdot]_{B^{r,q}(E,\R^d)},|\cdot|_{B^{r,q}(E,\R^d)}$ are seminorms and $\|\cdot\|_{L^q(E,\R^d)}$ is a norm, then $\|\cdot\|_1,\|\cdot\|_2$ are norms on the space $B^{r,q}(E,\R^d)$. The space $\left(B^{r,q}(E,\R^d),\|\cdot\|_2\right)$ is complete: let $\{u_n\}_{n=1}^\infty\subset B^{r,q}(E,\R^d)$ be a Cauchy sequence. Then, it is also a Cauchy sequence in $L^q(E,\R^d)$, so, since $L^q(E,\R^d)$ is complete, there exists a function $u\in L^q(E,\R^d)$ such that $\{u_n\}_{n=1}^\infty$ converges to $u$ in $L^q(E,\R^d)$. Let $\{u_{n_k}\}_{n=1}^\infty$ be a subsequence that converges to $u$ also $\mathcal{L}^N$-almost everywhere. Since $\{u_{n_k}\}_{n=1}^\infty$ is a Cauchy sequence, then it is bounded, so there exists a number $M$ such that $|u_{n_k}|_{B^{r,q}(E,\R^d)}\leq M$ for every $k\in\N$. By Fatou's lemma we get
\begin{equation}
|u|_{B^{r,q}(E,\R^d)}\leq \liminf_{k\to \infty}\sup_{\e\in(0,1)}\left(\int_{E}\frac{1}{\e^N}\int_{E\cap B_{\e}(x)}\frac{|u_{n_k}(x)-u_{n_k}(y)|^q}{|x-y|^{rq}}dydx\right)^{\frac{1}{q}}\leq M.
\end{equation}
Thus, $u\in B^{r,q}(E,\R^d)$. Let us prove that $u_n$ converges to $u$ in $B^{r,q}(E,\R^d)$. Let $\xi>0$. Since $\{u_n\}_{n=1}^\infty$ is a Cauchy sequence in $B^{r,q}(E,\R^d)$, there exists  $N_0\in \N$ such that for every $n,k>N_0$, we get
\begin{multline}
\xi\geq |u_n-u_k|_{B^{r,q}(E,\R^d)}=\sup_{\e\in(0,1)}\int_{E}\frac{1}{\e^N}\int_{E\cap B_{\e}(x)}\frac{|(u_n-u_k)(x)-(u_n-u_k)(y)|^q}{|x-y|^{rq}}dydx.
\end{multline}

Therefore, for every $n>N_0$, we obtain
\begin{multline}
\xi\geq \liminf_{k\to \infty}\left(\sup_{\e\in(0,1)}\int_{E}\frac{1}{\e^N}\int_{E\cap B_{\e}(x)}\frac{|(u_n-u_{n_k})(x)-(u_n-u_{n_k})(y)|^q}{|x-y|^{rq}}dydx\right)
\\
\geq \sup_{\e\in(0,1)}\left(\liminf_{k\to \infty}\int_{E}\frac{1}{\e^N}\int_{E\cap B_{\e}(x)}\frac{|(u_n-u_{n_k})(x)-(u_n-u_{n_k})(y)|^q}{|x-y|^{rq}}dydx\right)
\\
\geq \sup_{\e\in(0,1)}\left(\int_{E}\frac{1}{\e^N}\int_{E\cap B_{\e}(x)}\frac{|(u_n-u)(x)-(u_n-u)(y)|^q}{|x-y|^{rq}}dydx\right)=|u_n-u|_{B^{r,q}(E,\R^d)}.
\end{multline}
\end{proof}

\begin{proposition}(Continuous Embedding of $W^{r,q}$ into $B^{r,q}$, and Negligibility of the Upper Infinitesimal $B^{r,q}$-Seminorm for Sobolev Functions)
\label{prop:Negligibility of the upper infinitesimal $B^{r,q}$-seminorm for Sobolev functions}

Let $r\in(0,1)$, $q\in [1,\infty)$ and $E\subset\R^N$ be an $\mathcal{L}^N$-measurable set. Then, the space $W^{r,q}(E,\R^d)$ with the norm $[\cdot]_{W^{r,q}(E,\R^d)}+\|\cdot\|_{L^q(E,\R^d)}$ is continuously embedded in the space $B^{r,q}(E,\R^d)$ with the norm $\|\cdot\|_2$ defined in \eqref{eq:the norms 1 and 2}. Moreover, $[u]_{B^{r,q}(E,\R^d)}=0$ for every $u\in W^{r,q}(E,\R^d)$.
\end{proposition}
\begin{proof}
We have for every $\e\in(0,\infty)$ and $u\in W^{r,q}(E,\R^d)$
\begin{multline}
\label{eq:equation109}
\infty>\int_{E}\int_{E}\frac{|u(x)-u(y)|^q}{|x-y|^{N+rq}}dydx\geq \int_{E}\left(\int_{E\cap B_{\e}(x)}\frac{|u(x)-u(y)|^q}{|x-y|^{N+rq}}dy\right)dx
\\
\geq \int_{E}\frac{1}{\e^N}\int_{E\cap B_{\e}(x)}\frac{|u(x)-u(y)|^q}{|x-y|^{rq}}dydx.
\end{multline}
By \eqref{eq:equation109} we conclude that $W^{r,q}(E,\R^d)$ is continuously embedded in $B^{r,q}(E,\R^d)$. Notice that
\begin{multline}
\int_{E}\left(\sup_{\e\in(0,\infty)}\int_{E\cap B_{\e}(x)}\frac{|u(x)-u(y)|^q}{|x-y|^{N+rq}}dy\right)dx\leq \int_{E}\int_{E}\frac{|u(x)-u(y)|^q}{|x-y|^{N+rq}}dydx<\infty;
\\
\lim_{\e\to 0^+}\int_{E\cap B_{\e}(x)}\frac{|u(x)-u(y)|^q}{|x-y|^{N+rq}}dy=0,\quad \text{for $\mathcal{L}^N$-almost every $x\in E$}.
\end{multline}
Therefore, by Dominated Convergence Theorem
\begin{multline}
[u]_{B^{r,q}(E,\R^d)}\leq  \limsup_{\e\to 0^+}\int_{E}\int_{E\cap B_{\e}(x)}\frac{|u(x)-u(y)|^q}{|x-y|^{N+rq}}dydx
\\
=\int_{E}\lim_{\e\to 0^+}\left(\int_{E\cap B_{\e}(x)}\frac{|u(x)-u(y)|^q}{|x-y|^{N+rq}}dy\right)dx=0.
\end{multline}
\end{proof}

\begin{corollary}(Non-equivalence of the Seminorms $|\cdot|_{B^{r,q}},[\cdot]_{B^{r,q}}$)
\label{cor:Non-equivalence of the Seminorms}

Let $r\in(0,1)$, $q\in [1,\infty)$ and $\Omega\subset\R^N$ be an open set which is not empty. Let $\|\cdot\|_1,\|\cdot\|_2$ be the norms defined in \eqref{eq:the norms 1 and 2}. Then, the space $\left(B^{r,q}(\Omega,\R^d),\|\cdot\|_1\right)$ is not a Banach space. In particular, the seminorms $|\cdot|_{B^{r,q}},[\cdot]_{B^{r,q}}$ are not equivalent.
\end{corollary}

\begin{proof}
Let $u\in L^q(\Omega,\R^d)$ such that $[u]_{B^{r,q}(\Omega,\R^d)}=\infty$. Let $\{u_n\}_{n=1}^\infty\subset C^1_c(\Omega,\R^d)$ be a sequence which converges to $u$ in $L^q(\Omega,\R^d)$, so it is also a Cauchy sequence in $L^q(\Omega,\R^d)$. Therefore, by Proposition \ref{prop:Negligibility of the upper infinitesimal $B^{r,q}$-seminorm for Sobolev functions} we have that $\{u_n\}_{n=1}^\infty\subset W^{r,q}(\Omega,\R^d)\subset B^{r,q}(\Omega,\R^d)$ and this sequence is also a Cauchy sequence with respect to the norm $\|\cdot\|_1$ because $\|u_n-u_k\|_1=\|u_n-u_k\|_{L^q(\Omega,\R^d)}$ for every $k,n\in\N$. Thus, $\{u_n\}_{n=1}^\infty$ is a Cauchy sequence in the space $\left(B^{r,q}(\Omega,\R^d),\|\cdot\|_1\right)$ which does not have a limit in the space. Since by item 4 of Proposition \ref{prop:properties of Brq seminorms} $\left(B^{r,q}(\Omega,\R^d),\|\cdot\|_2\right)$ is a Banach space, then the norms $\|\cdot\|_1,\|\cdot\|_2$ are not equivalent and so as the seminorms $|\cdot|_{B^{r,q}},[\cdot]_{B^{r,q}}$.
\end{proof}

\begin{theorem}(Equivalence Between $B^{r,q}$-Spaces and Besov Spaces $B^r_{q,\infty}$)
\label{thm:characterization of Besov functions via double integral}

Let $r\in(0,1)$, $q\in[1,\infty)$. Then,
\begin{equation}
\label{eq:equality between global Besov and Brq}
B^r_{q,\infty}(\R^N,\R^d)=B^{r,q}(\R^N,\R^d),
\end{equation}
and for every open set $\Omega\subset\R^N$
\begin{equation}
\label{eq:equality between local Besov and Brq}
\left(B^r_{q,\infty}\right)_{\text{loc}}(\Omega,\R^d)=B^{r,q}_{\text{loc}}(\Omega,\R^d).
\end{equation}
\end{theorem}
\begin{proof}
Assume that $u\in B^r_{q,\infty}(\R^N,\R^d)$. Then, for every $\e\in(0,\infty)$
\begin{multline}
\int_{\R^N}\frac{1}{\e^N}\int_{B_{\e}(x)}\frac{|u(x)-u(y)|^q}{|x-y|^{rq}}dydx=\int_{\R^N}\int_{B_{1}(0)}\frac{|u(x)-u(x+\e z)|^q}{|\e z|^{rq}}dzdx
\\
=\int_{B_{1}(0)}\left(\int_{\R^N}\frac{|u(x)-u(x+\e z)|^q}{|\e z|^{rq}}dx\right)dz
\leq [u]^q_{B^{r}_{q,\infty}(\R^N,\R^d)}\mathcal{L}^N\left(B_1(0)\right)<\infty.
\end{multline}
Thus, we get
\begin{equation}
\sup_{\e\in(0,\infty)}\int_{\R^N}\frac{1}{\e^N}\int_{B_{\e}(x)}\frac{|u(x)-u(y)|^q}{|x-y|^{rq}}dydx\leq [u]^q_{B^{r}_{q,\infty}(\R^N,\R^d)}\mathcal{L}^N\left(B_1(0)\right)<\infty.
\end{equation}
Thus, $u\in B^{r,q}(\R^N,\R^d)$. Assume that $u\in B^{r,q}(\R^N,\R^d)$.
\\
Step 1: for every $h_1,h_2\in \R^N$ such that $0\notin\{h_1,h_2,h_1+h_2\}$ we have
\begin{multline}
\label{eq:equation95}
\int_{\R^N}\frac{|u(x+(h_1+h_2))-u(x)|^q}{|h_1+h_2|^{rq}}dx=\int_{\R^N}\frac{|\left(u(x+(h_1+h_2))-u(x+h_1)\right)+\left(u(x+h_1)-u(x)\right)|^q}{|h_1+h_2|^{rq}}dx
\\
\leq 2^{q-1}\int_{\R^N}\frac{|u(x+(h_1+h_2))-u(x+h_1)|^q}{|h_1+h_2|^{rq}}dx+2^{q-1}\int_{\R^N}\frac{|u(x+h_1)-u(x)|^q}{|h_1+h_2|^{rq}}dx
\\
=\frac{2^{q-1}|h_2|^{rq}}{|h_1+h_2|^{rq}}\int_{\R^N}\frac{|u(x+h_2)-u(x)|^q}{|h_2|^{rq}}dx+\frac{2^{q-1}|h_1|^{rq}}{|h_1+h_2|^{rq}}\int_{\R^N}\frac{|u(x+h_1)-u(x)|^q}{|h_1|^{rq}}dx.
\end{multline}
Step 2: let $\nu\in S^{N-1}$, $\e\in(0,\infty)$ and $z\in\R^N$. Denote $h_1:=\e z$ and $h_2:=\e(\nu-z)$. Note that $B_{1/2}\left(\frac{1}{2}\nu\right)\subset B_1(0)\cap B_1(\nu)$.
For $z\in B_{1/2}\left(\frac{1}{2}\nu\right)$, we get by \eqref{eq:equation95}
\begin{multline}
\label{eq:equation96}
\int_{\R^N}\frac{|u(x+\e\nu)-u(x)|^q}{\e^{rq}}dx
\\
\leq 2^{q-1}|\nu-z|^{rq}\int_{\R^N}\frac{|u(x+\e(\nu-z))-u(x)|^q}{|\e(\nu-z)|^{rq}}dx+2^{q-1}|z|^{rq}\int_{\R^N}\frac{|u(x+\e z)-u(x)|^q}{|\e z|^{rq}}dx
\\
\leq 2^{q-1}\int_{\R^N}\frac{|u(x+\e(\nu-z))-u(x)|^q}{|\e(\nu-z)|^{rq}}dx+2^{q-1}\int_{\R^N}\frac{|u(x+\e z)-u(x)|^q}{|\e z|^{rq}}dx.
\end{multline}
Taking the average with respect to $dz$ on the ball $B_{1/2}\left(\frac{1}{2}\nu\right)$ of both sides of the inequality \eqref{eq:equation96}, we get
\begin{multline}
\label{eq:equation97}
\int_{\R^N}\frac{|u(x+\e\nu)-u(x)|^q}{\e^{rq}}dx
\leq \frac{2^{q-1}}{\mathcal{L}^N\left(B_{1/2}\left(\frac{1}{2}\nu\right)\right)}\int_{B_{1/2}\left(\frac{1}{2}\nu\right)}\int_{\R^N}\frac{|u(x+\e(\nu-z))-u(x)|^q}{|\e(\nu-z)|^{rq}}dxdz
\\
+\frac{2^{q-1}}{\mathcal{L}^N\left(B_{1/2}\left(\frac{1}{2}\nu\right)\right)}\int_{B_{1/2}\left(\frac{1}{2}\nu\right)}\int_{\R^N}\frac{|u(x+\e z)-u(x)|^q}{|\e z|^{rq}}dxdz
\\
\leq \frac{2^{N+q-1}}{\mathcal{L}^N\left(B_1(0)\right)}\int_{B_1(\nu)}\int_{\R^N}\frac{|u(x+\e(\nu-z))-u(x)|^q}{|\e(\nu-z)|^{rq}}dxdz
\\
+\frac{2^{N+q-1}}{\mathcal{L}^N\left(B_1(0)\right)}\int_{B_1(0)}\int_{\R^N}\frac{|u(x+\e z)-u(x)|^q}{|\e z|^{rq}}dxdz
=\frac{2^{N+q}}{\mathcal{L}^N\left(B_1(0)\right)}\int_{B_1(0)}\int_{\R^N}\frac{|u(x+\e z)-u(x)|^q}{|\e z|^{rq}}dxdz.
\end{multline}
Therefore, since $u\in B^{r,q}(\R^N,\R^d)$, then
\begin{multline}
\label{eq:equation98}
\limsup_{\e\to 0^+}\left(\sup_{\nu\in S^{N-1}}\int_{\R^N}\frac{|u(x+\e\nu)-u(x)|^q}{\e^{rq}}dx\right)
\leq \frac{2^{N+q}}{\mathcal{L}^N\left(B_1(0)\right)}\limsup_{\e\to 0^+}\int_{B_1(0)}\int_{\R^N}\frac{|u(x+\e z)-u(x)|^q}{|\e z|^{rq}}dxdz
\\
= \frac{2^{N+q}}{\mathcal{L}^N\left(B_1(0)\right)}\limsup_{\e\to 0^+}\int_{\R^N}\frac{1}{\e^N}\int_{B_{\e}(x)}\frac{|u(x)-u(y)|^q}{|x-y|^{rq}}dydx<\infty.
\end{multline}
Step 3: notice that
\begin{multline}
\label{eq:equation99}
[u]^q_{B^{r}_{q,\infty}(\R^N,\R^d)}=\sup_{h\in \R^N\setminus\{0\}}\int_{\R^N}\frac{|u(x+h)-u(x)|^q}{|h|^{rq}}dx
\\
=\sup_{\e\in(0,\infty)}\left(\sup_{|h|=\e}\int_{\R^N}\frac{|u(x+h)-u(x)|^q}{\e^{rq}}dx\right)
=\sup_{\e\in(0,\infty)}\left(\sup_{h\in S^{N-1}}\int_{\R^N}\frac{|u(x+\e h)-u(x)|^q}{\e^{rq}}dx\right).
\end{multline}
By \eqref{eq:equation98} there exists $\delta\in(0,\infty)$ such that
\begin{equation}
\label{eq:equation100}
\sup_{\e\in(0,\delta)}\left(\sup_{h\in S^{N-1}}\int_{\R^N}\frac{|u(x+\e h)-u(x)|^q}{\e^{rq}}dx\right)<\infty.
\end{equation}
Therefore, by \eqref{eq:equation99} and \eqref{eq:equation100} we get
\begin{multline}
[u]^q_{B^{r}_{q,\infty}(\R^N,\R^d)}\leq \sup_{\e\in(0,\delta)}\left(\sup_{h\in S^{N-1}}\int_{\R^N}\frac{|u(x+\e h)-u(x)|^q}{\e^{rq}}dx\right)+\sup_{\e\in[\delta,\infty)}\left(\sup_{h\in S^{N-1}}\int_{\R^N}\frac{|u(x+\e h)-u(x)|^q}{\e^{rq}}dx\right)
\\
\leq \sup_{\e\in(0,\delta)}\left(\sup_{h\in S^{N-1}}\int_{\R^N}\frac{|u(x+\e h)-u(x)|^q}{\e^{rq}}dx\right)+\frac{2^q}{\delta^{rq}}\|u\|^q_{L^q(\R^N,\R^d)}<\infty.
\end{multline}
Thus, $u\in B^r_{q,\infty}(\R^N,\R^d)$. It completes the proof of \eqref{eq:equality between global Besov and Brq}. We will derive the local case \eqref{eq:equality between local Besov and Brq} from the global one \eqref{eq:equality between global Besov and Brq}.
\\
Assume now that $u\in \left(B^r_{q,\infty}\right)_{\text{loc}}(\Omega,\R^d)$. Let $K\subset\Omega$ be a compact set and let $\Omega_0\subset\subset\Omega$ be an open set such that $K\subset\Omega_0$. Let $g\in B^r_{q,\infty}(\R^N,\R^d)$ be such that $u=g$ $\mathcal{L}^N$-almost everywhere in $\Omega_0$. We have for $\e\in(0,\infty)$ such that $K+B_\e(0)\subset\Omega_0$
\begin{multline}
\int_{K}\frac{1}{\e^N}\int_{B_{\e}(x)}\frac{|u(x)-u(y)|^q}{|x-y|^{rq}}dydx=\int_{K}\frac{1}{\e^N}\int_{B_{\e}(x)}\frac{|g(x)-g(y)|^q}{|x-y|^{rq}}dydx
\\
\leq \int_{\R^N}\frac{1}{\e^N}\int_{B_{\e}(x)}\frac{|g(x)-g(y)|^q}{|x-y|^{rq}}dydx.
\end{multline}
By \eqref{eq:equality between global Besov and Brq}, we get
\begin{equation}
\limsup_{\e\to 0^+}\int_{K}\frac{1}{\e^N}\int_{B_{\e}(x)}\frac{|u(x)-u(y)|^q}{|x-y|^{rq}}dydx\leq\limsup_{\e\to 0^+}
\int_{\R^N}\frac{1}{\e^N}\int_{B_{\e}(x)}\frac{|g(x)-g(y)|^q}{|x-y|^{rq}}dydx<\infty.
\end{equation}
By item 3 of Proposition \ref{prop:properties of Brq seminorms} we conclude that $u\in B^{r,q}_{\text{loc}}(\Omega,\R^d)$.
Assume that  $u\in B^{r,q}_{\text{loc}}(\Omega,\R^d)$. Let $K\subset\Omega$ be a compact set and let $\Omega_0\subset\subset\Omega_1\subset\subset\Omega$ be open sets such that $K\subset\Omega_0$. Let $f\in C^{0,r}_{c}(\R^N)$\footnote{The space of Hölder continuous functions with exponent $r$ and compact support.} which is constant $1$ on $K$ and constant $0$ outside $\Omega_0$. We have for $g:=uf$ and $\e\in(0,\infty)$ such that $\R^N\setminus\Omega_1+B_\e(0)\subset \R^N\setminus\Omega_0$
\begin{multline}
\int_{\R^N}\frac{1}{\e^N}\int_{B_{\e}(x)}\frac{|g(x)-g(y)|^q}{|x-y|^{rq}}dydx
\\
=\int_{\Omega_1}\frac{1}{\e^N}\int_{B_{\e}(x)}\frac{|u(x)f(x)-u(x)f(y)+u(x)f(y)-u(y)f(y)|^q}{|x-y|^{rq}}dydx
\\
\leq 2^{q-1}\int_{\Omega_1}\frac{|u(x)|^q}{\e^N}\int_{B_{\e}(x)}\frac{|f(x)-f(y)|^q}{|x-y|^{rq}}dydx
+2^{q-1}\|f\|^q_{L^\infty(\R^N)}\int_{\Omega_1}\frac{1}{\e^N}\int_{B_{\e}(x)}\frac{|u(x)-u(y)|^q}{|x-y|^{rq}}dydx
\\
\leq 2^{q-1}C^q\mathcal{L}^N(B_1(0))\int_{\Omega_1}|u(x)|^qdx
+2^{q-1}\|f\|^q_{L^\infty(\R^N)}\int_{\Omega_1}\frac{1}{\e^N}\int_{B_{\e}(x)}\frac{|u(x)-u(y)|^q}{|x-y|^{rq}}dydx,
\end{multline}
where $C$ is a number such that $|f(x)-f(y)|\leq C|x-y|^r$ for $x,y\in\R^N$.
Therefore,
\begin{equation}
\limsup_{\e\to 0^+}\int_{\R^N}\frac{1}{\e^N}\int_{B_{\e}(x)}\frac{|g(x)-g(y)|^q}{|x-y|^{rq}}dydx<\infty,
\end{equation}
and by \eqref{eq:equality between global Besov and Brq} we conclude that $g\in B^r_{q,\infty}(\R^N,\R^d)$. Thus, $u\in \left(B^r_{q,\infty}\right)_{\text{loc}}(\Omega,\R^d)$.
\end{proof}

\section{Kernels}
In this section, we analyse the concept of a kernel (see Definition \ref{def:kernel}). Additionally, we discuss specific kernels, namely the logarithmic and trivial kernels (see Definitions \ref{def:logarithmic kernel} and \ref{def:definition of trivial kernel}), and establish their properties.

\begin{definition}
\label{def:compact support property}
(Compact Support Property)

Let $a\in(0,\infty]$ and
$\rho_\e:(0,\infty)\to [0,\infty),\e\in (0,a),$
be a family of functions. We say that the family $\{\rho_\e\}_{\e\in (0,a)}$ has the {\it compact support property} if for every $r>0$ there exists $\delta_r>0$ such that $\supp(\rho_\e)\subset B_r(0)$ for every
$\e\in (0,\delta_r)$.
\end{definition}

Note that, if the functions $\{\rho_\e\}_{\e\in (0,a)}$ are $\mathcal{L}^1$-measurable, the compact support property implies the decreasing support property (see Definition \ref{def:decreasing support property}).

\begin{definition}(Logarithmic Kernel)
\label{def:logarithmic kernel}

For every $\e\in (0,1/e)$ and $\omega\in(0,1)$ let us define a function
\begin{equation}
\rho_{\e,\omega}(r):=\frac{1}{\mathcal{H}^{N-1}\left(S^{N-1}\right)\left(|\ln \e|-|\ln R_{\e,\omega}|\right)}\frac{1}{r^N}\chi_{[\e,R_{\e,\omega})}(r),\quad \rho_{\e,\omega}:(0,\infty)\to [0,\infty),
\end{equation}
where $R_{\e,\omega}:=\frac{1}{|\ln\e|^{\omega}}$, and $\chi_{[\e,R_{\e,\omega})}$ is the characteristic function of the interval $[\e,R_{\e,\omega})$.
We call the family of functions $\{\rho_{\e,\omega}\}_{\e\in(0,1/e)}$ the {\it$N$-dimensional logarithmic kernel}, or just {\it logarithmic kernel}.
\end{definition}
\begin{remark}(Comments about the  Logarithmic Kernel)
\\
1. Note that for every $\e,\omega\in (0,1)$ we have $\e<R_{\e,\omega}$: $\e<R_{\e,\omega}$ if and only if $\e<\frac{1}{\ln\left(\frac{1}{\e}\right)^{\omega}}$ if and only if $\e^{1/\omega}\ln\left(\frac{1}{\e}\right)<1$. The last inequality holds since $\ln(z)<z$ for every $z\in (0,\infty)$.
\\
2. Note that for $\e\in (0,1)$, $\ln R_{\e,\omega}=-\omega\ln\left(\ln\left(\frac{1}{\e}\right)\right)$, and for $\e\in (0,1/e)$, $|\ln R_{\e,\omega}|=\omega\ln\left(\ln\left(\frac{1}{\e}\right)\right)$, so $|\ln \e|-|\ln R_{\e,\omega}|=\ln\left(\frac{1}{\e}\right)-\omega\ln\left(\ln\left(\frac{1}{\e}\right)\right)=\ln\left(\frac{1}{\e}\right)+\ln\left(\frac{1}{\left(\ln\left(\frac{1}{\e}\right)\right)^{\omega}}\right)=\ln\left(\frac{1}{\e\left(\ln\left(\frac{1}{\e}\right)\right)^{\omega}}\right)>0$. The last inequality holds since $\e\left(\ln\left(\frac{1}{\e}\right)\right)^{\omega}<1$.
\\
3. By  L'hopital's rule we have $\lim_{x \to \infty}\frac{\ln\left(\ln\left(x\right)\right)}{\ln\left(x\right)}=0$, so we get by definition of $R_{\e,\omega}$
\begin{align}
\label{eq:equation42}
\lim_{\e \to 0^+}\frac{|\ln R_{\e,\omega}|}{|\ln\e|}=\lim_{\e \to 0^+}\frac{\omega\ln\left(\ln\left(\frac{1}{\e}\right)\right)}{\ln\left(\frac{1}{\e}\right)}=0.
\end{align}
\end{remark}

\begin{proposition}(Properties of the Logarithmic Kernel)
\label{prop:Properties of the logarithmic kernel}

For $\omega\in(0,1)$, the logarithmic kernel $\{\rho_{\e,\omega}\}_{\e\in(0,1/e)}$ has the following properties:\\
1. The logarithmic kernel is a kernel that also possesses the compact support property;
\\
2. $\lim_{\e\to 0^+}\e^\alpha\int_{\R^N}\frac{\rho_{\e,\omega}(|z|)}{|z|^\alpha}dz=0$,\,\, $\forall \alpha\in (0,\infty)$.
\end{proposition}
\begin{proof}
1. It is easy to see that for every $\e\in(0,1/e)$, $\omega\in (0,1)$, the function $\rho_{\e,\omega}$ is $\mathcal{L}^1$-measurable. By polar coordinates
\begin{align}
&\int_{\R^N}\frac{1}{|z|^N}\chi_{[\e,R_{\e,\omega})}(|z|)dz=
\int_{B_{R_{\e,\omega}}(0)\setminus B_\e(0)}\frac{1}{|z|^N}dz=\int_\e^{R_{\e,\omega}}\left(\int_{\partial B_r(0)}\frac{1}{|z|^N}d\mathcal{H}^{N-1}(z)\right)dr\nonumber
\\
&=\int_\e^{R_{\e,\omega}}\frac{1}{r^N}r^{N-1}\mathcal{H}^{N-1}\left(S^{N-1}\right)dr=\mathcal{H}^{N-1}\left(S^{N-1}\right)\left(\ln R_{\e,\omega}-\ln \e\right)=\mathcal{H}^{N-1}\left(S^{N-1}\right)\left(|\ln \e|-|\ln R_{\e,\omega}|\right).
\end{align}
Note that, since $\e\in (0,1/e)$, then $\e,R_{\e,\omega}<1$, and therefore $-\ln\e=|\ln\e| $ and $\ln R_{\e,\omega}=-|\ln R_{\e,\omega}|$, so $\ln R_{\e,\omega}-\ln \e=|\ln \e|-|\ln R_{\e,\omega}|$. Thus, $\int_{\R^N}\rho_{\e,\omega}(|z|)dz=1$. The logarithmic kernel satisfies the compact support property:
for every $r\in (0,\infty)$ let $\delta_r:=e^{-\frac{1}{r^{1/\omega}}}$. Note that if $\e\in (0,\delta_r)$, then $R_{\e,\omega}<r$, so $\supp(\rho_{\e,\omega})\subset B_{R_{\e,\omega}}(0)\subset B_{r}(0)$, where $R_{\e,\omega}:=\frac{1}{|\ln\e|^{\omega}}$.
\\
2. By polar coordinates
\begin{multline}
\int_{\R^N}\frac{\rho_{\e,\omega}(|z|)}{|z|^\alpha}dz=
\frac{1}{\mathcal{H}^{N-1}\left(S^{N-1}\right)\left(|\ln \e|-|\ln R_{\e,\omega}|\right)}\int_{B_{R_{\e,\omega}}(0)\setminus B_\e(0)}\frac{1}{|z|^{N+\alpha}}dz
\\
=\frac{1}{\mathcal{H}^{N-1}\left(S^{N-1}\right)\left(|\ln \e|-|\ln R_{\e,\omega}|\right)}\int^{R_{\e,\omega}}_{\e}\frac{1}{r^{N+\alpha}}r^{N-1}\mathcal{H}^{N-1}\left(S^{N-1}\right)dr
\\
=\frac{1}{\left(|\ln \e|-|\ln R_{\e,\omega}|\right)}\frac{1}{\alpha}\left(\frac{1}{\e^\alpha}-\frac{1}{R_{\e,\omega}^\alpha}\right)
=\frac{1}{\alpha\e^\alpha}\left(\frac{1-\frac{\e^\alpha}{R_{\e,\omega}^\alpha}}{|\ln \e|-|\ln R_{\e,\omega}|}\right)
=\frac{1}{\alpha\e^\alpha}\left(\frac{1-\e^\alpha|\ln\e|^{\alpha\omega}}{|\ln\e|-\omega\ln|\ln\e|}\right).
\end{multline}
Hence,
\begin{equation}
\lim_{\e\to 0^+}\e^\alpha\int_{\R^N}\frac{\rho_{\e,\omega}(|z|)}{|z|^\alpha}dz
=\frac{1}{\alpha}\lim_{\e\to 0^+}\frac{1-\e^\alpha|\ln\e|^{\alpha\omega}}{|\ln\e|-\omega\ln|\ln\e|}
=\frac{1}{\alpha}\lim_{\e\to 0^+}\frac{\frac{1}{|\ln\e|}-\left(\e|\ln\e|^{\omega}\right)^\alpha\frac{1}{|\ln\e|}}{1-\omega\frac{\ln|\ln\e|}{|\ln\e|}}
=\frac{0}{1}=0.
\end{equation}
\end{proof}
\begin{definition}(The Trivial Kernel)
\label{def:definition of trivial kernel}

Let us define the {\it $N$-dimensional trivial kernel}, or just {\it trivial kernel}, to be
\begin{equation}
\tilde{\rho}_\e(r):=
\begin{cases}
\frac{1}{\e^N \Leb^{N}\left(B_1(0)\right)}& if \quad 0<r<\e\\
0 & if \quad r\geq \e
\end{cases},\quad \e\in (0,\infty).
\end{equation}
\end{definition}
\begin{remark}
\label{rem:compact support property for trivial kernel}
Notice that the trivial kernel is a kernel. Moreover, it satisfies the compact support property: for every $r\in (0,\infty)$, let $\delta_r:=r$. Thus, if $\e\in (0,\delta_r)$, then $\supp(\tilde{\rho}_\e)\subset B_{\e}(0)\subset B_{r}(0)$.
\end{remark}

\begin{definition}($\sigma$-Approximating Kernels)

For every number $\sigma\in (0,\infty)$, the {\it $N$-dimensional $\sigma$-approximating kernel} is defined to be
\begin{equation}
\label{eq:sigma approximating kernel}
\rho^\sigma_\e(r):=\frac{1}{2\sigma\mathcal{H}^{N-1}(S^{N-1})r^{N-1}}\chi_{[\e-\sigma,\e+\sigma]}(r),\quad \rho^\sigma_\e:(0,\infty)\to [0,\infty).
\end{equation}
\end{definition}

\begin{remark}($\sigma$-Approximating Kernels Give us Kernels)
\label{rem:sigma-approximating kernels give us kernels}

Note that $\sigma$-approximating kernels are not kernels because they lack the decreasing support property (see Definition \ref{def:decreasing support property}). However, if we select a number $\sigma_\e \in (0,\e)$ for every $\e \in (0,\infty)$, then the family $\{\rho^{\sigma_\e}_\e\}_{\e\in(0,\infty)}$ possesses the compact support property, and in particular, it satisfies the decreasing support property. By employing polar coordinates, we find that $\int_{\mathbb{R}^N}\rho^\sigma_\e(|z|)dz=1$ for every choice of $\e$ and $\sigma$ in $(0,\infty)$ with $\sigma<\e$. Therefore, $\{\rho^{\sigma_\e}_\e\}_{\e\in(0,\infty)}$ is a kernel, as defined in Definition \ref{def:kernel}.
\end{remark}

\section{Variations and Besov constants}
In this section, we introduce the notion of $(r,q)$-Variation (see Definition \ref{def:(r,q)-variation}). We prove that $(r,q)$-variations control Besov Constants (see Lemma \ref{lem:sandwich lemma}). Furthermore, we demonstrate that $(r,q)$-variation can be represented as a Besov constant (see Corollary \ref{cor:Representability of Variations as Besov Constants}). Additionally, we establish the continuity of Variations and Besov constants with respect to convergence in Besov Space (see Lemma \ref{lem:Continuity of (r,q)-variation in Besov spaces Brq}).

\begin{definition}($(r,q)$-Variation and Directional $(r,q)$-Variation)
\label{def:(r,q)-variation}

Let $r,q\in (0,\infty)$, and $u:\R^N\to \R^d$ be an $\mathcal{L}^N$-measurable function. Suppose $E\subset\R^N$ is an $\mathcal{L}^N$-measurable set, and let $n\in S^{N-1}$ be a direction. Then, the \textit{$(r,q)$ upper variation of $u$ in $E$ in the direction $n$} is defined by
\begin{equation}
(r,q)-\overline{V}(u,E,n):=\limsup_{\e\to 0^+}\int_E\chi_{E}(x+\e n)\frac{|u(x+\e n)-u(x)|^q}{\e^{rq}}dx.
\end{equation}
Similarly, replacing the $\limsup$ by the $\liminf$, we define the \textit{$(r,q)$ lower variation of $u$ in $E$ in the direction $n$} and denote it by $(r,q)-\underline{V}(u,E,n)$. If the limit exists, we denote it by $(r,q)-V(u,E,n)$, and we call it the \textit{$(r,q)$ variation of $u$ in $E$ in the direction $n$}.

The \textit{$(r,q)$ upper variation of $u$ in $E$} is defined by
\begin{equation}
(r,q)-\overline{V}(u,E):=\limsup_{\e\to 0^+}\int_{S^{N-1}}\int_E\chi_{E}(x+\e n)\frac{|u(x+\e n)-u(x)|^q}{\e^{rq}}dxd\mathcal{H}^{N-1}(n).
\end{equation}
Similarly, replacing the $\limsup$ by the $\liminf$, we define the \textit{$(r,q)$ lower variation of $u$ in $E$} and denote it by $(r,q)-\underline{V}(u,E)$. If the limit exists, we denote it by $(r,q)-V(u,E)$, and we call it the \textit{$(r,q)$ variation of $u$ in $E$}. We also define the notions of $(r,q)$ lower (upper) essential variation of $u$ in $E$, replacing the lower (upper) limit by the essential lower (upper) limit.
\end{definition}

\begin{definition}(Besov Constants)
\label{def:Besov constants}

Let $r,q\in (0,\infty)$, and $u:\R^N\to \R^d$ be an $\mathcal{L}^N$-measurable function. Suppose $E\subset\R^N$ is an $\mathcal{L}^N$-measurable set, and let $\{\rho_\e\}_{\e\in (0,a)}$ be a kernel for some $a\in(0,\infty]$. The {\it upper infinitesimal $(r,q)$ Besov constant of $u$ in $E$ with respect to the kernel $\rho_\e$} is defined as the quantity:
\begin{equation}
\label{eq:upper Besov constant}
\limsup_{\e\to 0^+}\int_{E}\int_{E}\rho_\e(|x-y|)\frac{|u(x)-u(y)|^q}{|x-y|^{rq}}dydx.
\end{equation}
Similarly, replacing the $\limsup$ by the $\liminf$, we define the {\it lower infinitesimal $(r,q)$ Besov constant of $u$ in $E$ with respect to the kernel $\rho_\e$}. If the limit exists, we refer to it as the {\it infinitesimal $(r,q)$ Besov constant of $u$ in $E$ with respect to the kernel $\rho_\e$}.
\end{definition}

\begin{remark}(The Upper Infinitesimal $B^{r,q}$-seminorm is a Besov Constant)
\label{rem:The Upper Infinitesimal Brq-seminorm is a Besov Constant}

Note that if we select the trivial kernel as defined in Definition \ref{def:definition of trivial kernel} in \eqref{eq:upper Besov constant}, multiply the result by $\mathcal{L}^N(B_1(0))$, and then take the result to the power of $\frac{1}{q}$, we obtain the upper infinitesimal $B^{r,q}$-seminorm as defined in \ref{def:Brq seminorms}.
\end{remark}

\begin{remark}(Variations of $W^{1,q}$, $BV$ and $B^r_{q,\infty}$)

From the $BBM$ formula, for an open and bounded set $\Omega \subset \mathbb{R}^N$ with a Lipschitz boundary, where $1 < q < \infty$ and $u \in W^{1,q}(\Omega)$, we have
\begin{equation}
(1,q)-V(u,\Omega)=C_{q,N}\|\nabla u\|^q_{L^q(\Omega)};
\end{equation}
for $u \in BV(\Omega)$, we have
\begin{equation}
\label{eq:ineqality6}
(1,1)-V(u,\Omega)=C_{1,N}\|D u\|(\Omega),
\end{equation}
where $C_{q,N}:=\int_{S^{N-1}}|z_1|^qd\mathcal{H}^{N-1}(n)$ for every $q \geq 1$. For proof of this result see \cite{Poliakovsky}.

For $r \in (0,1)$ and $q \in [1,\infty)$, we observe from Sandwich Lemma \ref{lem:sandwich lemma} and Theorem \ref{thm:characterization of Besov functions via double integral} that the finiteness of the upper variation $(r,q)-\overline{V}(u,\mathbb{R}^N)$ of $u$ together with $u \in L^q(\mathbb{R}^N,\mathbb{R}^d)$ is equivalent to $u \in B^r_{q,\infty}(\mathbb{R}^N,\mathbb{R}^d)$.
\end{remark}

\begin{lemma}(The Sandwich Lemma with Variations and Besov Constants Included)
\label{lem:sandwich lemma}

Let $E\subset\R^N$ be an $\mathcal{L}^N$-measurable set and
$u:E\to \R^d$ be an $\mathcal{L}^N$-measurable function. Let $a\in(0,\infty]$ and let $\rho_\e:(0,\infty)\to [0,\infty),\e\in (0,a),$
be a kernel, and $\alpha,q\in (0,\infty)$. Assume that at least one of the following three assumptions holds:
\begin{enumerate}

\item
\label{item:assumption about finitness of variation}
$\esssup_{\e\in(0,\infty)}\int_{S^{N-1}}\int_E\chi_{E}(x+\e n)\frac{|u(x+\e n)-u(x)|^q}{\e^\alpha}dxd\mathcal{H}^{N-1}(n)<\infty$;

\item $u\in L^q(\R^N,\R^d)$;

\item The kernel $\{\rho_\e\}_{\e\in(0,a)}$ has the compact support property as defined in Definition \ref{def:compact support property}.
\end{enumerate}
Then,
\begin{multline}
\label{eq:equation19}
\liminf_{\e\to 0^+}\fint_{S^{N-1}}\int_E\chi_{E}(x+\e n)\frac{|u(x+\e n)-u(x)|^q}{\e^\alpha}dxd\mathcal{H}^{N-1}(n)
\\
\leq \essliminf_{\e\to 0^+}\fint_{S^{N-1}}\int_E\chi_{E}(x+\e n)\frac{|u(x+\e n)-u(x)|^q}{\e^\alpha}dxd\mathcal{H}^{N-1}(n)
\\
\leq \liminf_{\e\to 0^+}\int_{E}\int_{E}\rho_\e(|x-y|)\frac{|u(x)-u(y)|^q}{|x-y|^\alpha}dydx\leq \limsup_{\e\to 0^+}\int_{E}\int_{E}\rho_\e(|x-y|)\frac{|u(x)-u(y)|^q}{|x-y|^\alpha}dydx
\\
\leq\esslimsup_{\e\to 0^+}\fint_{S^{N-1}}\int_E\chi_{E}(x+\e n)\frac{|u(x+\e n)-u(x)|^q}{\e^\alpha}dxd\mathcal{H}^{N-1}(n)
\\
\leq\limsup_{\e\to 0^+}\fint_{S^{N-1}}\int_E\chi_{E}(x+\e n)\frac{|u(x+\e n)-u(x)|^q}{\e^\alpha}dxd\mathcal{H}^{N-1}(n).
\end{multline}
In particular, for $r\in (0,\infty)$ and $\alpha=rq$, we get \eqref{eq:equation19} for $(r,q)$ variations and Besov constants.
\end{lemma}

\begin{proof}
By using polar coordinates, we get for every $\delta\in (0,\infty)$
\begin{multline}
\label{eq:equation55}
\int_{E}\int_{E}\rho_\e(|x-y|)\frac{|u(x)-u(y)|^q}{|x-y|^\alpha}dydx=\int_{E}\left(\int_{\R^N}\chi_E(y)\rho_\e(|x-y|)\frac{|u(x)-u(y)|^q}{|x-y|^\alpha}dy\right)dx
\\
=\int_{E}\left(\int_{\R^N}\chi_E(x+z)\rho_\e(|z|)\frac{|u(x)-u(x+z)|^q}{|z|^\alpha}dz\right)dx
\\
=\int_{\R^N}\left(\int_{E}\chi_E(x+z)\rho_\e(|z|)\frac{|u(x)-u(x+z)|^q}{|z|^\alpha}dx\right)dz
\\
=\int_0^\infty \int_{\partial B_t(0)}\left(\int_{E}\chi_E(x+z)\rho_\e(|z|)\frac{|u(x)-u(x+z)|^q}{|z|^\alpha}dx\right)d\mathcal{H}^{N-1}(z)dt
\\
=\int_0^\infty \int_{S^{N-1}}t^{N-1}\left(\int_{E}\chi_E(x+tn)\rho_\e(t)\frac{|u(x)-u(x+tn)|^q}{t^\alpha}dx\right)d\mathcal{H}^{N-1}(n)dt
\\
=\int_0^\infty t^{N-1}\rho_\e(t)\left(\int_{S^{N-1}}\int_{E}\chi_E(x+tn)\frac{|u(x)-u(x+tn)|^q}{t^\alpha}dxd\mathcal{H}^{N-1}(n)\right)dt
\\
=\int_0^\delta t^{N-1}\rho_\e(t)V(t)dt
+\int_\delta^\infty t^{N-1}\rho_\e(t)V(t)dt.
\end{multline}
In formula \eqref{eq:equation55} we denote
\begin{equation}
V(t):=\int_{S^{N-1}}\int_{E}\chi_E(x+tn)\frac{|u(x)-u(x+tn)|^q}{t^\alpha}dxd\mathcal{H}^{N-1}(n).
\end{equation}
By polar coordinates we see that
\begin{equation}
\label{eq:calculating the integral of the kernel}
\int_{\R^N}\rho_\e(|z|)dz=1\quad \Longrightarrow \quad \frac{1}{\mathcal{H}^{N-1}(S^{N-1})}=\int_0^\infty t^{N-1}\rho_\e(t)dt.
\end{equation}
Since $\{\rho_\e\}_{\e\in (0,a)}$ is a kernel, then it has the decreasing support property (see Definition \ref{def:decreasing support property}). Therefore, for every $\delta>0$, we get $\lim_{\e\to 0^+}\int^\infty_\delta t^{N-1}\rho_\e(t)dt=0$, and by \eqref{eq:calculating the integral of the kernel} we obtain
\begin{equation}
\frac{1}{\mathcal{H}^{N-1}(S^{N-1})}=\lim_{\e\to 0^+}\left\{\int_0^\delta t^{N-1}\rho_\e(t)dt+\int_\delta^\infty t^{N-1}\rho_\e(t)dt\right\}=\lim_{\e\to 0^+}\int_0^\delta t^{N-1}\rho_\e(t)dt.
\end{equation}

By equation \eqref{eq:equation55} we obtain
\begin{multline}
\label{eq:estimate including H}
\esssup_{t\in(0,\delta)}V(t)\int_0^\delta t^{N-1}\rho_\e(t)dt
+\int_\delta^\infty t^{N-1}\rho_\e(t)V(t)dt
\\
\geq \int_{E}\int_{E}\rho_\e(|x-y|)\frac{|u(x)-u(y)|^q}{|x-y|^\alpha}dydx
\\
\geq \essinf_{t\in(0,\delta)}V(t)\int_0^\delta t^{N-1}\rho_\e(t)dt
+\int_\delta^\infty t^{N-1}\rho_\e(t)V(t)dt.
\end{multline}
If $\esssup_{t\in(0,\infty)}V(t)<\infty$, then we get
$\lim_{\e\to 0^+}\int_\delta^\infty t^{N-1}\rho_\e(t)V(t)dt=0$. If $u\in L^q(\R^N,\R^d)$, then
\begin{equation}
\sup_{t\in(\delta,\infty)}V(t)\leq \frac{2^q\mathcal{H}^{N-1}(S^{N-1})}{\delta^{\alpha}}\|u\|^q_{L^q(\R^N,\R^d)}<\infty.
\end{equation}
So we get again that $\lim_{\e\to 0^+}\int_\delta^\infty t^{N-1}\rho_\e(t)V(t)dt=0$. Therefore, in both cases, we obtain \eqref{eq:equation19} by first taking the liminf(limsup) as $\e\to 0^+$ and then the limit as $\delta\to 0^+$ in inequality \eqref{eq:estimate including H}.

In case $\{\rho_\e\}_{\e\in (0,a)}$ has the compact support property, for $r>0$ there exists $\delta_r$ such that for every  $\e\in (0,\delta_r)$ we obtain $\supp(\rho_\e)\subset(0,r)$, and by \eqref{eq:equation55} we get

\begin{multline}
\label{eq:estimate including H in compact case}
\int_{E}\int_{E}\rho_\e(|x-y|)\frac{|u(x)-u(y)|^q}{|x-y|^\alpha}dydx=\int_0^r t^{N-1}\rho_\e(t)V(t)dt
\\
\leq \left(\int_0^r t^{N-1}\rho_\e(t)dt\right)\esssup_{t\in(0,r)}V(t)=\frac{1}{\mathcal{H}^{N-1}(S^{N-1})}\esssup_{t\in(0,r)}V(t),
\end{multline}

\begin{multline}
\label{eq:estimate including H in compact case 1}
\int_{E}\int_{E}\rho_\e(|x-y|)\frac{|u(x)-u(y)|^q}{|x-y|^\alpha}dydx=\int_0^r t^{N-1}\rho_\e(t)V(t)dt
\\
\geq \left(\int_0^r t^{N-1}\rho_\e(t)dt\right)\essinf_{t\in(0,r)}V(t)=\frac{1}{\mathcal{H}^{N-1}(S^{N-1})}\essinf_{t\in(0,r)}V(t).
\end{multline}
Taking the upper limit as $\e\to 0^+$ and then the limit as $r \to 0^+$ in \eqref{eq:estimate including H in compact case}, we obtain the forth inequality in \eqref{eq:equation19}. Similarly, taking the lower limit as $\e\to 0^+$ and then the limit as $r \to 0^+$ in \eqref{eq:estimate including H in compact case 1}, we obtain the second inequality of \eqref{eq:equation19}.
\end{proof}

\begin{lemma}(Variations and Essential Variations)
\label{lem:variations and essential variations}

Let $q,\alpha\in (0,\infty)$, and let $u\in L^q(\R^N,\R^d)$. Assume that $E\subset \R^N$ is a Lebesgue measurable set such that for every $v\in \R^N$ we have $\mathcal{L}^N\left(E\cap (\partial E+v)\right)=0$. Then,
\begin{multline}
\liminf_{\e\to 0^+}\fint_{S^{N-1}}\int_E\chi_{E}(x+\e n)\frac{|u(x+\e n)-u(x)|^q}{\e^\alpha}dxd\mathcal{H}^{N-1}(n)
\\
=\essliminf_{\e\to 0^+}\fint_{S^{N-1}}\int_E\chi_{E}(x+\e n)\frac{|u(x+\e n)-u(x)|^q}{\e^\alpha}dxd\mathcal{H}^{N-1}(n)
\end{multline}
and
\begin{multline}
\limsup_{\e\to 0^+}\fint_{S^{N-1}}\int_E\chi_{E}(x+\e n)\frac{|u(x+\e n)-u(x)|^q}{\e^\alpha}dxd\mathcal{H}^{N-1}(n)
\\
=\esslimsup_{\e\to 0^+}\fint_{S^{N-1}}\int_E\chi_{E}(x+\e n)\frac{|u(x+\e n)-u(x)|^q}{\e^\alpha}dxd\mathcal{H}^{N-1}(n).
\end{multline}
In particular, for $r\in (0,\infty)$ and $\alpha=rq$, we get the result for $(r,q)$ essential variations and Besov constants.
\end{lemma}

\begin{proof}
Let us denote
\begin{equation}
\label{eq:definition of V}
V(t):=\int_{S^{N-1}}\int_{E}\chi_E(x+tn)\frac{|u(x)-u(x+tn)|^q}{t^\alpha}dxd\mathcal{H}^{N-1}(n),
\end{equation}
and
\begin{equation}
F(t):=\int_{S^{N-1}}\int_E\chi_{E}(x+tn)|u(x+tn)-u(x)|^qdxd\mathcal{H}^{N-1}(n),\quad F:\R \to \R.
\end{equation}
Note that $F(t)=t^\alpha V(t)$. We prove the continuity of $F$ in $\mathbb{R}$, and consequently establish the continuity of $V$ in $(0,\infty)$. Thus, every point in $(0,\infty)$ is a Lebesgue point of $V$. Therefore, by Proposition \ref{prop:extremal sets for essential infimum and supremum} and Corollary \ref{cor:existence of extremal sets for Lebesgue functions}, we conclude that
\begin{equation}
\liminf_{\e\to 0^+}V(\e)=\essliminf_{\e\to 0^+}V(\e),\quad \limsup_{\e\to 0^+}V(\e)=\esslimsup_{\e\to 0^+}V(\e).
\end{equation}

Let $t_0\in\R$ be any number, and let us show that $F$ is continuous at $t_0$. Note first that
\begin{multline}
\label{eq:estimate for integrand for proving continuity of F}
\Big|\chi_{E}(x+tn)|u(x+tn)-u(x)|^q-\chi_{E}(x+t_0n)|u(x+t_0n)-u(x)|^q\Big|
\\
\leq \chi_{E}(x+tn)\Big||u(x+tn)-u(x)|^q-|u(x+t_0n)-u(x)|^q\Big|
\\
+\left|\chi_{E}(x+tn)-\chi_{E}(x+t_0n)\right||u(x+t_0n)-u(x)|^q
\\
\leq \Big||u(x+tn)-u(x)|^q-|u(x+t_0n)-u(x)|^q\Big|
\\
+\chi_{(E-tn)\Delta (E-t_0n)}(x)|u(x+t_0n)-u(x)|^q.
\end{multline}
Therefore, by \eqref{eq:estimate for integrand for proving continuity of F}
\begin{multline}
\label{eq:estimate for F(t)-F(t0)}
|F(t)-F(t_0)|=\Bigg|\int_{S^{N-1}}\int_E\chi_{E}(x+tn)|u(x+tn)-u(x)|^qdxd\mathcal{H}^{N-1}(n)
\\
-\int_{S^{N-1}}\int_E\chi_{E}(x+t_0n)|u(x+t_0n)-u(x)|^qdxd\mathcal{H}^{N-1}(n)\Bigg|
\\
\leq \int_{S^{N-1}}\int_E\Big|\chi_{E}(x+tn)|u(x+tn)-u(x)|^q-\chi_{E}(x+t_0n)|u(x+t_0n)-u(x)|^q\Big|dxd\mathcal{H}^{N-1}(n)
\\
\leq \int_{S^{N-1}}\int_E\Big||u(x+tn)-u(x)|^q-|u(x+t_0n)-u(x)|^q\Big|dxd\mathcal{H}^{N-1}(n)
\\
+\int_{S^{N-1}}\int_E\chi_{(E-tn)\Delta (E-t_0n)}(x)|u(x+t_0n)-u(x)|^qdxd\mathcal{H}^{N-1}(n).
\end{multline}

By Dominated Convergence Theorem and continuity of translations in $L^q$ we obtain
\begin{multline}
\label{eq:limit estimate for F(t)-F(t0) 1}
\lim_{t\to t_0}\int_{S^{N-1}}\int_{E}\Big||u(x+tn)-u(x)|^q-|u(x+t_0n)-u(x)|^q\Big|dxd\mathcal{H}^{N-1}(n)
\\
=\int_{S^{N-1}}\left(\lim_{t\to t_0}\int_{E}\Big||u(x+tn)-u(x)|^q-|u(x+t_0n)-u(x)|^q\Big|dx\right)d\mathcal{H}^{N-1}(n)=0.
\end{multline}
We utilized the continuity of translations in $L^q$ as follows: since $u(\cdot + tn)$ converges to $u(\cdot + t_0n)$ in $L^q$ as $t$ tends to $t_0$, then $u(\cdot + tn) - u$ converges to $u(\cdot + t_0n) - u$ in $L^q$ as $t$ tends to $t_0$. Consequently, $|u(\cdot + tn) - u|$ converges to $|u(\cdot + t_0n) - u|$ in $L^q$ as $t$ tends to $t_0$, and thus $|u(\cdot + tn) - u|^q$ converges to $|u(\cdot + t_0n) - u|^q$ in $L^1$ as $t$ tends to $t_0$.

Let us define for every $\e\in (0,\infty)$ the $\e$-neighbourhood of $\partial E-t_0n$ by
\begin{equation}
E_\e:=\Set{x\in\R^N}[\dist(x,\partial E-t_0n)\leq \e].
\end{equation}
Note that $\cap_{\e\in(0,\infty)}E_\e=\partial E-t_0n$. Therefore, for every $\e\in (0,\infty)$, there exists $R(\e)\in (0,\infty)$ such that for every $t\in \mathbb{R}$ with $|t-t_0|<R(\e)$, we have $(E-tn)\Delta (E-t_0n)\subset E_\e$ and so
$\chi_{(E-tn)\Delta (E-t_0n)}(x)\leq \chi_{E_\e}(x)$ for every $x\in\R^N$. Therefore,
\begin{multline}
\label{eq:inequality for part including characteristic function under the sign of the integral}
\limsup_{t\to t_0}\int_{S^{N-1}}\int_E\chi_{(E-tn)\Delta (E-t_0n)}(x)|u(x+t_0n)-u(x)|^qdxd\mathcal{H}^{N-1}(n)
\\
\leq \int_{S^{N-1}}\int_E\chi_{E_\e}(x)|u(x+t_0n)-u(x)|^qdxd\mathcal{H}^{N-1}(n).
\end{multline}
Therefore, taking the limit as $\e\to 0^+$ in \eqref{eq:inequality for part including characteristic function under the sign of the integral}, we get by Dominated Convergence Theorem and the assumption about $E$
\begin{multline}
\label{eq:limit estimate for F(t)-F(t0) 2}
\limsup_{t\to t_0}\int_{S^{N-1}}\int_E\chi_{(E-tn)\Delta (E-t_0n)}(x)|u(x+t_0n)-u(x)|^qdxd\mathcal{H}^{N-1}(n)
\\
\leq \int_{S^{N-1}}\int_{E\cap \left(\cap_{\e>0}E_\e\right)}|u(x+t_0n)-u(x)|^qdxd\mathcal{H}^{N-1}(n)
\\
=\int_{S^{N-1}}\int_{E\cap\left(\partial E-t_0n\right)}|u(x+t_0n)-u(x)|^qdxd\mathcal{H}^{N-1}(n)=0.
\end{multline}
Using \eqref{eq:estimate for F(t)-F(t0)}, \eqref{eq:limit estimate for F(t)-F(t0) 1}, and \eqref{eq:limit estimate for F(t)-F(t0) 2}, we conclude the continuity of $F$ at $t_0\in\R$. It completes the proof.
\end{proof}

\begin{proposition}(Besov Constants and Essential Variations)
\label{prop:Besov constants and essential variations}

Let $q,\alpha\in (0,\infty)$, and let $u\in L^q(\R^N,\R^d)$. Assume that $E\subset \R^N$ is a Lebesgue measurable set. Then, there exists a kernel $\{\rho_\e\}_{\e\in(0,\infty)}$ such that
\begin{multline}
\label{eq:equation199111}
\liminf_{\e\to 0^+}\int_{E}\int_{E}\rho_\e(|x-y|)\frac{|u(x)-u(y)|^q}{|x-y|^\alpha}dydx=\essliminf_{\e\to 0^+}\int_{E}\int_{E}\rho_\e(|x-y|)\frac{|u(x)-u(y)|^q}{|x-y|^\alpha}dydx
\\
=\essliminf_{\e\to 0^+}\fint_{S^{N-1}}\int_E\chi_{E}(x+\e n)\frac{|u(x+\e n)-u(x)|^q}{\e^\alpha}dxd\mathcal{H}^{N-1}(n),
\end{multline}
and
\begin{multline}
\label{eq:equation199222}
\limsup_{\e\to 0^+}\int_{E}\int_{E}\rho_\e(|x-y|)\frac{|u(x)-u(y)|^q}{|x-y|^\alpha}dydx=\esslimsup_{\e\to 0^+}\int_{E}\int_{E}\rho_\e(|x-y|)\frac{|u(x)-u(y)|^q}{|x-y|^\alpha}dydx
\\
=\esslimsup_{\e\to 0^+}\fint_{S^{N-1}}\int_E\chi_{E}(x+\e n)\frac{|u(x+\e n)-u(x)|^q}{\e^\alpha}dxd\mathcal{H}^{N-1}(n).
\end{multline}
In particular, for $r\in (0,\infty)$ and $\alpha=rq$, we get the result for $(r,q)$ variations and Besov constants.
\end{proposition}

\begin{proof}
For every $\e\in (0,\infty)$ and $\sigma\in (0,\e)$ let $\rho_\e^{\sigma}$ as in \eqref{eq:sigma approximating kernel}. By \eqref{eq:equation55}, we get
\begin{multline}
\label{eq:equation5555}
\int_{E}\int_{E}\rho^\sigma_\e(|x-y|)\frac{|u(x)-u(y)|^q}{|x-y|^\alpha}dydx
\\
=\int_0^\infty t^{N-1}\rho^\sigma_\e(t)\left(\int_{S^{N-1}}\int_{E}\chi_E(x+tn)\frac{|u(x)-u(x+tn)|^q}{t^\alpha}dxd\mathcal{H}^{N-1}(n)\right)dt
\\
=\frac{1}{2\sigma\mathcal{H}^{N-1}(S^{N-1})}\int_{\e-\sigma}^{\e+\sigma}V(t)dt,
\end{multline}
where
\begin{equation}
V(t):=\int_{S^{N-1}}\int_{E}\chi_E(x+tn)\frac{|u(x)-u(x+tn)|^q}{t^\alpha}dxd\mathcal{H}^{N-1}(n).
\end{equation}

Since $u\in L^q(\R^N,\R^d)$, the function $V$ is locally integrable in $(0,\infty)$, so almost every point in $(0,\infty)$ is a Lebesgue point of $V$. Let $\e\in (0,\infty)$ be a Lebesgue point of $V$. There exists $0<\sigma_\e<\e$ such that $\left|\frac{1}{2\sigma_\e}\int_{\e-\sigma_\e}^{\e+\sigma_\e}V(t)dt-V(\e)\right|<\e$. Therefore,
\begin{multline}
\label{eq:equation555551}
\int_{E}\int_{E}\rho^{\sigma_\e}_\e(|x-y|)\frac{|u(x)-u(y)|^q}{|x-y|^\alpha}dydx
=\frac{1}{2{\sigma_\e}\mathcal{H}^{N-1}(S^{N-1})}\int_{\e-{\sigma_\e}}^{\e+{\sigma_\e}}V(t)dt
\\
=\frac{1}{\mathcal{H}^{N-1}(S^{N-1})}V(\e)+\frac{1}{\mathcal{H}^{N-1}(S^{N-1})}\left(\frac{1}{2{\sigma_\e}}\int_{\e-{\sigma_\e}}^{\e+{\sigma_\e}}V(t)dt-V(\e)\right).
\end{multline}
By taking the lower limit in \eqref{eq:equation555551} as $\e\to 0^+$, with $\e$ being a Lebesgue point of $V$, we derive the second equation in \eqref{eq:equation199111} using Proposition \ref{prop:extremal sets for essential infimum and supremum} and Corollary \ref{cor:existence of extremal sets for Lebesgue functions}. Similarly, by taking the upper limit in \eqref{eq:equation555551} as $\e\to 0^+$, with $\e$ being a Lebesgue point of $V$, we obtain the second equation in \eqref{eq:equation199222}. Note that $\{\rho^{\sigma_\e}_\e\}_{\e\in(0,\infty)}$ is a kernel as was explained in Remark \ref{rem:sigma-approximating kernels give us kernels}.

By the definition of $\essliminf$ and the second equation of \eqref{eq:equation199111} we obtain
\begin{equation}
\label{eq:lower Besov constant is less than the essential lower variation}
\liminf_{\e\to 0^+}\int_{E}\int_{E}\rho^{\sigma_\e}_\e(|x-y|)\frac{|u(x)-u(y)|^q}{|x-y|^\alpha}dydx
\leq \essliminf_{\e\to 0^+}\frac{1}{\mathcal{H}^{N-1}(S^{N-1})}V(\e).
\end{equation}
By Lemma \ref{lem:sandwich lemma}, we get
\begin{equation}
\label{eq:lower Besov constant is more than the essential lower variation}
\liminf_{\e\to 0^+}\int_{E}\int_{E}\rho^{\sigma_\e}_\e(|x-y|)\frac{|u(x)-u(y)|^q}{|x-y|^\alpha}dydx
\geq \essliminf_{\e\to 0^+}\frac{1}{\mathcal{H}^{N-1}(S^{N-1})}V(\e).
\end{equation}
We get the first equation of \eqref{eq:equation199111} by \eqref{eq:lower Besov constant is less than the essential lower variation} and \eqref{eq:lower Besov constant is more than the essential lower variation}. We get the first equation of \eqref{eq:equation199222} in a similar way.
\end{proof}

\begin{corollary}(Representability of Variations as Besov Constants)
\label{cor:Representability of Variations as Besov Constants}

Let $q,\alpha\in (0,\infty)$, and let $u\in L^q(\R^N,\R^d)$. Assume that $E\subset \R^N$ is a Lebesgue measurable set such that for every $v\in \R^N$ we have $\mathcal{L}^N\left(E\cap (\partial E+v)\right)=0$. Then, there exists a kernel $\{\rho_\e\}_{\e\in(0,\infty)}$ such that
\begin{multline}
\label{eq:lower Besov constant equals to lower variation}
\liminf_{\e\to 0^+}\int_{E}\int_{E}\rho_\e(|x-y|)\frac{|u(x)-u(y)|^q}{|x-y|^\alpha}dydx
\\
=\liminf_{\e\to 0^+}\fint_{S^{N-1}}\int_E\chi_{E}(x+\e n)\frac{|u(x+\e n)-u(x)|^q}{\e^\alpha}dxd\mathcal{H}^{N-1}(n),
\end{multline}
and
\begin{multline}
\label{eq:upper Besov constant equals to upper variation}
\limsup_{\e\to 0^+}\int_{E}\int_{E}\rho_\e(|x-y|)\frac{|u(x)-u(y)|^q}{|x-y|^\alpha}dydx
\\
=\limsup_{\e\to 0^+}\fint_{S^{N-1}}\int_E\chi_{E}(x+\e n)\frac{|u(x+\e n)-u(x)|^q}{\e^\alpha}dxd\mathcal{H}^{N-1}(n).
\end{multline}
In particular, for $r\in (0,\infty)$ and $\alpha=rq$, we get the result for $(r,q)$ variations and Besov constants.
\end{corollary}

\begin{proof}
Formulas \eqref{eq:lower Besov constant equals to lower variation} and \eqref{eq:upper Besov constant equals to upper variation} follow immediately from Lemma \ref{lem:variations and essential variations} and Proposition \ref{prop:Besov constants and essential variations}.
\end{proof}

\begin{lemma}(Continuity of Variations and Besov Constants in Besov Spaces $B^r_{q,\infty}$)
\label{lem:Continuity of (r,q)-variation in Besov spaces Brq}

Let $q\in[1,\infty)$, $r\in (0,1)$, and $E\subset\R^N$ be an $\mathcal{L}^N$-measurable set. Consider a sequence $\{u_k\}_{k=1}^\infty\subset B^r_{q,\infty}(\R^N,\R^d)$ such that $u_k$ converges to $u$ in $B^r_{q,\infty}(\R^N,\R^d)$. Then,
\begin{enumerate}
\item for every $n\in\R^N$, we have
\begin{multline}
\label{eq: continuity of upper surface Besov term}
\lim_{k\to\infty}\left(\limsup_{\e\to 0^+}\int_E\chi_{E}(x+\e n)\frac{|u_k(x+\e n)-u_k(x)|^q}{\e^{rq}}dx\right)
\\
=\limsup_{\e\to 0^+}\int_E\chi_{E}(x+\e n)\frac{|u(x+\e n)-u(x)|^q}{\e^{rq}}dx,
\end{multline}
and a similar result also holds when replacing the $\limsup$ with the $\liminf$.

\item
\label{assertion 2 of Continuity of (r,q)-variation in Besov spaces Brq}
It follows that
\begin{multline}
\label{eq: continuity of upper surface Besov term1}
\lim_{k\to\infty}\left(\limsup_{\e\to 0^+}\int_{S^{N-1}}\int_E\chi_{E}(x+\e n)\frac{|u_k(x+\e n)-u_k(x)|^q}{\e^{rq}}dxd\mathcal{H}^{N-1}(n)\right)
\\
=\limsup_{\e\to 0^+}\int_{S^{N-1}}\int_E\chi_{E}(x+\e n)\frac{|u(x+\e n)-u(x)|^q}{\e^{rq}}dxd\mathcal{H}^{N-1}(n),
\end{multline}
and a similar result also holds when replacing the $\limsup$ with the $\liminf$.

\item
\label{assertion 3 of Lemma about continuity}
 Let $a\in(0,\infty]$. For every kernel $\rho_\e:(0,\infty)\to [0,\infty),\e\in (0,a)$, we get
\begin{multline}
\label{eq:continuity for upper Besov constant dependent on kernel}
\lim_{k\to\infty}\left(\limsup_{\e\to 0^+}\int_{E}\int_{E}\rho_\e(|x-y|)\frac{|u_k(x)-u_k(y)|^q}{|x-y|^{rq}}dydx\right)
\\
=\limsup_{\e\to 0^+}\int_{E}\int_{E}\rho_\e(|x-y|)\frac{|u(x)-u(y)|^q}{|x-y|^{rq}}dydx,
\end{multline}
and a similar result also holds when replacing the $\limsup$ with the $\liminf$.

\end{enumerate}
\end{lemma}
\begin{proof}
Let us prove assertion 1. Note that if $n=0$, then equation \eqref{eq: continuity of upper surface Besov term} in both $\liminf$ and $\limsup$ cases trivially holds. Assume $n\neq 0$.
Let us denote
\begin{equation}
I_{\e}(u_k,n)(x):=\chi_{E}(x+\e n)\frac{|u_k(x+\e n)-u_k(x)|^q}{\e^{rq}},
\end{equation}
and
\begin{equation}
I_{\e}(u,n)(x):=\chi_{E}(x+\e n)\frac{|u(x+\e n)-u(x)|^q}{\e^{rq}}.
\end{equation}

By Lemma \ref{lem:liminfsup lemma}, Minkowski's inequality, and the definition of the Besov seminorm $[\cdot]_{B^r_{q,\infty}}$, we obtain
\begin{multline}
\label{eq:estimate for continuity of upper surface Besov term}
\left|\limsup_{\e\to 0^+}\left(\int_EI_{\e}(u_k,n)(x)dx\right)^{\frac{1}{q}}-\limsup_{\e\to 0^+}\left(\int_EI_{\e}(u,n)(x)dx\right)^{\frac{1}{q}}\right|
\\
\leq \limsup_{\e\to 0^+}\left|\left(\int_EI_{\e}(u_k,n)(x)dx\right)^{\frac{1}{q}}-\left(\int_EI_{\e}(u,n)(x)dx\right)^{\frac{1}{q}}\right|
\\
=\limsup_{\e\to 0^+}\left|\left(\int_E\left[\left(I_{\e}(u_k,n)(x)\right)^{1/q}\right]^qdx\right)^{\frac{1}{q}}-\left(\int_E\left[\left(I_{\e}(u,n)(x)\right)^{1/q}\right]^qdx\right)^{\frac{1}{q}}\right|
\\
\leq \limsup_{\e\to 0^+}\left(\int_E\left|\left(I_{\e}(u_k,n)(x)\right)^{1/q}-\left(I_{\e}(u,n)(x)\right)^{1/q}\right|^qdx\right)^{\frac{1}{q}}
\\
=\limsup_{\e\to 0^+}\left(\int_E\left|\chi_{E}(x+\e n)\frac{|u_k(x+\e n)-u_k(x)|}{\e^{r}}-\chi_{E}(x+\e n)\frac{|u(x+\e n)-u(x)|}{\e^{r}}\right|^qdx\right)^{\frac{1}{q}}
\\
=\limsup_{\e\to 0^+}\left(\int_E\chi_{E}(x+\e n)\frac{\left||u_k(x+\e n)-u_k(x)|-|u(x+\e n)-u(x)|\right|^q}{\e^{rq}}dx\right)^{\frac{1}{q}}
\\
\leq \limsup_{\e\to 0^+}\left(\int_E\chi_{E}(x+\e n)\frac{\left|(u_k-u)(x+\e n)-(u_k-u)(x)\right|^q}{\e^{rq}}dx\right)^{\frac{1}{q}}
\leq |n|^r[u_k-u]_{B^r_{q,\infty}(\R^N,\R^d)}.
\end{multline}

We take the limit as $k\to \infty$ on both sides of \eqref{eq:estimate for continuity of upper surface Besov term} to obtain \eqref{eq: continuity of upper surface Besov term}. Similarly, we get
\begin{equation}
\label{eq:estimate for continuity of lower surface Besov term}
\left|\liminf_{\e\to 0^+}\left(\int_E I_{\e}(u_k,n)(x)dx\right)^{\frac{1}{q}}-\liminf_{\e\to 0^+}\left(\int_EI_{\e}(u,n)(x)dx\right)^{\frac{1}{q}}\right|\leq |n|^r[u_k-u]_{B^r_{q,\infty}(\R^N,\R^d)}.
\end{equation}

Assertion 2 of the Lemma is proven in the same way. By replacing the integral $\int_{E}(\cdot)dx$ with the integral $\int_{S^{N-1}}\int_E(\cdot)dxd\mathcal{H}^{N-1}(n)$ in \eqref{eq:estimate for continuity of upper surface Besov term} throughout, we obtain
\begin{multline}
\label{eq:inequality for assertion 2 with limsup}
\Bigg|\limsup_{\e\to 0^+}\left(\int_{S^{N-1}}\int_EI_{\e}(u_k,n)(x)dxd\mathcal{H}^{N-1}(n)\right)^{\frac{1}{q}}
-\limsup_{\e\to 0^+}\left(\int_{S^{N-1}}\int_EI_{\e}(u,n)(x)dxd\mathcal{H}^{N-1}(n)\right)^{\frac{1}{q}}\Bigg|
\\
\leq \left(\mathcal{H}^{N-1}(S^{N-1})\right)^{1/q}|n|^r[u_k-u]_{B^r_{q,\infty}(\R^N,\R^d)},
\end{multline}
and
\begin{multline}
\label{eq:inequality for assertion 2 with liminf}
\Bigg|\liminf_{\e\to 0^+}\left(\int_{S^{N-1}}\int_EI_{\e}(u_k,n)(x)dxd\mathcal{H}^{N-1}(n)\right)^{\frac{1}{q}}
-\liminf_{\e\to 0^+}\left(\int_{S^{N-1}}\int_EI_{\e}(u,n)(x)dxd\mathcal{H}^{N-1}(n)\right)^{\frac{1}{q}}\Bigg|
\\
\leq \left(\mathcal{H}^{N-1}(S^{N-1})\right)^{1/q}|n|^r[u_k-u]_{B^r_{q,\infty}(\R^N,\R^d)}.
\end{multline}
Taking the limit as $k$ tends to infinity in inequalities \eqref{eq:inequality for assertion 2 with limsup} and \eqref{eq:inequality for assertion 2 with liminf}, we get formula \eqref{eq: continuity of upper surface Besov term1} in both cases $\liminf$ and $\limsup$.

We prove assertion 3. Let us denote
\begin{equation}
B_{\e,u_k}(x,y):=\rho_\e(|x-y|)\frac{|u_k(x)-u_k(y)|^q}{|x-y|^{rq}},\quad B_{\e,u}(x,y):=\rho_\e(|x-y|)\frac{|u(x)-u(y)|^q}{|x-y|^{rq}}.
\end{equation}
As in \eqref{eq:estimate for continuity of upper surface Besov term}, for $n\neq0$, by Lemma \ref{lem:liminfsup lemma}, Minkowski's inequality, Sandwich Lemma with $\alpha=rq$ (Lemma \ref{lem:sandwich lemma}) and the definition of the Besov seminorm $[\cdot]_{B^r_{q,\infty}}$, we obtain
\begin{multline}
\label{eq:estimate for continuity of upper  Besov term dependent of gereral kernel}
\left|\limsup_{\e\to 0^+}\left(\int_{E}\int_{E}B_{\e,u_k}(x,y)dydx\right)^{\frac{1}{q}}-\limsup_{\e\to 0^+}\left(\int_{E}\int_{E}B_{\e,u}(x,y)dydx\right)^{\frac{1}{q}}\right|
\\
\leq \limsup_{\e\to 0^+}\Bigg|\left(\int_{E}\int_{E}\left[\left(B_{\e,u_k}(x,y)\right)^{\frac{1}{q}}\right]^qdydx\right)^{\frac{1}{q}}
-\left(\int_{E}\int_{E}\left[\left(B_{\e,u}(x,y)\right)^{\frac{1}{q}}\right]^qdydx\right)^{\frac{1}{q}}\Bigg|
\\
\leq \limsup_{\e\to 0^+}\left(\int_{E}\int_{E}\left|\left(B_{\e,u_k}(x,y)\right)^{\frac{1}{q}}-\left(B_{\e,u}(x,y)\right)^{\frac{1}{q}}\right|^qdydx\right)^{\frac{1}{q}}
\\
=\limsup_{\e\to 0^+}\left(\int_{E}\int_{E}\left|\left(\rho_\e(|x-y|)\right)^{\frac{1}{q}}\frac{|u_k(x)-u_k(y)|}{|x-y|^{r}}-\left(\rho_\e(|x-y|)\right)^{\frac{1}{q}}\frac{|u(x)-u(y)|}{|x-y|^{r}}\right|^qdydx\right)^{\frac{1}{q}}
\\
\leq\limsup_{\e\to 0^+}\left(\int_{E}\int_{E}\rho_\e(|x-y|)\frac{|(u_k-u)(x)-(u_k-u)(y)|^q}{|x-y|^{rq}}dydx\right)^{\frac{1}{q}}
\\
\leq \limsup_{\e\to 0^+}\left(\fint_{S^{N-1}}\int_E\chi_{E}(x+\e n)\frac{|(u_k-u)(x+\e n)-(u_k-u)(x)|^q}{\e^{rq}}dxd\mathcal{H}^{N-1}(n)\right)^{\frac{1}{q}}.
\\
\leq |n|^r[u_k-u]_{B^r_{q,\infty}(\R^N,\R^d)}.
\end{multline}
Taking the limit as $k\to \infty$ in the inequality \eqref{eq:estimate for continuity of upper  Besov term dependent of gereral kernel}, we obtain \eqref{eq:continuity for upper Besov constant dependent on kernel}. We get this result for $\liminf$ in a similar way.
\end{proof}

\begin{remark}
In Lemma \ref{lem:Continuity of (r,q)-variation in Besov spaces Brq}, we can utilize Corollary \ref{cor:Representability of Variations as Besov Constants} to derive assertion \ref{assertion 2 of Continuity of (r,q)-variation in Besov spaces Brq} from assertion \ref{assertion 3 of Lemma about continuity} in Lemma \ref{lem:Continuity of (r,q)-variation in Besov spaces Brq}, provided that we limit ourselves to sets $E$ satisfying the conditions of Corollary \ref{cor:Representability of Variations as Besov Constants}.
\end{remark}

\section{Equivalence Between Gagliardo Constants and Besov Constants}

In this section, we demonstrate that the upper and lower variations control Gagliardo constants (refer to Theorem \ref{thm:Sandwich inequality for Gagliardo constants}). Furthermore, we establish that, under certain conditions, Gagliardo constants and Besov constants are equivalent (see Theorem \ref{thm:connection between Gagliardo constant and Besov constant dependent on arbitrary kernel}). As a special case, we derive the equivalence between Gagliardo constants and infinitesimal $B^{r,q}$-seminorms (refer to Corollary \ref{cor:main result}).

\begin{corollary}(Besov Constants Bounded by Besov Seminorms)
\label{cor:Finiteness of infinitesimal double integral including kernel}

Let $q\in [1,\infty)$, $r\in(0,1)$ and $u\in B^r_{q,\infty}(\R^N,\R^d)$. Let $\rho_\e:(0,\infty)\to [0,\infty),\e\in (0,a),$ be a kernel for some $a\in(0,\infty)$. Then
\begin{align}
\limsup_{\e\to 0^+}\int_{\R^N}\int_{\R^N}\rho_\e(|x-y|)\frac{|u(x)-u(y)|^q}{|x-y|^{rq}}dydx\leq [u]^q_{B^{r}_{q,\infty}(\R^N,\R^d)}<\infty.
\end{align}
\end{corollary}
\begin{proof}
By Lemma \ref{lem:sandwich lemma} with $\alpha=rq$ and $E=\R^N$, Definition \ref{def:Besov seminorm} (definition of Besov seminorm) and the assumption $u\in B^r_{q,\infty}(\R^N,\R^d)$ we get
\begin{multline}
\limsup_{\e\to 0^+}\int_{\R^N}\int_{\R^N}\rho_\e(|x-y|)\frac{|u(x)-u(y)|^q}{|x-y|^{rq}}dydx
\\
\leq\frac{1}{\mathcal{H}^{N-1}(S^{N-1})}\limsup_{\e\to 0^+}\int_{S^{N-1}}\int_{\R^N}\frac{|u(x+\e n)-u(x)|^q}{\e^{rq}}dxd\mathcal{H}^{N-1}(n)
\leq [u]^q_{B^{r}_{q,\infty}(\R^N,\R^d)}<\infty.
\end{multline}
\end{proof}

\begin{lemma}(Approximation of Gagliardo Constants by Besov Constants through the Logarithmic Kernel)
\label{lem:J equals to Besov approximation with mollified u up to a small error}

Let
$1\leq q<\infty$, $r\in (0,1)$. Let $\omega\in(0,1)$ be such that $rq<{1/\omega}$. Let $u\in B^r_{q,\infty}(\R^N,\R^d)$, $E\subset\R^N$ be an $\mathcal{L}^N$-measurable set, and let $\eta$ be such that
\begin{equation}
\label{eq:assumptions about eta}
\eta\in  W^{1,1}\left(\R^N\right),\quad \int_{\R^N}|\nabla\eta(v)||v|^{rq}dv<\infty.
\end{equation}
Then for every $\e\in (0,1/e)$
it follows that
\begin{equation}
\frac{1}{|\ln\e|}\left[u_{\e}\right]^q_{W^{r,q}(E,\R^d)}=\mathcal{H}^{N-1}\left(S^{N-1}\right)\int_{E}\int_{E}\rho_{\e,\omega}(|x-y|)\frac{|u_\e(x)-u_\e(y)|^q}{|x-y|^{rq}}dydx+o_\e(1),
\end{equation}
where $u_\e(x):=\int_{\R^N}\eta(z)u(x-\e z)dz$, $\lim_{\e\to 0^+}o_\e(1)=0$ and $\rho_{\e,\omega}$ is the logarithmic kernel defined in Definition \ref{def:logarithmic kernel}.
\end{lemma}
\begin{proof}
Let $\e\in (0,1/e)$ be fixed. By definition of Gagliardo seminorm $\left[\cdot\right]_{W^{r,q}}$ (Definition \ref{def:Gagliardo seminorm}), change of variable formula, Fubini's theorem and additivity of integral
\begin{multline}
\label{eq:equation10}
\frac{1}{|\ln\e|}\left[u_{\e}\right]^q_{W^{r,q}(E,\R^d)}=
\frac{1}{|\ln\e|}\int_{E}\left(\int_{E}\frac{|u_\e(x)-u_\e(y)|^q}{|x-y|^{N+rq}}dy\right)dx
\\
=\frac{1}{|\ln\e|}\int_{\R^N}\left(\int_{\R^N}\frac{|u_\e(x)-u_\e(x+z)|^q}{|z|^{N+rq}}\chi_{E}(x+z)\chi_{E}(x)dz\right)dx
\\
=\frac{1}{|\ln\e|}\int_{\R^N}\left(\int_{\R^N}\frac{|u_\e(x)-u_\e(x+z)|^q}{|z|^{N+rq}}\chi_{E}(x+z)\chi_{E}(x)dx\right)dz
\\
=\frac{1}{|\ln\e|}\left\{\int_{\R^N\setminus B_{R_{\e,\omega}}(0)}g^\e(z)dz+\int_{ B_{R_{\e,\omega}}(0)\setminus B_\e(0)}g^\e(z)dz+\int_{B_\e(0)}g^\e(z)dz\right\},
\end{multline}
where we set
\begin{align}
\label{eq:equation1}
g^\e(z):&=\int_{\R^N}\frac{|u_\e(x)-u_\e(x+z)|^q}{|z|^{N+rq}}\chi_{E}(x+z)\chi_{E}(x)dx.
\end{align}
Using \eqref{eq:equation68} with $\gamma=R_{\e,\omega}:=|\ln\e|^{-\omega}$ we get
\begin{equation}
\label{eq:equation8}
\frac{1}{|\ln\e|}\int_{\R^N\setminus B_{R_{\e,\omega}}(0)}g^\e(z)dz\leq \|\eta\|^q_{L^1(\R^N)}2^q\|u\|^q_{L^q(\R^N,\R^d)}\frac{\mathcal{H}^{N-1}\left(S^{N-1}\right)}{|\ln\e|^{1-rq\omega}rq}=o_\e(1).
\end{equation}
Using \eqref{eq:equation72} we get
\begin{multline}
\label{eq:equation84}
\frac{1}{|\ln\e|}\int_{B_\e(0)}g^\e(z)dz
\\
\leq \|\nabla\eta\|^{q-1}_{L^1(\R^N,\R^N)}\left(\int_{\R^N}|\nabla\eta(v)|(|v|+2)^{rq}dv\right)[u]^q_{B^{r}_{q,\infty}(\R^N,\R^d)}\frac{\mathcal{H}^{N-1}\left(S^{N-1}\right)}{q-rq}\frac{1}{|\ln\e|}
=o_\e(1).
\end{multline}
Therefore, we obtain by \eqref{eq:equation10},\eqref{eq:equation8},\eqref{eq:equation84} and the definition of the logarithmic kernel $\rho_{\e,\omega}$
\begin{multline}
\frac{1}{|\ln\e|}\left[u_{\e}\right]^q_{W^{r,q}(E,\R^d)}
=\frac{1}{|\ln\e|}\int_{ B_{R_{\e,\omega}}(0)\setminus B_\e(0)}g^\e(z)dz+o_\e(1)
\\
=\frac{1}{|\ln\e|}\int_{ B_{R_{\e,\omega}}(0)\setminus B_\e(0)}\left(\int_{\R^N}\frac{|u_\e(x)-u_\e(x+z)|^q}{|z|^{N+rq}}\chi_{E}(x+z)\chi_{E}(x)dx\right)dz+o_\e(1)
\\
=\frac{1}{|\ln\e|}\int_{\R^N}\frac{\chi_{[\e,R_{\e,\omega})}(|z|)}{|z|^N}\left(\int_{\R^N}\frac{|u_\e(x)-u_\e(x+z)|^q}{|z|^{rq}}\chi_{E}(x+z)\chi_{E}(x)dx\right)dz+o_\e(1)
\\
=\left(1-\frac{|\ln R_{\e,\omega}|}{|\ln\e|}\right)\mathcal{H}^{N-1}\left(S^{N-1}\right)\int_{\R^N}\rho_{\e,\omega}(|z|)\left(\int_{\R^N}\frac{|u_\e(x)-u_\e(x+z)|^q}{|z|^{rq}}\chi_{E}(x+z)\chi_{E}(x)dx\right)dz+o_\e(1)
\\
=\mathcal{H}^{N-1}\left(S^{N-1}\right)\int_{E}\int_{E}\rho_{\e,\omega}(|x-y|)\frac{|u_\e(x)-u_\e(y)|^q}{|x-y|^{rq}}dxdy
\\
-\frac{|\ln R_{\e,\omega}|}{|\ln\e|}\mathcal{H}^{N-1}\left(S^{N-1}\right)\int_{\R^N}\rho_{\e,\omega}(|z|)\left(\int_{\R^N}\frac{|u_\e(x)-u_\e(x+z)|^q}{|z|^{rq}}\chi_{E}(x+z)\chi_{E}(x)dx\right)dz
+o_\e(1)
\\
=\mathcal{H}^{N-1}\left(S^{N-1}\right)\int_{E}\int_{E}\rho_{\e,\omega}(|x-y|)\frac{|u_\e(x)-u_\e(y)|^q}{|x-y|^{rq}}dxdy+o_\e(1).
\end{multline}
In the last equality we used \eqref{eq:equation42}, item 1 of Proposition \ref{prop:Properties of the logarithmic kernel}, \eqref{eq:equation21} and $u\in B^r_{q,\infty}(\R^N,\R^d)$ in order to get
\begin{multline}
\frac{|\ln R_{\e,\omega}|}{|\ln\e|}\mathcal{H}^{N-1}\left(S^{N-1}\right)\int_{\R^N}\rho_{\e,\omega}(|z|)\left(\int_{\R^N}\frac{|u_\e(x)-u_\e(x+z)|^q}{|z|^{rq}}\chi_{E}(x+z)\chi_{E}(x)dx\right)dz
\\
\leq \frac{|\ln R_{\e,\omega}|}{|\ln\e|}\mathcal{H}^{N-1}\left(S^{N-1}\right)\left(\int_{\R^N}|\eta(v)|dv\right)^{q} [u]^q_{B^{r}_{q,\infty}(\R^N,\R^d)}=o_\e(1).
\end{multline}
It completes the proof.
\end{proof}
\begin{lemma}(The $\eta$-Separating Lemma)
\label{lem:eta-separating lemma}

Assume
$q\in[1,\infty)$, $r\in(0,1)$ and
$u\in B^r_{q,\infty}(\R^N,\R^d)$.
Let $\eta\in L^1(\R^N)$ be such that
\begin{equation}
\label{eq:assumptions of eta1}
\int_{\R^N}|\eta(z)||z|^{rq}dz<\infty.
\end{equation}
Let $\{\rho_\e\}_{\e\in (0,a)},a\in(0,\infty]$, $\rho_\e:(0,\infty)\to [0,\infty)$, be a kernel such that
\begin{equation}
\label{eq:assumption on the family rho}
\lim_{\e\to 0^+}\e^{rq}\int_{\R^N}\frac{\rho_\e(|z|)}{|z|^{rq}}dz=0.
\end{equation}
Then for every $\mathcal{L}^N$-measurable set $E\subset \R^N$ we have
\begin{multline}
\label{eq:equation52}
\liminf_{\e\to 0^+}\int_{E}\int_{E}\rho_\e(|x-y|)\frac{\left|u_\e(x)-u_\e(y)\right|^q}{|x-y|^{rq}}dydx
\\
=\left|\int_{\R^N}\eta(z)dz\right|^q\liminf_{\e\to 0^+}\int_{E}\int_{E}\rho_\e(|x-y|)\frac{\left|u(x)-u(y)\right|^q}{|x-y|^{rq}}dydx,
\end{multline}
and
\begin{multline}
\label{eq:equation53}
\limsup_{\e\to 0^+}\int_{E}\int_{E}\rho_\e(|x-y|)\frac{\left|u_\e(x)-u_\e(y)\right|^q}{|x-y|^{rq}}dydx
\\
=\left|\int_{\R^N}\eta(z)dz\right|^q\limsup_{\e\to 0^+}\int_{E}\int_{E}\rho_\e(|x-y|)\frac{\left|u(x)-u(y)\right|^q}{|x-y|^{rq}}dydx.
\end{multline}
\end{lemma}
\begin{proof}
Let $0<\alpha<1$. It follows for $\mathcal{L}^N$-almost every $x,z\in \R^N$ that
\begin{multline}
\label{eq:equation93}
\left|u(x)-u(x+z)\right|^q
=\left|(u(x)-u_\e(x))+(u_\e(x)-u_\e(x+z))+(u_\e(x+z)-u(x+z))\right|^q
\\
\leq \left(|u(x)-u_\e(x)|+|u_\e(x)-u_\e(x+z)|+|u_\e(x+z)-u(x+z)|\right)^q
\\
\leq \frac{1}{\alpha^{q-1}}\left|u_\e(x)-u_\e(x+z)\right|^q
+\frac{1}{(1-\alpha)^{q-1}}\left(|u(x)-u_\e(x)|+|u_\e(x+z)-u(x+z)|\right)^q.
\end{multline}
In the last inequality we use the following convex inequality:
for numbers $A,B\geq 0$ and convex function $\Psi:[0,\infty)\to \R$ it follows that
\begin{equation}
\Psi(A+B)=\Psi\left(\alpha\frac{A}{\alpha}+(1-\alpha)\frac{B}{1-\alpha}\right)\leq \alpha\Psi\left(\frac{A}{\alpha}\right)+(1-\alpha)\Psi\left(\frac{B}{1-\alpha}\right).
\end{equation}
In the inequality \eqref{eq:equation93} we choose
\begin{equation}
A=|u_\e(x)-u_\e(x+z)|,\quad
B=|u(x)-u_\e(x)|+|u_\e(x+z)-u(x+z)|,\quad \Psi(r)=r^q.
\end{equation}
Therefore, by \eqref{eq:equation93}
\begin{multline}
\label{eq:equation46}
\int_{E}\int_{E}\rho_\e(|x-y|)\frac{\left|u(x)-u(y)\right|^q}{|x-y|^{rq}}dydx=\int_{\R^N}\int_{\R^N}\rho_\e(|x-y|)\frac{\left|u(x)-u(y)\right|^q}{|x-y|^{rq}}\chi_{E}(y)\chi_{E}(x)dydx
\\
=\int_{\R^N}\int_{\R^N}\rho_\e(|z|)\frac{\left|u(x)-u(x+z)\right|^q}{|z|^{rq}}\chi_{E}(x+z)\chi_{E}(x)dzdx
\\
\leq \frac{1}{\alpha^{q-1}}\int_{\R^N}\int_{\R^N}\rho_\e(|z|)\frac{\left|u_\e(x)-u_\e(x+z)\right|^q}{|z|^{rq}}\chi_{E}(x+z)\chi_{E}(x)dzdx
\\
+\frac{1}{(1-\alpha)^{q-1}} \int_{\R^N}\int_{\R^N}\rho_\e(|z|)\frac{\left(|u(x)-u_\e(x)|+|u_\e(x+z)-u(x+z)|\right)^q}{|z|^{rq}}dzdx
\\
=\frac{1}{\alpha^{q-1}}\int_{E}\int_{E}\rho_\e(|x-y|)\frac{\left|u_\e(x)-u_\e(y)\right|^q}{|x-y|^{rq}}dydx
\\
+\frac{1}{(1-\alpha)^{q-1}} \int_{\R^N}\int_{\R^N}\rho_\e(|z|)\frac{\left(|u(x)-u_\e(x)|+|u_\e(x+z)-u(x+z)|\right)^q}{|z|^{rq}}dzdx.
\end{multline}
Notice that
\begin{multline}
\label{eq:equation45}
\int_{\R^N}\int_{\R^N}\rho_\e(|z|)\frac{\left(|u(x)-u_\e(x)|+|u(x+z)-u_\e(x+z)|\right)^q}{|z|^{rq}}dzdx
\\
\leq 2^{q-1}\int_{\R^N}\int_{\R^N}\rho_\e(|z|)\frac{|u(x)-u_\e(x)|^q+|u(x+z)-u_\e(x+z)|^q}{|z|^{rq}}dzdx
\\
=2^q\int_{\R^N}\int_{\R^N}\rho_\e(|z|)\frac{|u(x)-u_\e(x)|^q}{|z|^{rq}}dzdx
=2^q\left(\int_{\R^N}\frac{\rho_\e(|z|)}{|z|^{rq}}dz\right)\|u-u_\e\|^q_{L^q(\R^N,\R^d)}.
\end{multline}
Assume for a moment that $\int_{\R^N}\eta(z)dz=1$. Then, by Hölder's inequality
\begin{multline}
\label{eq:equation44}
\|u-u_\e\|^q_{L^q(\R^N,\R^d)}=\int_{\R^N}|u(x)-u_\e(x)|^qdx
=\int_{\R^N}\left|\int_{\R^N}\eta(v)\left(u(x)-u(x-\e v)\right)dv\right|^qdx
\\
\leq \int_{\R^N}\left(\int_{\R^N}|\eta(v)|^{\frac{q-1}{q}}\left(|\eta(v)|^{\frac{1}{q}}\left|u(x)-u(x-\e v)\right|\right)dv\right)^qdx
\\
\leq \|\eta\|^{q-1}_{L^1(\R^N)}\int_{\R^N}\left(\int_{\R^N}|\eta(v)|\left|u(x)-u(x-\e v)\right|^qdv\right)dx
\\
= \|\eta\|^{q-1}_{L^1(\R^N)}\int_{\R^N}|\eta(v)|\left(\int_{\R^N}\left|u(x)-u(x-\e v)\right|^qdx\right)dv
\\
\leq \e^{rq}\|\eta\|^{q-1}_{L^1(\R^N)}[u]^q_{B^{r}_{q,\infty}(\R^N,\R^d)}\int_{\R^N}|\eta(v)||v|^{rq}dv.
\end{multline}
Hence, by \eqref{eq:equation46}, \eqref{eq:equation45} and \eqref{eq:equation44}
\begin{multline}
\label{eq:equation47}
\int_{E}\int_{E}\rho_\e(|x-y|)\frac{\left|u(x)-u(y)\right|^q}{|x-y|^{rq}}dydx
\leq \frac{1}{\alpha^{q-1}}\int_{E}\int_{E}\rho_\e(|x-y|)\frac{\left|u_\e(x)-u_\e(y)\right|^q}{|x-y|^{rq}}dydx
\\
+\frac{2^q}{(1-\alpha)^{q-1}}\left(\e^{rq}\int_{\R^N}\frac{\rho_\e(|z|)}{|z|^{rq}}dz\right)\frac{\|u-u_\e\|^q_{L^q(\R^N,\R^d)}}{\e^{rq}}
\leq \frac{1}{\alpha^{q-1}}\int_{E}\int_{E}\rho_\e(|x-y|)\frac{\left|u_\e(x)-u_\e(y)\right|^q}{|x-y|^{rq}}dydx
\\
+\frac{2^q}{(1-\alpha)^{q-1}}\left(\e^{rq}\int_{\R^N}\frac{\rho_\e(|z|)}{|z|^{rq}}dz\right)\|\eta\|^{q-1}_{L^1(\R^N)}[u]^q_{B^{r}_{q,\infty}(\R^N,\R^d)}\int_{\R^N}|\eta(v)||v|^{rq}dv.
\end{multline}
By \eqref{eq:assumptions of eta1}, \eqref{eq:assumption on the family rho}, $u\in B^r_{q,\infty}(\R^N,\R^d)$ and \eqref{eq:equation47}
we obtain
\begin{align}
\liminf_{\e\to 0^+}\int_{E}\int_{E}\rho_\e(|x-y|)\frac{\left|u(x)-u(y)\right|^q}{|x-y|^{rq}}dydx\leq \frac{1}{\alpha^{q-1}}\liminf_{\e\to 0^+}\int_{E}\int_{E}\rho_\e(|x-y|)\frac{\left|u_\e(x)-u_\e(y)\right|^q}{|x-y|^{rq}}dydx,
\end{align}
and
\begin{align}
\limsup_{\e\to 0^+}\int_{E}\int_{E}\rho_\e(|x-y|)\frac{\left|u(x)-u(y)\right|^q}{|x-y|^{rq}}dydx\leq \frac{1}{\alpha^{q-1}}\limsup_{\e\to 0^+}\int_{E}\int_{E}\rho_\e(|x-y|)\frac{\left|u_\e(x)-u_\e(y)\right|^q}{|x-y|^{rq}}dydx.
\end{align}
Taking the limit as
$\alpha\to 1^-$
we get
\begin{align}
\liminf_{\e\to 0^+}\int_{E}\int_{E}\rho_\e(|x-y|)\frac{\left|u(x)-u(y)\right|^q}{|x-y|^{rq}}dydx\leq\liminf_{\e\to 0^+}\int_{E}\int_{E}\rho_\e(|x-y|)\frac{\left|u_\e(x)-u_\e(y)\right|^q}{|x-y|^{rq}}dydx,
\end{align}
and
\begin{align}
\limsup_{\e\to 0^+}\int_{E}\int_{E}\rho_\e(|x-y|)\frac{\left|u(x)-u(y)\right|^q}{|x-y|^{rq}}dydx\leq \limsup_{\e\to 0^+}\int_{E}\int_{E}\rho_\e(|x-y|)\frac{\left|u_\e(x)-u_\e(y)\right|^q}{|x-y|^{rq}}dydx.
\end{align}
Replacing the roles of
$u_\e(x),u_\e(y)$
with
$u(x),u(y),$
respectively,
one can prove similarly (to the inequality \eqref{eq:equation47}) the inequality
\begin{multline}
\int_{E}\int_{E}\rho_\e(|x-y|)\frac{\left|u_\e(x)-u_\e(y)\right|^q}{|x-y|^{rq}}dydx
\leq \frac{1}{\alpha^{q-1}}\int_{E}\int_{E}\rho_\e(|x-y|)\frac{\left|u(x)-u(y)\right|^q}{|x-y|^{rq}}dydx
\\
+\frac{2^q}{(1-\alpha)^{q-1}}\left(\e^{rq}\int_{\R^N}\frac{\rho_\e(|z|)}{|z|^{rq}}dz\right)\|\eta\|^{q-1}_{L^1(\R^N)}[u]^q_{B^{r}_{q,\infty}(\R^N,\R^d)}\int_{\R^N}|\eta(v)||v|^{rq}dv
\end{multline}
in order to obtain
\begin{align}
\liminf_{\e\to 0^+}\int_{E}\int_{E}\rho_\e(|x-y|)\frac{\left|u_\e(x)-u_\e(y)\right|^q}{|x-y|^{rq}}dydx\leq \liminf_{\e\to 0^+}\int_{E}\int_{E}\rho_\e(|x-y|)\frac{\left|u(x)-u(y)\right|^q}{|x-y|^{rq}}dydx,
\end{align}
and
\begin{align}
\limsup_{\e\to 0^+}\int_{E}\int_{E}\rho_\e(|x-y|)\frac{\left|u_\e(x)-u_\e(y)\right|^q}{|x-y|^{rq}}dydx\leq\limsup_{\e\to 0^+}\int_{E}\int_{E}\rho_\e(|x-y|)\frac{\left|u(x)-u(y)\right|^q}{|x-y|^{rq}}dydx.
\end{align}

Assume now that $\int_{\R^N}\eta(z)dz\neq 0$. Replacing $\eta$ with $c\eta$, where $c:=\frac{1}{\int_{\R^N}\eta(z)dz}$, and using the homogeneity of the convolution $u*(c\eta)=c\left(u*\eta\right)$, one can get \eqref{eq:equation52} and \eqref{eq:equation53}. In case $\int_{\R^N}\eta(z)dz=0$, let us choose any $\eta_0\in C_c(\R^N)$ such that $\int_{\R^N}\eta_0(z)dz=1$, and for each $n\in\N$ define $\eta_n:=\eta-\frac{1}{n}\eta_0$. It follows that
\begin{multline}
\left|u*\eta_{(\e)}(x)-u*\eta_{(\e)}(y)\right|=\left|u*\left(\eta_n+\frac{1}{n}\eta_0\right)_{(\e)}(x)-u*\left(\eta_n+\frac{1}{n}\eta_0\right)_{(\e)}(y)\right|
\\
=\left|u*\left(\left(\eta_n\right)_{(\e)}+\left(\frac{1}{n}\eta_0\right)_{(\e)}\right)(x)-u*\left(\left(\eta_n\right)_{(\e)}+\left(\frac{1}{n}\eta_0\right)_{(\e)}\right)(y)\right|
\\
=\left|\left(u*\left(\eta_n\right)_{(\e)}(x)-u*\left(\eta_n\right)_{(\e)}(y)\right)+\left(u*\left(\frac{1}{n}\eta_0\right)_{(\e)}(x)-u*\left(\frac{1}{n}\eta_0\right)_{(\e)}(y)\right)\right|.
\end{multline}
Therefore,
\begin{multline}
\left|u*\eta_{(\e)}(x)-u*\eta_{(\e)}(y)\right|^q
\\
\leq 2^{q-1}\left(\left|u*\left(\eta_n\right)_{(\e)}(x)-u*\left(\eta_n\right)_{(\e)}(y)\right|^q+\frac{1}{n^q}\left|u*\left(\eta_0\right)_{(\e)}(x)-u*\left(\eta_0\right)_{(\e)}(y)\right|^q\right).
\end{multline}
Thus, since $\int_{\R^N}\eta_n(v)dv=-\frac{1}{n}\neq 0$, $\int_{\R^N}\eta_0(v)dv=1\neq 0$, then
\begin{multline}
\limsup_{\e\to 0^+}\int_{E}\int_{E}\rho_\e(|x-y|)\frac{\left|u_\e(x)-u_\e(y)\right|^q}{|x-y|^{rq}}dydx
\\
\leq 2^{q-1}\limsup_{\e\to 0^+}\int_{E}\int_{E}\rho_\e(|x-y|)\frac{\left|u*\left(\eta_n\right)_{(\e)}(x)-u*\left(\eta_n\right)_{(\e)}(y)\right|^q}{|x-y|^{rq}}dydx
\\
+\frac{2^{q-1}}{n^q}\limsup_{\e\to 0^+}\int_{E}\int_{E}\rho_\e(|x-y|)\frac{\left|u*\left(\eta_0\right)_{(\e)}(x)-u*\left(\eta_0\right)_{(\e)}(y)\right|^q}{|x-y|^{rq}}dydx
\\
=\frac{2^{q-1}}{n^q}\limsup_{\e\to 0^+}\int_{E}\int_{E}\rho_\e(|x-y|)\frac{\left|u(x)-u(y)\right|^q}{|x-y|^{rq}}dydx
+\frac{2^{q-1}}{n^q}\limsup_{\e\to 0^+}\int_{E}\int_{E}\rho_\e(|x-y|)\frac{\left|u(x)-u(y)\right|^q}{|x-y|^{rq}}dydx
\\
=\frac{2^{q}}{n^q}\limsup_{\e\to 0^+}\int_{E}\int_{E}\rho_\e(|x-y|)\frac{\left|u(x)-u(y)\right|^q}{|x-y|^{rq}}dydx.
\end{multline}
Taking the limit as $n\to \infty$ and using Corollary \ref{cor:Finiteness of infinitesimal double integral including kernel} we get
\begin{equation}
\lim_{\e\to 0^+}\int_{E}\int_{E}\rho_\e(|x-y|)\frac{\left|u_\e(x)-u_\e(y)\right|^q}{|x-y|^{rq}}dydx=0.
\end{equation}
\end{proof}

\begin{corollary}(Equivalence Between Gagliardo and Besov Constants Including the Logarithmic Kernel)
\label{cor: the upper(lower) limit of J equals to the upper(lower) limit of Besov approximation of u}

Let $q\in[1,\infty)$, $r\in(0,1)$. Let $\omega\in (0,1)$ be such that $rq<{1/\omega}$. Let $u\in B^r_{q,\infty}(\R^N,\R^d)$, $E\subset\R^N$ be an $\mathcal{L}^N$-measurable set and $\eta\in  W^{1,1}\left(\R^N\right)$. Then,
\begin{multline}
\label{eq:equation13}
\liminf_{\e\to
0^+}\frac{1}{|\ln{\e}|}\left[u*\eta_{(\e)}\right]^q_{W^{r,q}(E,\R^d)}
\\
=\left|\int_{\R^N}\eta(z)dz\right|^q\mathcal{H}^{N-1}\left(S^{N-1}\right)\liminf_{\e\to 0^+}\int_{E}\int_{E}\rho_{\e,\omega}(|x-y|)\frac{|u(x)-u(y)|^q}{|x-y|^{rq}}dydx,
\end{multline}
\begin{multline}
\label{eq:equation14}
\limsup_{\e\to
0^+}\frac{1}{|\ln{\e}|}\left[u*\eta_{(\e)}\right]^q_{W^{r,q}(E,\R^d)}
\\
=\left|\int_{\R^N}\eta(z)dz\right|^q\mathcal{H}^{N-1}\left(S^{N-1}\right)\limsup_{\e\to 0^+}\int_{E}\int_{E}\rho_{\e,\omega}(|x-y|)\frac{|u(x)-u(y)|^q}{|x-y|^{rq}}dydx,
\end{multline}
where $\rho_{\e,\omega}$ is the logarithmic kernel.
\end{corollary}
\begin{proof}
Assume for a moment that $\eta\in C^1_c(\R^N)$. By Lemma \ref{lem:J equals to Besov approximation with mollified u up to a small error} we have
\begin{equation}
\label{eq:equation15}
\liminf_{\e\to 0^+}\frac{1}{|\ln{\e}|}\left[u_{\e}\right]^q_{W^{r,q}(E,\R^d)}=\mathcal{H}^{N-1}\left(S^{N-1}\right)\liminf_{\e\to 0^+}\int_{E}\int_{E}\rho_{\e,\omega}(|x-y|)\frac{|u_\e(x)-u_\e(y)|^q}{|x-y|^{rq}}dydx,
\end{equation}
\begin{equation}
\label{eq:equation16}
\limsup_{\e\to 0^+}\frac{1}{|\ln{\e}|}\left[u_{\e}\right]^q_{W^{r,q}(E,\R^d)}=\mathcal{H}^{N-1}\left(S^{N-1}\right)\limsup_{\e\to 0^+}\int_{E}\int_{E}\rho_{\e,\omega}(|x-y|)\frac{|u_\e(x)-u_\e(y)|^q}{|x-y|^{rq}}dydx.
\end{equation}

Note that by item 2 of Proposition \ref{prop:Properties of the logarithmic kernel} with $\alpha=rq$, the logarithmic kernel satisfies condition \eqref{eq:assumption on the family rho} of Lemma \ref{lem:eta-separating lemma}. Therefore, by Lemma \ref{lem:eta-separating lemma}
\begin{multline}
\label{eq:equation17}
\liminf_{\e\to 0^+}\int_{E}\int_{E}\rho_{\e,\omega}(|x-y|)\frac{\left|u_\e(x)-u_\e(y)\right|^q}{|x-y|^{rq}}dydx
\\
=\left|\int_{\R^N}\eta(z)dz\right|^q\liminf_{\e\to 0^+}\int_{E}\int_{E}\rho_{\e,\omega}(|x-y|)\frac{\left|u(x)-u(y)\right|^q}{|x-y|^{rq}}dydx,
\end{multline}
and
\begin{multline}
\label{eq:equation18}
\limsup_{\e\to 0^+}\int_{E}\int_{E}\rho_{\e,\omega}(|x-y|)\frac{\left|u_\e(x)-u_\e(y)\right|^q}{|x-y|^{rq}}dydx
\\
=\left|\int_{\R^N}\eta(z)dz\right|^q\limsup_{\e\to 0^+}\int_{E}\int_{E}\rho_{\e,\omega}(|x-y|)\frac{\left|u(x)-u(y)\right|^q}{|x-y|^{rq}}dydx.
\end{multline}
Now, \eqref{eq:equation13} and \eqref{eq:equation14} follow from \eqref{eq:equation15},\eqref{eq:equation16},\eqref{eq:equation17},\eqref{eq:equation18}. For $\eta\in W^{1,1}(\R^N)$ choose any sequence  $\{\eta_n\}_{n=1}^\infty\subset C^1_c(\R^N)$ which converges to $\eta$ in $W^{1,1}(\R^N)$. So we have \eqref{eq:equation13} and \eqref{eq:equation14} for $\eta_n$, for every $n\in \N$. Taking the limit as $n$ goes to $\infty$ and using item 1 of Lemma \ref{lem:continuity of upper and lower $G$-functionals}, we obtain \eqref{eq:equation13} and \eqref{eq:equation14} for $\eta\in W^{1,1}(\R^N)$.
\end{proof}

\begin{corollary}(Gagliardo Constants are Controlled by Besov Seminorms)
\label{cor:upper bound for $G$- functionals in terms of Besov seminorm}

Let $1\leq q<\infty$, $r\in (0,1)$, $u\in B^r_{q,\infty}(\R^N,\R^d)$ and $E\subset\R^N$ be an $\mathcal{L}^N$-measurable set. Let $\eta\in W^{1,1}(\R^N)$.
Then,
\begin{align}
\limsup_{\e\to
0^+}\frac{1}{|\ln{\e}|}\left[u_\e\right]^q_{W^{r,q}(E,\R^d)}\leq \left|\int_{\R^N}\eta(z)dz\right|^q \mathcal{H}^{N-1}(S^{N-1})[u]^q_{B^{r}_{q,\infty}(E,\R^d)}<\infty.
\end{align}
\end{corollary}

\begin{proof}
Let $\omega\in (0,1)$ be such that $rq<1/\omega$ and $\rho_{\e,\omega}$ be the logarithmic kernel as defined in Definition \ref{def:logarithmic kernel}.
By \eqref{eq:equation14}, \eqref{eq:equation19} and Definition \ref{def:Besov seminorm} (Definition of Besov seminorm) we get
\begin{multline}
\limsup_{\e\to
0^+}\frac{1}{|\ln{\e}|}\left[u_\e\right]^q_{W^{r,q}(E,\R^d)}=\left|\int_{\R^N}\eta(z)dz\right|^q\mathcal{H}^{N-1}\left(S^{N-1}\right)\limsup_{\e\to 0^+}\int_{E}\int_{E}\rho_{\e,\omega}(|x-y|)\frac{|u(x)-u(y)|^q}{|x-y|^{rq}}dydx
\\
\leq \left|\int_{\R^N}\eta(z)dz\right|^q\limsup_{\e\to 0^+}\int_{S^{N-1}}\int_E\chi_{E}(x+\e n)\frac{|u(x+\e n)-u(x)|^q}{\e^{rq}}dxd\mathcal{H}^{N-1}(n)
\\
\leq \left|\int_{\R^N}\eta(z)dz\right|^q\mathcal{H}^{N-1}(S^{N-1})[u]^q_{B^{r}_{q,\infty}(E,\R^d)}<\infty.
\end{multline}
\end{proof}

\begin{theorem}(Variations Control Gagliardo Constants)
\label{thm:Sandwich inequality for Gagliardo constants}

Let $q \in [1,\infty)$ and $r \in (0,1)$. Suppose $u \in B^r_{q,\infty}(\mathbb{R}^N,\mathbb{R}^d)$, $E \subset \mathbb{R}^N$ be an $\mathcal{L}^N$-measurable set and $\eta \in W^{1,1}(\mathbb{R}^N)$. Then,
\begin{multline}
\label{eq:Sandwich inequality for Gagliardo constants}
\left|\int_{\R^N}\eta(z)dz\right|^q\liminf_{\e\to 0^+}\int_{S^{N-1}}\int_E\chi_{E}(x+\e n)\frac{|u(x+\e n)-u(x)|^q}{\e^{rq}}dxd\mathcal{H}^{N-1}(n)
\\
\leq\liminf_{\e\to
0^+}\frac{1}{|\ln{\e}|}\left[u_\e\right]^q_{W^{r,q}(E,\R^d)}
\leq\limsup_{\e\to
0^+}\frac{1}{|\ln{\e}|}\left[u_\e\right]^q_{W^{r,q}(E,\R^d)}
\\
\leq  \left|\int_{\R^N}\eta(z)dz\right|^q\limsup_{\e\to 0^+}\int_{S^{N-1}}\int_E\chi_{E}(x+\e n)\frac{|u(x+\e n)-u(x)|^q}{\e^{rq}}dxd\mathcal{H}^{N-1}(n).
\end{multline}
\end{theorem}

\begin{proof}
Let $\omega\in (0,1)$ be such that $rq<{1/\omega}$ and $\rho_{\e,\omega}$ be the logarithmic kernel as defined in Definition \ref{def:logarithmic kernel}. By Sandwich Lemma (Lemma \ref{lem:sandwich lemma}) we get for $\alpha=rq$ and $\rho_\e=\rho_{\e,\omega}$
\begin{multline}
\label{eq:Sandwich inequality for Gagliardo constants 2}
\liminf_{\e\to 0^+}\fint_{S^{N-1}}\int_E\chi_{E}(x+\e n)\frac{|u(x+\e n)-u(x)|^q}{\e^{rq}}dxd\mathcal{H}^{N-1}(n)
\\
\leq \liminf_{\e\to 0^+}\int_{E}\int_{E}\rho_{\e,\omega}(|x-y|)\frac{|u(x)-u(y)|^q}{|x-y|^{rq}}dydx\leq \limsup_{\e\to 0^+}\int_{E}\int_{E}\rho_{\e,\omega}(|x-y|)\frac{|u(x)-u(y)|^q}{|x-y|^{rq}}dydx
\\
\leq\limsup_{\e\to 0^+}\fint_{S^{N-1}}\int_E\chi_{E}(x+\e n)\frac{|u(x+\e n)-u(x)|^q}{\e^{rq}}dxd\mathcal{H}^{N-1}(n).
\end{multline}
If $\int_{\R^N}\eta(z)dz=0$, then \eqref{eq:Sandwich inequality for Gagliardo constants} follows from Corollary \ref{cor:upper bound for $G$- functionals in terms of Besov seminorm}.   Assume that $\int_{\R^N}\eta(z)dz\neq 0$. By Corollary \ref{cor: the upper(lower) limit of J equals to the upper(lower) limit of Besov approximation of u} and \eqref{eq:Sandwich inequality for Gagliardo constants 2} we get
\begin{multline}
\label{eq:Sandwich inequality for Gagliardo constants 1}
\liminf_{\e\to 0^+}\fint_{S^{N-1}}\int_E\chi_{E}(x+\e n)\frac{|u(x+\e n)-u(x)|^q}{\e^{rq}}dxd\mathcal{H}^{N-1}(n)
\\
\leq \frac{\liminf_{\e\to
0^+}\frac{1}{|\ln{\e}|}\left[u*\eta_{(\e)}\right]^q_{W^{r,q}(E,\R^d)}}{
\left|\int_{\R^N}\eta(z)dz\right|^q\mathcal{H}^{N-1}\left(S^{N-1}\right)}\leq \frac{\limsup_{\e\to
0^+}\frac{1}{|\ln{\e}|}\left[u*\eta_{(\e)}\right]^q_{W^{r,q}(E,\R^d)}}{
\left|\int_{\R^N}\eta(z)dz\right|^q\mathcal{H}^{N-1}\left(S^{N-1}\right)}
\\
\leq\limsup_{\e\to 0^+}\fint_{S^{N-1}}\int_E\chi_{E}(x+\e n)\frac{|u(x+\e n)-u(x)|^q}{\e^{rq}}dxd\mathcal{H}^{N-1}(n).
\end{multline}
Multiplying both sides of inequality \eqref{eq:Sandwich inequality for Gagliardo constants 1} by $\left|\int_{\R^N}\eta(z)dz\right|^q\mathcal{H}^{N-1}\left(S^{N-1}\right)$ we obtain \eqref{eq:Sandwich inequality for Gagliardo constants}.
\end{proof}

\begin{theorem}(Equivalence Between Gagliardo and Beosv Constants)
\label{thm:connection between Gagliardo constant and Besov constant dependent on arbitrary kernel}

Let $q\in[1,\infty)$, $r\in(0,1)$. Let $u\in B^r_{q,\infty}(\R^N,\R^d)$, $E\subset\R^N$ be an $\mathcal{L}^N$-measurable set and $\eta\in  W^{1,1}\left(\R^N\right)$. If the following limit exists:
\begin{equation}
\label{eq:equation200}
\lim_{\e\to 0^+}\int_{S^{N-1}}\int_E\chi_{E}(x+\e n)\frac{|u(x+\e n)-u(x)|^q}{\e^{rq}}dxd\mathcal{H}^{N-1}(n),
\end{equation}
then, for every kernel $\rho_\e$ we get
\begin{multline}
\label{eq:equivalence between G-constant, variation and B-constant}
\lim_{\e\to
0^+}\frac{1}{|\ln{\e}|}\left[u_\e\right]^q_{W^{r,q}(E,\R^d)}=\left|\int_{\R^N}\eta(z)dz\right|^q\lim_{\e\to 0^+}\int_{S^{N-1}}\int_E\chi_{E}(x+\e n)\frac{|u(x+\e n)-u(x)|^q}{\e^{rq}}dxd\mathcal{H}^{N-1}(n)
\\
=\left|\int_{\R^N}\eta(z)dz\right|^q\mathcal{H}^{N-1}\left(S^{N-1}\right)\lim_{\e\to 0^+}\int_{E}\int_{E}\rho_{\e}(|x-y|)\frac{|u(x)-u(y)|^q}{|x-y|^{rq}}dydx.
\end{multline}
\end{theorem}

\begin{proof}
Formulas \eqref{eq:equivalence between G-constant, variation and B-constant} follow from assumption \eqref{eq:equation200}, Theorem \ref{thm:Sandwich inequality for Gagliardo constants} and Sandwich Lemma with $\alpha=rq$ (Lemma \ref{lem:sandwich lemma}).
\end{proof}

\begin{corollary}(Equivalence Between Gagliardo Constants and $B^{r,q}$-Seminorms)
\label{cor:main result}

Let $1\leq q<\infty$, $r\in(0,1)$, $u\in B^r_{q,\infty}(\R^N,\R^d)$, $E\subset\R^N$ be an $\mathcal{L}^N$-measurable set  and $\eta\in  W^{1,1}\left(\R^N\right)$. If the following limit exists:
\begin{equation}
\label{eq:equation20}
\lim_{\e\to 0^+}\int_{S^{N-1}}\int_E\chi_{E}(x+\e n)\frac{|u(x+\e n)-u(x)|^q}{\e^{rq}}dxd\mathcal{H}^{N-1}(n),
\end{equation}
then
\begin{multline}
\label{eq:equation110}
\lim_{\e\to
0^+}\frac{1}{|\ln{\e}|}\left[u_\e\right]^q_{W^{r,q}(E,\R^d)}=N\left|\int_{\R^N}\eta(z)dz\right|^q\lim_{\e\to 0^+}\int_{E}\frac{1}{\e^N}\int_{E\cap B_\e(x)}\frac{|u(x)-u(y)|^q}{|x-y|^{rq}}dydx
\\
=N\left|\int_{\R^N}\eta(z)dz\right|^q[u]^q_{B^{r,q}(E,\R^d)},
\end{multline}
where $[u]_{B^{r,q}(E,\R^d)}$ is the upper infinitesimal $B^{r,q}$-seminorm defined in \ref{def:Brq seminorms}.
\end{corollary}

\begin{proof}
By Remark \ref{rem:compact support property for trivial kernel} the trivial kernel, $\tilde{\rho}_\e$, is a kernel. Therefore, by assumption \eqref{eq:equation20} and Theorem \ref{thm:connection between Gagliardo constant and Besov constant dependent on arbitrary kernel} we obtain
\begin{multline}
\lim_{\e\to
0^+}\frac{1}{|\ln{\e}|}\left[u_\e\right]^q_{W^{r,q}(E,\R^d)}=\left|\int_{\R^N}\eta(z)dz\right|^q\mathcal{H}^{N-1}\left(S^{N-1}\right)\lim_{\e\to 0^+}\int_{E}\int_{E}\tilde{\rho}_\e(|x-y|)\frac{|u(x)-u(y)|^q}{|x-y|^{rq}}dydx
\\
=\left|\int_{\R^N}\eta(z)dz\right|^q\frac{\mathcal{H}^{N-1}\left(S^{N-1}\right)}{\Leb^{N}\left(B_1(0)\right)}\lim_{\e\to 0^+}\int_{E}\frac{1}{\e^N}\int_{E\cap B_\e(x)}\frac{|u(x)-u(y)|^q}{|x-y|^{rq}}dydx
\\
=N\left|\int_{\R^N}\eta(z)dz\right|^q\lim_{\e\to 0^+}\int_{E}\frac{1}{\e^N}\int_{E\cap B_\e(x)}\frac{|u(x)-u(y)|^q}{|x-y|^{rq}}dydx.
\end{multline}
The equation $\Leb^{N}\left(B_1(0)\right)=\frac{\mathcal{H}^{N-1}\left(S^{N-1}\right)}{N}$ follows from polar coordinates.

\end{proof}

\begin{remark}(Consistency with Previous Results)

Equation \eqref{eq:equation110} can be derived for functions $u\in BV(\R^N,\R^d)\cap L^\infty(\R^N,\R^d)$, $1<q<\infty$, $r=\frac{1}{q}$, $\eta\in W^{1,1}(\R^N)$, and an open set $\Omega\subset\R^N$ with a bounded Lipschitz boundary such that $\|Du\|(\partial \Omega)=0$. By combining Theorem 1.2 in \cite{PAsymptotic} and Theorem 1.1 in \cite{P}, we get
\begin{multline}
\lim_{\e\to
0^+}\frac{1}{|\ln{\e}|}\left[u_\e\right]^q_{W^{\frac{1}{q},q}(\Omega,\R^d)}
\\
=\frac{\int_{\R^{N-1}}2\left(1+|v|^2\right)^{-\frac{N+1}{2}}dv}{\frac{1}{N}\int_{S^{N-1}}|z_1|d\mathcal{H}^{N-1}(z)}\left|\int_{\R^N}\eta(z)dz\right|^q\lim_{\e\to 0^+}\int_{\Omega}\frac{1}{\e^N}\int_{\Omega\cap B_\e(x)}\frac{|u(x)-u(y)|^q}{|x-y|}dydx.
\end{multline}
According to Proposition \ref{prop:calculation of NC_N}, we obtain $\int_{\R^{N-1}}2\left(1+|v|^2\right)^{-\frac{N+1}{2}}dv=\int_{S^{N-1}}|z_1|d\mathcal{H}^{N-1}(z)$.
\end{remark}

\section{Jump Detection in $BV\cap B^{1/p,p}$}
In this section we prove formulas for $\lim_{\e\to
0^+}\frac{1}{|\ln{\e}|}\left[u_\e\right]^q_{W^{1/q,q}(B,\R^d)}$, where $u\in BV\cap B^{1/p,p}$ and $B\subset\R^N$ is a Borel set (refer to Corollary \ref{cor:connection between $G$-functional and the $q$-jump variation}).
\begin{remark}($BV\cap L^\infty$ is a Subset of $B^r_{q,\infty},rq\leq 1$)
\label{rem: BV and bounded lies in Besov}

Let $u\in BV\left(\R^N,\R^d\right)\cap L^\infty\left(\R^N,\R^d\right)$. Let $1\leq q<\infty$ and $r\in (0,1)$ be such that $rq\leq1$. By Lemma \ref{lem:variation inequality} we get
\begin{multline}
\sup_{h\in\R^N\setminus\{0\}}\int_{\R^N}\frac{|u(x+h)-u(x)|^q}{|h|^{rq}}dx
\\
\leq \sup_{h\in\R^N\setminus B_1(0)}\int_{\R^N}\frac{|u(x+h)-u(x)|^q}{|h|^{rq}}dx
+\sup_{h\in B_1(0)\setminus\{0\}}\int_{\R^N}\frac{|u(x+h)-u(x)|^q}{|h|^{rq}}dx
\\
\leq 2^{q}\|u\|^{q}_{L^q(\R^N,\R^d)}
+2^{q-1}\|u\|^{q-1}_{L^\infty(\R^N,\R^d)} \sup_{h\in B_1(0)\setminus\{0\}}\int_{\R^N}\frac{|u(x+h)-u(x)|}{|h|}dx
\\
\leq  2^{q}\|u\|^{q}_{L^q(\R^N,\R^d)}
+2^{q-1}\|u\|^{q-1}_{L^\infty(\R^N,\R^d)}\|Du\|(\R^N)<\infty.
\end{multline}
Note that since $u\in L^1\left(\R^N,\R^d\right)\cap L^\infty\left(\R^N,\R^d\right)$, then $u\in  L^q\left(\R^N,\R^d\right)$. Thus, by Definition \ref{def:definition of Besov space} (definition of Besov space) we get $u\in B^{r}_{q,\infty}(\R^N,\R^d)$.
\end{remark}

\begin{lemma}(Interpolation for Besov Seminorms)
\label{lem:estimate for Besov seminorm by the variation}

Let $p\in(1,\infty)$ and $u\in BV(\R^N,\R^d)\cap B^{1/p,p}(\R^N,\R^d)$. Then for every $q\in(1,p)$ we have $u\in B^{1/q,q}(\R^N,\R^d)$ and
\begin{equation}
[u]^q_{B^{1/q}_{q,\infty}(\R^N,\R^d)}\leq \left(\|Du\|(\R^N)\right)^\alpha\left([u]^{p}_{B^{1/p}_{p,\infty}(\R^N,\R^d)}\right)^{1-\alpha},
\end{equation}
where $\alpha:=\frac{p-q}{p-1}$.
\end{lemma}
\begin{proof}
Since $\alpha=\frac{p-q}{p-1}$, then $q=\alpha+(1-\alpha)p$. By H{\"o}lder's inequality and Lemma \ref{lem:variation inequality} we get
\begin{multline}
[u]^{q}_{B^{1/q}_{q,\infty}(\R^N,\R^d)}=\sup_{h\in \R^N\setminus\{0\}}\int_{\R^N}\frac{|u(x+h)-u(x)|^q}{|h|}dx
\\
=\sup_{h\in \R^N\setminus\{0\}}\int_{\R^N}\left(\frac{|u(x+h)-u(x)|}{|h|}\right)^{\alpha}\left(\frac{|u(x+h)-u(x)|^p}{|h|}\right)^{1-\alpha}dx
\\
\leq \left(\sup_{h\in \R^N\setminus\{0\}}\int_{\R^N}\frac{|u(x+h)-u(x)|}{|h|}dx\right)^{\alpha}\left(\sup_{h\in \R^N\setminus\{0\}}\int_{\R^N}\frac{|u(x+h)-u(x)|^p}{|h|}dx\right)^{1-\alpha}
\\
\leq\left(\|Du\|(\R^N)\right)^\alpha\left([u]^{p}_{B^{1/p}_{p,\infty}(\R^N,\R^d)}\right)^{1-\alpha}.
\end{multline}
\end{proof}

\begin{corollary}(Convergence of the Truncated Family in Besov Seminorm)
\label{cor:convergence of truncated family in Besov seminorm}

Let $p\in(1,\infty)$ and $u\in BV(\R^N,\R^d)\cap B^{1/p}_{p,\infty}(\R^N,\R^d)$. Then, for every $q\in (1,p)$ we have
\begin{equation}
\label{eq:convergence of truncated family in Besov seminorm}
\lim_{l\to\infty}[u-u_l]_{B^{1/q}_{q,\infty}(\R^N,\R^d)}=0,
\end{equation}
where $\{u_l\}_{l\in[0,\infty)}$ is the truncated family obtained by $u$ as defined in Definition \ref{def:truncated family}. In particular, the truncated family $u_l$ converges to $u$ in the norm of the space $B^{1/q}_{q,\infty}(\R^N,\R^d)$, which means that
\begin{equation}
\lim_{l\to\infty}\left([u-u_l]_{B^{1/q}_{q,\infty}(\R^N,\R^d)}+\|u-u_l\|_{L^q(\R^N,\R^d)}\right)=0.
\end{equation}
\end{corollary}
\begin{proof}
Let $q\in(1,p)$ and denote $\alpha:=\frac{p-q}{p-1}$. From Lemma \ref{lem:estimate for Besov seminorm by the variation} we get
\begin{equation}
[u-u_l]^q_{B^{1/q}_{q,\infty}(\R^N,\R^d)}\leq \left(\|D(u-u_l)\|(\R^N)\right)^\alpha\left([u-u_l]^{p}_{B^{1/p}_{p,\infty}(\R^N,\R^d)}\right)^{1-\alpha}.
\end{equation}
Note that
\begin{multline}
[u-u_l]^{p}_{B^{1/p}_{p,\infty}(\R^N,\R^d)}=\sup_{h\in\R^N\setminus\{0\}}\int_{\R^N}\frac{|(u-u_l)(x+h)-(u-u_l)(x)|^p}{|h|}dx
\\
\leq 2^{p-1}\sup_{h\in\R^N\setminus\{0\}}\int_{\R^N}\frac{|u(x+h)-u(x)|^p}{|h|}dx+2^{p-1}\sup_{h\in\R^N\setminus\{0\}}\int_{\R^N}\frac{|u_l(x+h)-u_l(x)|^p}{|h|}dx
\\
\leq 2^p[u]^{p}_{B^{1/p}_{p,\infty}(\R^N,\R^d)}<\infty.
\end{multline}
Since by Lemma \ref{lem:Convergence of the truncated family in BV} we have $\lim_{l\to \infty}\|D(u-u_l)\|(\R^N)=0$, then \eqref{eq:convergence of truncated family in Besov seminorm} follows. The convergence of $u_l$ to $u$ as $l\to\infty$ in the norm of the space $B^{1/q}_{q,\infty}(\R^N,\R^d)$ follows from \eqref{eq:convergence of truncated family in Besov seminorm} and Lemma \ref{lem:convergence of the truncated in Lebesgue space}.
\end{proof}

Recall Definition \ref{def:approximate jump points} for $u^+,u^-,\mathcal{J}_u,\nu_u$.
\begin{theorem}(Proposition 2.4 in \cite{P})
\label{thm:convergence of $Ie(h;K)$ to the pre jump variation}

Let $\Omega\subset\R^N$ be an open set, $1<q<\infty$, $u\in BV_{\text{loc}}(\Omega,\R^d)\cap L_{\text{loc}}^\infty (\Omega,\R^d)$. Then, for every $h\in \R^N$ and every compact set $K\subset \Omega$ such that $\|Du\|\left(\partial K\right)=0$ we have
\begin{equation}
\lim_{\e\to 0^+}\int_{K}\frac{|u(x+\e h)-u(x)|^q}{\e}dx=\int_{K\cap \mathcal{J}_u}\left|u^+(x)-u^-(x)\right|^q|\nu_u(x)\cdot h|d\mathcal{H}^{N-1}(x).
\end{equation}
\end{theorem}

\begin{remark}(Variation Negligibility of Boundaries of Sets)
\label{rem:negligibility with respect to variation}

The purpose of this remark is to explain the condition $\|Du\|\left(\partial K\right)=0$ in Theorem \ref{thm:convergence of $Ie(h;K)$ to the pre jump variation}. For an open set $\Omega\subset \mathbb{R}^N$ and a function $u\in BV(\Omega,\mathbb{R}^d)$, it follows that $\|Du\|\geq |u^+-u^-|\mathcal{H}^{N-1}\llcorner \mathcal{J}_u$ (refer to Lemma 3.76 in \cite{AFP}). It is important to note that according to Definition \ref{def:approximate jump points}, $|u^+(x)-u^-(x)|>0$ for $x\in \mathcal{J}_u$. Therefore, for a set $E\subset\Omega$, the assumption $\|Du\|(\partial E)=0$ indicates that $\mathcal{H}^{N-1}\left(\partial E\cap\mathcal{J}_u\right)=0$, implying that the portion of the jump set $\mathcal{J}_u$ within the topological boundary of $E$ is negligible with respect to $\mathcal{H}^{N-1}$.
\end{remark}

\begin{lemma}(Equivalence Between Variation and Jump Variation in the $BV$ Case)
\label{lem: limit for Besov integral in terms of jumps}

Let $1<p<\infty$, $u\in BV(\R^N,\R^d)\cap B^{\frac{1}{p},p} (\R^N,\R^d)$ and $1<q<p$. Then, for every $n\in \R^N$ and every Borel set $B\subset\R^N$ such that $\mathcal{H}^{N-1}(\partial B\cap \mathcal{J}_u)=0$ we have
\begin{equation}
\label{eq:equation27}
\lim_{\e\to 0^+}\int_B\chi_{B}(x+\e n)\frac{|u(x+\e n)-u(x)|^q}{\e}dx
=\int_{B\cap \mathcal{J}_u}\left|u^+(x)-u^-(x)\right|^q|\nu_u(x)\cdot n|d\mathcal{H}^{N-1}(x).
\end{equation}
In particular, the following limit exists:
\begin{multline}
\label{eq:equation28}
\lim_{\e\to 0^+}\int_{S^{N-1}}\int_B\chi_{B}(x+\e n)\frac{|u(x+\e n)-u(x)|^q}{\e}dxd\mathcal{H}^{N-1}(n)
\\
=\left(\int_{S^{N-1}}|z_1|~d\Haus^{N-1}(z)\right)\int_{B\cap \mathcal{J}_u}\left|u^+(x)-u^-(x)\right|^qd\mathcal{H}^{N-1}(x).
\end{multline}
\end{lemma}

\begin{remark}(The Assumption of Bounded Variation in Lemma \ref{lem: limit for Besov integral in terms of jumps})

In Lemma \ref{lem: limit for Besov integral in terms of jumps}, the assumption that $u$ has bounded variation cannot be dropped in general to obtain inequality \eqref{eq:equation28}. There are examples of functions in $B^{\frac{1}{p},p}(\R^N,\R^d)$ for which equation \eqref{eq:equation28} does not hold. Examples can be found in \cite{PA}. Here, we will mention that bi-H{\"o}lder functions can be used to demonstrate that the jump variation of such a function is zero (the right-hand side of \eqref{eq:equation28}), but the variation of $u$ (the left-hand side of \eqref{eq:equation28}) is positive.
\end{remark}

\begin{proof}
\textbf{Step 1}: $u\in L^\infty(\R^N,\R^d)$.
Let $n\in \R^N$, and let $B\subset\R^N$ be a Borel set. Let $K\subset B^o$ be a compact set such that $\|Du\|\left(\partial K\right)=0$, where $B^o$ is the topological interior of $B$. By Theorem \ref{thm:convergence of $Ie(h;K)$ to the pre jump variation} we obtain
\begin{multline}
\liminf_{\e\to 0^+}\int_B\chi_{B}(x+\e n)\frac{|u(x+\e n)-u(x)|^q}{\e}dx\geq \liminf_{\e\to 0^+}\int_{B^o}\chi_{B^o}(x+\e n)\frac{|u(x+\e n)-u(x)|^q}{\e}dx
\\
\geq \liminf_{\e\to 0^+}\int_{K}\frac{|u(x+\e n)-u(x)|^q}{\e}dx
=\int_{K\cap \mathcal{J}_u}\left|u^+(x)-u^-(x)\right|^q|\nu_u(x)\cdot n|d\mathcal{H}^{N-1}(x).
\end{multline}
Taking the supremum over compact sets $K\subset B^o$ such that $\|Du\|\left(\partial K\right)=0$ we get by Lemma \ref{lem:compact negligible boundary property}
\begin{align}
\label{eq:equation25}
&\liminf_{\e\to 0^+}\int_B\chi_{B}(x+\e n)\frac{|u(x+\e n)-u(x)|^q}{\e}dx\geq \int_{B^o\cap \mathcal{J}_u}\left|u^+(x)-u^-(x)\right|^q|\nu_u(x)\cdot n|d\mathcal{H}^{N-1}(x).
\end{align}
Let $\Omega$ be an open set such that $\overline{B}\subset \Omega$ and $\|Du\|\left(\partial \Omega\right)=0$.
By Lemma \ref{lem: F is at most countable}, there exists a sequence of numbers $\{R_k\}_{k=1}^\infty$ such that for every $k\in \N$: $R_k>0$, $R_k<R_{k+1}$, $\|Du\|\left(\partial B_{R_k}(0)\right)=0$, and $\lim_{k\to \infty}R_k=\infty$. Note that since $\partial \left(\overline{\Omega}\cap \overline{B}_{R_{k}}(0)\right)\subset \partial\overline{\Omega}\cup \partial\overline{B}_{R_{k}}(0)$, then $\|Du\|\left(\partial\left(\overline{\Omega}\cap \overline{B}_{R_{k}}(0)\right)\right)=0$. Note that if $n=0$, then equation \eqref{eq:equation27} holds trivially. Assume $n\neq 0$. It follows from Theorem \ref{thm:convergence of $Ie(h;K)$ to the pre jump variation}, Lemma \ref{lem:variation inequality} and Remark \ref{rem:negligibility with respect to variation}
\begin{multline}
\label{eq:equation30}
\limsup_{\e\to 0^+}\int_B\chi_{B}(x+\e n)\frac{|u(x+\e n)-u(x)|^q}{\e}dx
\leq \limsup_{\e\to 0^+}\int_\Omega\frac{|u(x+\e n)-u(x)|^q}{\e}dx
\\
\leq\limsup_{\e\to 0^+}\int_{\overline{\Omega}}\frac{|u(x+\e n)-u(x)|^q}{\e}dx
\\
\leq \limsup_{\e\to 0^+}\int_{\overline{\Omega}\cap \overline{B}_{R_{k+1}}(0)}\frac{|u(x+\e n)-u(x)|^q}{\e}dx
+ \limsup_{\e\to 0^+}\int_{\overline{\Omega}\setminus \overline{B}_{R_{k+1}}(0)}\frac{|u(x+\e n)-u(x)|^q}{\e}dx
\\
\leq \int_{\left(\overline{\Omega}\cap \overline{B}_{R_{k+1}}(0)\right)\cap \mathcal{J}_u}\left|u^+(x)-u^-(x)\right|^q|\nu_u(x)\cdot n|d\mathcal{H}^{N-1}(x)
\\
+2^{q-1}\|u\|^{q-1}_{L^\infty(\R^N,\R^d)}\limsup_{\e\to 0^+}\int_{\R^N\setminus \overline{B}_{R_{k+1}}(0)}\frac{|u(x+\e n)-u(x)|}{\e}dx
\\
\leq\int_{\Omega\cap \mathcal{J}_u}\left|u^+(x)-u^-(x)\right|^q|\nu_u(x)\cdot n|d\mathcal{H}^{N-1}(x)+2^{q-1}\|u\|^{q-1}_{L^\infty(\R^N,\R^d)}|n|\|Du\|\left(\R^N\setminus \overline{B}_{R_{k}}(0)\right).
\end{multline}
Taking the limit as $k\to \infty$ in \eqref{eq:equation30}, we get
\begin{align}
&\limsup_{\e\to 0^+}\int_B\chi_{B}(x+\e n)\frac{|u(x+\e n)-u(x)|^q}{\e}dx
\leq\int_{\Omega\cap \mathcal{J}_u}\left|u^+(x)-u^-(x)\right|^q|\nu_u(x)\cdot n|d\mathcal{H}^{N-1}(x).
\end{align}
Therefore, by the Lemma \ref{lem:the open negligible boundary property} we get
\begin{align}
\label{eq:equation31}
&\limsup_{\e\to 0^+}\int_B\chi_{B}(x+\e n)\frac{|u(x+\e n)-u(x)|^q}{\e}dx
\leq\int_{\overline{B}\cap \mathcal{J}_u}\left|u^+(x)-u^-(x)\right|^q|\nu_u(x)\cdot n|d\mathcal{H}^{N-1}(x).
\end{align}
By \eqref{eq:equation25} and \eqref{eq:equation31} we get \eqref{eq:equation27} for every Borel set $B\subset \R^N$ such that $\mathcal{H}^{N-1}(\partial B\cap \mathcal{J}_u)=0$.

Since by Lemma \ref{lem:variation inequality}
\begin{multline}
\sup_{\e\in (0,\infty)}\left(\int_B\chi_{B}(x+\e n)\frac{|u(x+\e n)-u(x)|^q}{\e}dx\right)\leq\sup_{\e\in (0,\infty)}\left(\int_B\frac{|u(x+\e n)-u(x)|^q}{\e}dx\right)
\\
\leq 2^{q-1}\|u\|^{q-1}_{L^\infty(\R^N,\R^d)}\sup_{\e\in (0,\infty)}\left(\int_B\frac{|u(x+\e n)-u(x)|}{\e}dx\right)
\\
\leq 2^{q-1}\|u\|^{q-1}_{L^\infty(\R^N,\R^d)}|n|\|Du\|(\R^N)<\infty,
\end{multline}
then we get by Dominated Convergence Theorem, equation \eqref{eq:equation27}, Fubini's Theorem and Proposition \ref{prop:independence of integrals on the sphere}
\begin{multline}
\label{eq:equation32}
\lim_{\e\to 0^+}\int_{S^{N-1}}\int_B\chi_{B}(x+\e n)\frac{|u(x+\e n)-u(x)|^q}{\e}dxd\mathcal{H}^{N-1}(n)
\\
=\int_{S^{N-1}}\left(\lim_{\e\to 0^+}\int_B\chi_{B}(x+\e n)\frac{|u(x+\e n)-u(x)|^q}{\e}dx\right)d\mathcal{H}^{N-1}(n)
\\
=\int_{S^{N-1}}\left(\int_{B\cap \mathcal{J}_u}\left|u^+(x)-u^-(x)\right|^q|\nu_u(x)\cdot n|d\mathcal{H}^{N-1}(x)\right)d\mathcal{H}^{N-1}(n)
\\
=\int_{B\cap \mathcal{J}_u}\left|u^+(x)-u^-(x)\right|^q\left(\int_{S^{N-1}}|\nu_u(x)\cdot n|d\mathcal{H}^{N-1}(n)\right)d\mathcal{H}^{N-1}(x)
\\
=\left(\int_{S^{N-1}}|e_1\cdot n|d\mathcal{H}^{N-1}(n)\right)\int_{B\cap \mathcal{J}_u}\left|u^+(x)-u^-(x)\right|^qd\mathcal{H}^{N-1}(x).
\end{multline}
In particular, the limit in \eqref{eq:equation28} exists.
\\
\textbf{Step 2}: $u$ is not necessarily bounded. For every $l\in [0,\infty)$ we have $u_l\in L^\infty(\R^N,\R^d)$, where $\{u_l\}_{l\in [0,\infty)}$ is the truncated family defined in Definition \ref{def:truncated family}. So we get for every $l\in [0,\infty)$ by the previous step the formulas
\begin{equation}
\label{eq:equation27 with l}
\lim_{\e\to 0^+}\int_B\chi_{B}(x+\e n)\frac{|u_l(x+\e n)-u_l(x)|^q}{\e}dx
=\int_{B\cap \mathcal{J}_{u_l}}\left|(u_l)^+(x)-(u_l)^-(x)\right|^q|\nu_{u_l}(x)\cdot n|d\mathcal{H}^{N-1}(x),
\end{equation}
and
\begin{multline}
\label{eq:equation28 with l}
\lim_{\e\to 0^+}\int_{S^{N-1}}\int_B\chi_{B}(x+\e n)\frac{|u_l(x+\e n)-u_l(x)|^q}{\e}dxd\mathcal{H}^{N-1}(n)
\\
=\left(\int_{S^{N-1}}|z_1|~d\Haus^{N-1}(z)\right)\int_{B\cap \mathcal{J}_{u_l}}\left|(u_l)^+(x)-(u_l)^-(x)\right|^qd\mathcal{H}^{N-1}(x).
\end{multline}
By Lemma \ref{lem:limit of jump variation of truncated family} we obtain
\begin{equation}
\label{eq:equation28 with l including normal}
\lim_{l\to\infty}\int_{B\cap \mathcal{J}_{u_l}}\left|(u_l)^+(x)-(u_l)^-(x)\right|^q|\nu_{u_l}(x)\cdot n|d\mathcal{H}^{N-1}(x)=\int_{B\cap \mathcal{J}_{u}}\left|u^+(x)-u^-(x)\right|^q|\nu_{u}(x)\cdot n|d\mathcal{H}^{N-1}(x),
\end{equation}
and
\begin{equation}
\lim_{l\to\infty}\int_{B\cap \mathcal{J}_{u_l}}\left|(u_l)^+(x)-(u_l)^-(x)\right|^qd\mathcal{H}^{N-1}(x)=\int_{B\cap \mathcal{J}_{u}}\left|u^+(x)-u^-(x)\right|^qd\mathcal{H}^{N-1}(x).
\end{equation}
By Corollary \ref{cor:convergence of truncated family in Besov seminorm}, we know that the truncated family $u_l$ converges to $u$ in Besov space $B^{1/q}_{q,\infty}$. Let us denote
\begin{equation}
F_\e(u):=\int_B\chi_{B}(x+\e n)\frac{|u(x+\e n)-u(x)|^q}{\e}dx.
\end{equation}
By Lemma \ref{lem:Continuity of (r,q)-variation in Besov spaces Brq}, we get
\begin{equation}
\label{eq:limits of F epsilon}
\lim_{l\to \infty}\left(\limsup_{\e\to 0^+}F_\e(u_l)\right)=\limsup_{\e\to 0^+}F_\e(u),\quad \lim_{l\to \infty}\left(\liminf_{\e\to 0^+}F_\e(u_l)\right)=\liminf_{\e\to 0^+}F_\e(u).
\end{equation}
By \eqref{eq:equation27 with l} the limit $\lim_{\e\to 0^+}F_\e(u_l)$ exists for every $l\in [0,\infty)$. Thus, by \eqref{eq:limits of F epsilon}, we conclude the existence of the limit $\lim_{\e\to 0^+}F_\e(u)$, and
\begin{equation}
\label{eq:limit of F epsilon}
\lim_{l\to \infty}\left(\lim_{\e\to 0^+}F_\e(u_l)\right)=\lim_{\e\to 0^+}F_\e(u).
\end{equation}
Taking the limit in \eqref{eq:equation27 with l} as $l\to \infty$, and using \eqref{eq:equation28 with l including normal} and \eqref{eq:limit of F epsilon}, we obtain \eqref{eq:equation27}. By the Dominated Convergence Theorem, we deduce \eqref{eq:equation28} from \eqref{eq:equation27}, as shown in calculation \eqref{eq:equation32}.
\end{proof}

\begin{definition}($q$-Jump Variation)
\label{def:jump variation}

Let $\Omega\subset\R^N$ be an open set, $u\in L^1_{\text{loc}}(\Omega,\R^d)$, $q\in \R$, and $S\subset\Omega$ is an $\mathcal{H}^{N-1}$-measurable set. We define the \textit{$q$-jump variation of $u$ in $S$} by
\begin{equation}
JV_{u,q}(S):=\int_{S\cap \mathcal{J}_u}|u^+(x)-u^-(x)|^qd\mathcal{H}^{N-1}(x).
\end{equation}
Let $n\in S^{N-1}$. We define the \textit{$q$-jump variation of $u$ in $S$ in direction $n$} by
\begin{equation}
JV_{u,q,n}(S):=\int_{S\cap \mathcal{J}_u}|u^+(x)-u^-(x)|^q|\nu_{u}(x)\cdot n|d\mathcal{H}^{N-1}(x).
\end{equation}
\end{definition}

\begin{corollary}(Equivalence Between Gagliardo Constants and the $q$-Jump Variations)
\label{cor:connection between $G$-functional and the $q$-jump variation}

Let $p\in(1,\infty)$, $q\in (1,p)$, $u\in BV(\R^N,\R^d)\cap B^{\frac{1}{p},p}(\R^N,\R^d)$, $\eta\in W^{1,1}\left(\R^N\right)$ and $B\subset\R^N$ be a Borel set such that $\mathcal{H}^{N-1}(\partial B\cap \mathcal{J}_u)=0$. Then, for every kernel $\rho_\e$, we obtain

\begin{multline}
\label{eq:equation104}
\lim_{\e\to
0^+}\frac{1}{|\ln{\e}|}\left[u_\e\right]^q_{W^{1/q,q}(B,\R^d)}
\\
=\left|\int_{\R^N}\eta(z)dz\right|^q\mathcal{H}^{N-1}\left(S^{N-1}\right)\lim_{\e\to 0^+}\int_{B}\int_{B}\rho_{\e}(|x-y|)\frac{|u(x)-u(y)|^q}{|x-y|}dydx
\\
=\left|\int_{\R^N}\eta(z)dz\right|^q\lim_{\e\to 0^+}\int_{S^{N-1}}\int_B\chi_{B}(x+\e n)\frac{|u(x+\e n)-u(x)|^q}{\e}dxd\mathcal{H}^{N-1}(n)
\\
=\left|\int_{\R^N}\eta(z)dz\right|^q\left(\int_{S^{N-1}}|z_1|~d\Haus^{N-1}(z)\right)\int_{\mathcal{J}_u\cap B}
\Big|u^+(x)-u^-(x)\Big|^q d\mathcal{H}^{N-1}(x).
\end{multline}
\end{corollary}

\begin{proof}
By Lemma \ref{lem: limit for Besov integral in terms of jumps} the limit
\begin{equation}
\label{eq:equation28 with l1}
\lim_{\e\to 0^+}\int_{S^{N-1}}\int_B\chi_{B}(x+\e n)\frac{|u(x+\e n)-u(x)|^q}{\e}dxd\mathcal{H}^{N-1}(n)
\end{equation}
exists, and
\begin{multline}
\label{eq:equation28 with l11}
\lim_{\e\to 0^+}\int_{S^{N-1}}\int_B\chi_{B}(x+\e n)\frac{|u(x+\e n)-u(x)|^q}{\e}dxd\mathcal{H}^{N-1}(n)
\\
=\left(\int_{S^{N-1}}|z_1|~d\Haus^{N-1}(z)\right)\int_{B\cap \mathcal{J}_u}\left|u^+(x)-u^-(x)\right|^qd\mathcal{H}^{N-1}(x).
\end{multline}
Since the limit in \eqref{eq:equation28 with l1} exists, we get by Theorem \ref{thm:connection between Gagliardo constant and Besov constant dependent on arbitrary kernel} with $r=\frac{1}{q}$ and $E=B$ that
\begin{multline}
\label{eq:equation1044}
\lim_{\e\to
0^+}\frac{1}{|\ln{\e}|}\left[u_\e\right]^q_{W^{1/q,q}(B,\R^d)}
=\left|\int_{\R^N}\eta(z)dz\right|^q\mathcal{H}^{N-1}\left(S^{N-1}\right)\lim_{\e\to 0^+}\int_{B}\int_{B}\rho_{\e}(|x-y|)\frac{|u(x)-u(y)|^q}{|x-y|}dydx
\\
=\left|\int_{\R^N}\eta(z)dz\right|^q\lim_{\e\to 0^+}\int_{S^{N-1}}\int_B\chi_{B}(x+\e n)\frac{|u(x+\e n)-u(x)|^q}{\e}dxd\mathcal{H}^{N-1}(n).
\end{multline}

We get \eqref{eq:equation104} by equations \eqref{eq:equation28 with l11} and \eqref{eq:equation1044}.
\end{proof}

\subsection{Some observations about jumps of functions in $B^r_{q,\infty}=B^{r,q}$}
\begin{lemma}(Besov Spaces Embed in Fractional Sobolev Spaces)
\label{lem:Besov spaces lie in fractional Sobolev}

Let $0<r<s<1$, $q\in[1,\infty)$. Then,
\begin{equation}
B^s_{q,\infty}(\R^N,\R^d)\subset W^{r,q}_{\text{loc}}(\R^N,\R^d).
\end{equation}
\end{lemma}
\begin{proof}
Let $u\in B^s_{q,\infty}(\R^N,\R^d)$ and $K\subset \R^N$ be a compact set. We have by additivity of integral
\begin{multline}
\label{eq:decomposition of Gagliardo expression}
\int_{K}\int_{K}\frac{|u(x)-u(y)|^q}{|x-y|^{rq+N}}dydx
\\
=\int_{K}\left(\int_{K\cap B_1(x)}\frac{|u(x)-u(y)|^q}{|x-y|^{rq+N}}dy\right)dx+\int_{K}\left(\int_{K\setminus B_1(x)}\frac{|u(x)-u(y)|^q}{|x-y|^{rq+N}}dy\right)dx.
\end{multline}
By Change of variable formula, Fubini's theorem, definition of the Besov seminorm, polar coordinates and the assumption $u\in B^s_{q,\infty}(\R^N,\R^d)$, we have that
\begin{multline}
\label{eq:estimate for Gagliardo expression1}
\int_{K}\left(\int_{K\cap B_1(x)}\frac{|u(x)-u(y)|^q}{|x-y|^{rq+N}}dy\right)dx=\int_{K}\left(\int_{B_1(x)}\chi_{K}(y)\frac{|u(x)-u(y)|^q}{|x-y|^{rq+N}}dy\right)dx
\\
=\int_{K}\left(\int_{B_1(0)}\chi_K(x+z)\frac{|u(x)-u(x+z)|^q}{|z|^{rq+N}}dz\right)dx
=\int_{B_1(0)}\left(\int_{K}\chi_K(x+z)\frac{|u(x)-u(x+z)|^q}{|z|^{rq+N}}dx\right)dz
\\
=\int_{B_1(0)}|z|^{sq-rq-N}\left(\int_{K}\chi_K(x+z)\frac{|u(x)-u(x+z)|^q}{|z|^{sq}}dx\right)dz
\\
\leq [u]^q_{B^s_{q,\infty}(\R^N,\R^d)}\int_{B_1(0)}|z|^{sq-rq-N}dz
=[u]^q_{B^s_{q,\infty}(\R^N,\R^d)}\frac{\mathcal{H}^{N-1}(S^{N-1})}{(s-r)q}<\infty.
\end{multline}
By monotonicity of integral, the convexity of the function $r\longmapsto r^q,r\in[0,\infty),$ and $u\in L^q(K,\R^d)$ we obtain that
\begin{multline}
\label{eq:estimate for Gagliardo expression2}
\int_{K}\left(\int_{K\setminus B_1(x)}\frac{|u(x)-u(y)|^q}{|x-y|^{rq+N}}dy\right)dx\leq \int_{K}\left(\int_{K\setminus B_1(x)}|u(x)-u(y)|^qdy\right)dx
\\
\leq \int_{K}\left(\int_{K}|u(x)-u(y)|^qdy\right)dx\leq \int_{K}\left(2^{q-1}\int_{K}|u(x)|^q+|u(y)|^qdy\right)dx
\\
=\int_{K}\left(2^{q-1}|u(x)|^q\mathcal{L}^N(K)+2^{q-1}\|u\|^q_{L^q(K,\R^d)}\right)dx=2^q\|u\|^q_{L^q(K,\R^d)}\mathcal{L}^N(K)<\infty.
\end{multline}
Thus, we derive from \eqref{eq:decomposition of Gagliardo expression}, \eqref{eq:estimate for Gagliardo expression1} and \eqref{eq:estimate for Gagliardo expression2} that $u\in W^{r,q}_{\text{loc}}(\R^N,\R^d)$.
\end{proof}

\begin{theorem}($\mathcal{H}^{N-1}$-Negligibility of the Jump Set of Fractional Sobolev Functions, Theorem 1.7 in \cite{PA})
\label{thm:negligibility of the jump set for fructional sobolev spaces}

Let $\Omega\subset\R^N$ be an open set, $q\in(1,\infty)$ and $u\in W^{1/q,q}_{\text{loc}}(\Omega,\R^d)$. Then $\mathcal{H}^{N-1}(\mathcal{J}_u)=0$.
\end{theorem}

\begin{corollary}($\mathcal{H}^{N-1}$-Negligibility of the Jump Set of $u\in B^r_{q,\infty},rq>1$)

Let $r\in(0,1)$ and $q\in[1,\infty)$ be such that $rq>1$ and $u\in B^r_{q,\infty}(\R^N,\R^d)$. Then $\mathcal{H}^{N-1}(\mathcal{J}_u)=0$.
\end{corollary}
\begin{proof}
By Lemma \ref{lem:Besov spaces lie in fractional Sobolev} we have $B^r_{q,\infty}(\R^N,\R^d)\subset W^{1/q,q}_{\text{loc}}(\R^N,\R^d)$, so by Theorem \ref{thm:negligibility of the jump set for fructional sobolev spaces} we get $\mathcal{H}^{N-1}(\mathcal{J}_u)=0$.
\end{proof}

\begin{remark}(Functions in $B^r_{q,\infty},rq\leq 1,$ Have Jumps)
If $r\in(0,1)$, $q\in[1,\infty)$ are such that $rq\leq 1$, then, as was proved in Remark \ref{rem: BV and bounded lies in Besov}, $BV(\R^N,\R^d)\cap L^\infty(\R^N,\R^d)\subset B^r_{q,\infty}(\R^N,\R^d)$. Therefore, for functions $u\in B^r_{q,\infty}(\R^N,\R^d)$, the measure of the jump set with respect to Hausdorff measure, $\mathcal{H}^{N-1}(\mathcal{J}_u)$, can be any value in the interval $[0,\infty]$.
\end{remark}

\section{Open questions}
\begin{question}
\label{qu:existence about limit}
Let $1<q<\infty$ and $u\in B^{1/q}_{q,\infty}(\R^N,\R^d)$. Does the following limit exist?
\begin{equation}
\label{qu:existence of limit I}
\lim_{\e\to 0^+}\int_{S^{N-1}}\int_{\R^N}\frac{|u(x+\e n)-u(x)|^q}{\e}dxd\mathcal{H}^{N-1}(n).
\end{equation}
Note that if the limit
\begin{equation}
\lim_{\e\to 0^+}\int_{\R^N}\frac{|u(x+\e n)-u(x)|^q}{\e}dx
\end{equation}
exists for $\mathcal{H}^{N-1}$-almost every $n\in S^{N-1}$, then the limit in \eqref{qu:existence of limit I} exists by Dominated Convergence Theorem:
Since $u\in B^{1/q}_{q,\infty}(\R^N,\R^d)$, then we get by Definition \ref{def:Besov seminorm} that
\begin{equation}
\sup_{n\in S^{N-1}}\sup_{\e\in(0,\infty)}\int_{\R^N}\frac{|u(x+\e n)-u(x)|^q}{\e}dx\leq \left[u\right]^q_{B^{1/q}_{q,\infty}(\R^N,\R^d)}<\infty,
\end{equation}
so by Dominated Convergence Theorem we have the existence of the limit in \eqref{qu:existence of limit I} and
\begin{multline}
\lim_{\e\to 0^+}\int_{S^{N-1}}\int_{\R^N}\frac{|u(x+\e n)-u(x)|^q}{\e}dxd\mathcal{H}^{N-1}(n)
\\
=\int_{S^{N-1}}\lim_{\e\to 0^+}\int_{\R^N}\frac{|u(x+\e n)-u(x)|^q}{\e}dxd\mathcal{H}^{N-1}(n).
\end{multline}
\end{question}
\begin{question}
Let $1<q<\infty$, $u\in B^{1/q}_{q,\infty}(\R^N,\R^d)$, $\eta\in W^{1,1}(\R^N)$. Does the following inequality hold?
\begin{multline}
\label{qu:inqequality containing Gagliardo constant and the q-jump variation}
\liminf_{\e\to
0^+}\frac{1}{|\ln{\e}|}\left[u_\e\right]^q_{W^{1/q,q}(\R^N,\R^d)}
\\
\geq \left|\int_{\R^N}\eta(z)dz\right|^q\left(\int_{S^{N-1}}|z_1|~d\Haus^{N-1}(z)\right)\int_{\mathcal{J}_u}
\Big|u^+(x)-u^-(x)\Big|^q d\mathcal{H}^{N-1}(x).
\end{multline}
\end{question}
\begin{question}
Let $1<q<\infty$, $u\in L^q(\R^N,\R^d)$. Does the following implication hold?
\begin{equation}
\label{qu:implication for Besov property}
\forall\eta\in W^{1,1}(\R^N),\quad \limsup_{\e\to
0^+}\frac{1}{|\ln{\e}|}\left[u_\e\right]^q_{W^{1/q,q}(\R^N,\R^d)}<\infty\quad\Longrightarrow \quad u\in B^{1/q}_{q,\infty}(\R^N,\R^d).
\end{equation}
\end{question}

\begin{theorem}(Theorem 1.3 in \cite{PA})
\label{thm:estimate for the jump q-variation}
Let $1\leq q<\infty$,
$\Omega\subset \mathbb{R}^N$ be an open set and $u\in L^1_{\text{loc}}(\Omega,\R^d)$. Then,
\begin{multline}
\label{eq:inequality30}
\left(\frac{1}{N}\int_{S^{N-1}}|z_1|~d\Haus^{N-1}(z)\right)\int_{\mathcal{J}_u}|u^+(x)-u^-(x)|^qd\Haus^{N-1}(x)
\\
\leq \liminf_{\e\to 0^+}\int_{\Omega}\left(\int_{\Omega\cap B_{\e}(x)}\frac{1}{\e^N}\frac{|u(x)-u(y)|^q}{|x-y|}dy\right)dx.
\end{multline}
\end{theorem}

\begin{remark}
If the limit in \eqref{qu:existence of limit I} exists, then the answer on the other questions is yes: If the limit in  \eqref{qu:existence of limit I} exists, then we get \eqref{qu:inqequality containing Gagliardo constant and the q-jump variation} from equation \eqref{eq:equation110}, Theorem \ref{thm:estimate for the jump q-variation} and Proposition \ref{prop:calculation of NC_N}; and we get \eqref{qu:implication for Besov property} from equation \eqref{eq:equation110} and Theorem \ref{thm:characterization of Besov functions via double integral}.
\end{remark}

\begin{question}
Assume $r\in(0,1)$, $q\in[1,\infty)$ and $u\in B^r_{q,\infty}(\R^N,\R^d)$. Does the following limit hold?
\begin{equation}
\lim_{l\to \infty}\left(\sup_{h\in \R^N\setminus\{0\}}\int_{\Set{x\in\R^N}[|u(x)|>l]}\frac{|u(x+h)-u(x)|^q}{|h|^{rq}}dx\right)=0.
\end{equation}
\end{question}

\section{Appendix}

\subsection{Aspects of Measure Theory}
\begin{lemma}(Countability of Measurable Sets with Finite Measure)

\label{lem: F is at most countable}
Let $(X,\mathcal{E},\sigma)$ be a measure space, which means that $X$ is a set, $\mathcal{E}$ is a sigma-algebra on $X$ and
 $\sigma:\mathcal{E}\to [0,\infty]$ is a measure. Assume that $E\in \mathcal{E}$ is such that $\sigma(E)<\infty$.
Assume $\{E_\alpha\}_{\alpha\in I}$ is a family of sets, where $I$ is a set of indexes, such that for every $\alpha\in I$, $E_\alpha\subset E$,$E_\alpha\in \mathcal{E}$, and $E_\alpha \cap E_{\alpha'}=\emptyset$ for every different $\alpha,\alpha'\in I$. Define the set
\begin{equation}
F:=\bigg\{\alpha\in I:\sigma(E_\alpha)>0\bigg\}.
\end{equation}
Then, $F$ is at most countable.
\end{lemma}

\begin{proof}
Let us decompose $F=\cup_{k\in\N} F_k,F_k:=\left\{\alpha\in I:\sigma(E_\alpha)>\frac{1}{k}\right\}$. For each $k\in \mathbb{N}$ the set $F_k$ is finite. Otherwise, there exists a sequence
$\{\alpha_j\}_{j\in\N}\subset F_k$ of different elements such that
\begin{equation}
\infty>\sigma(E)\geq \sigma\left(\bigcup_{j\in\N} E_{\alpha_j}\right)=\sum_{j\in\N} \sigma(E_{\alpha_j})\geq \sum_{j\in\N}\frac{1}{k}=\infty.
\end{equation}
This contradiction shows that each $F_k$ is a finite set and hence $F$ is at most countable set as a countable union of finite sets.
\end{proof}

\begin{lemma}(The Compact Negligible Boundary Property)
\label{lem:compact negligible boundary property}

Let $(X,d)$ be a locally compact metric space and let $\mu$ be a positive Borel measure on $X$ which is finite on compact sets. Then for every compact set $K\subset X$ there exists a compact set $E\subset X$ such that $K\subset E$ and $\mu(\partial E)=0$.
\end{lemma}
\begin{proof}
Since $K$ is compact and $X$ is locally compact, then there exists an open set $W$ such that $K\subset W$ and $\overline{W}$ is compact, where $\overline{W}$ is the topological closure of $W$. Note that since $\partial W\subset X\setminus W$, then $d\left(\partial W,K\right)\geq d\left(X\setminus W,K\right)$. Since $K$ is compact and $X\setminus W$ is closed and  $K\cap (X\setminus W)=\emptyset$, then $d\left(X\setminus W,K\right)>0$. Therefore, $D:=d\left(\partial W,K\right)>0$. For each $\e\in (0,\infty)$ we define a set
\begin{align}
W_\e:=\bigg\{x\in W:d(x,\partial W)\geq \e\bigg\}.
\end{align}
Note that $K\subset W_\e$ for every $\e\in (0,D)$.
If $\partial W=\emptyset$, then we can choose $E=\overline{W}$, because $E$ is compact and since $\partial E\subset \partial W=\emptyset$, then $\mu(\partial E)=0$. So we can assume that $\partial W\neq\emptyset$. Notice that for a general non-empty set $S\subset X$, the map $f(x):=d(x,S),f:X\to [0,\infty)$ is Lipschitz and so continuous. Thus, the set $W_\e$ is a closed set. Since $W_\e$ is a subset of the compact set $\overline{W}$, then it is compact. For every different $\e,\e'\in (0,D)$ we have $\partial W_\e\cap \partial W_{\e'}=\emptyset$: since $W$ is open and the distance function $f$ is continuous, then $W\cap \big\{x\in X:d(x,\partial W)>\e\big\}$ is an open set, and it is a subset of $W_\e$. Therefore, $\big\{x\in W:d(x,\partial W)>\e\big\}\subset W^{\mathrm{o}}_\e$, where $W^{\mathrm{o}}_\e$ is the topological interior of $W_\e$.  Hence,
\begin{equation}
\partial W_\e=W_\e\setminus W^{\mathrm{o}}_\e\subset W_\e\setminus \big\{x\in W:d(x,\partial W)>\e\big\}=\big\{x\in W:d(x,\partial W)=\e\big\},
\end{equation}
and the sets $\big\{x\in W:d(x,\partial W)=\e\big\}$ are disjoint for different numbers $\e$. Using Lemma \ref{lem: F is at most countable} with the family of sets $\{\partial W_\e\}_{\e\in (0,D)}\subset W$, $\mu(W)<\infty$, we derive the existence of $\e\in (0,D)$ such that $\mu(\partial W_{\e})=0$. We choose $E:=W_\e$.
\end{proof}

\begin{lemma}(The Open Negligible Boundary Property)
\label{lem:the open negligible boundary property}

Let $(X,d)$ be a metric space and let $\mu$ be a finite positive Borel measure on $X$. Let $C\subset X$ be a closed set. Then, there exists a monotone decreasing sequence of open sets $\Omega_k\subset X$ such that for every $k\in\N$ $\mu(\partial \Omega_k)=0$, and $C=\bigcap_{k\in\N}\Omega_k$.
\end{lemma}

\begin{proof}
Define for every $\e\in (0,\infty)$
\begin{equation}
\Omega_\e:=\big\{x\in X:d(x,C)<\e\big\}.
\end{equation}
Assume that $C\neq \emptyset$; if $C=\emptyset$, then we can choose $\Omega_k=\emptyset$. Since the function $x\longmapsto d(x,C)$ is continuous, then $\Omega_\e$ is an open set. We have
\begin{equation}
\partial \Omega_\e=\overline{\Omega}_\e\setminus \Omega_\e\subset\big\{x\in X:d(x,C)\leq\e\big\}\setminus \Omega_\e=\big\{x\in X:d(x,C)=\e\big\}.
\end{equation}
Therefore, for every different $\e_1,\e_2\in(0,\infty)$ we get $\partial \Omega_{\e_1}\cap\partial \Omega_{\e_2}=\emptyset$. Thus, we get by Lemma \ref{lem: F is at most countable} for the family $\{\partial\Omega_\e\}_{\e\in(0,\infty)}$ the existence of an infinitesimal sequence $\e_k\in(0,\infty)$ such that $\mu(\partial\Omega_{\e_k})=0$. Since $C$ is closed we have $C=\bigcap_{k\in\N}\Omega_{\e_k}$.
\end{proof}

\begin{proposition}(Extremal Sets for Essential Infimum and Supremum)
\label{prop:extremal sets for essential infimum and supremum}

Let $X$ be a set and $\mu$ be a positive measure on $X$. Let $f:X\to \R$ be a $\mu$-measurable function. Assume that $K\subset X$ is a set with the following two properties:
\begin{enumerate}
\item $\mu(X\setminus K)=0$;
\item For every $\sigma\in (0,\infty)$ and $x_0\in K$, $\mu\left(\Set{x\in X}[|f(x)-f(x_0)|<\sigma]\right)>0$.
\end{enumerate}
Then,
\begin{equation}
\essinf_{x\in X}f(x)=\inf_{x\in K}f(x),\quad \esssup_{x\in X}f(x)=\sup_{x\in K}f(x).
\end{equation}
We call $K$ an {\it extremal set} for the function $f$.
\end{proposition}

\begin{proof}
Recall that
\begin{equation}
\essinf_{x\in X}f(x):=\sup_{\Theta\subset X,\mu(\Theta)=0}\left(\inf_{x\in X\setminus \Theta}f(x)\right),\quad \esssup_{x\in X}f(x):=\inf_{\Theta\subset X,\mu(\Theta)=0}\left(\sup_{x\in X\setminus \Theta}f(x)\right).
\end{equation}
By property 1 of $K$, we obtain
\begin{equation}
\essinf_{x\in X}f(x)=\essinf_{x\in K}f(x),\quad \esssup_{x\in X}f(x)=\esssup_{x\in K}f(x).
\end{equation}
Let us consider a set $\Theta \subset K$ such that $\mu(\Theta) = 0$. We aim to show that $\inf_{x\in K\setminus \Theta} f(x) = \inf_{x\in K} f(x)$. By taking the supremum over all such $\Theta$, we obtain $\essinf_{x\in K} f(x) = \inf_{x\in K} f(x)$, and hence $\essinf_{x\in X} f(x) = \inf_{x\in K} f(x)$.

It follows from the definition of infimum that $\inf_{x\in K\setminus \Theta} f(x) \geq \inf_{x\in K} f(x)$. Suppose, by contradiction, that $\inf_{x\in K\setminus \Theta} f(x) > \inf_{x\in K} f(x)$. This implies that $\inf_{x\in K} f(x) = \inf_{x\in \Theta} f(x)$. Otherwise, if $\inf_{x\in K} f(x) < \inf_{x\in \Theta} f(x)$, then $\inf_{x\in K} f(x) = \min\big\{\inf_{x\in \Theta} f(x),\inf_{x\in K\setminus \Theta} f(x)\big\} > \inf_{x\in K} f(x)$, which leads to a contradiction.

Therefore, for any $\e\in (0,\infty)$, there exists $x_0 \in \Theta$ such that $f(x_0) - \inf_{x\in K} f(x) < \frac{\e}{2}$. By properties 1,2 of $K$, there exists $y \in K\setminus \Theta$ such that $|f(y) - f(x_0)| < \frac{\e}{2}$. Hence,
\begin{equation}
0<\inf_{x\in K\setminus \Theta}f(x)-\inf_{x\in K}f(x)=\left(f(x_0)-\inf_{x\in K}f(x)\right)+\left(f(y)-f(x_0)\right)+\left(\inf_{x\in K\setminus \Theta}f(x)-f(y)\right)<\e.
\end{equation}
Since $\e$ is arbitrarily small, we arrive at a contradiction, which proves that $\inf_{x\in K\setminus \Theta} f(x) = \inf_{x\in K} f(x)$. The proof of of formula $\esssup_{x\in X}f(x)=\sup_{x\in K}f(x)$ is similar.
\end{proof}

\begin{corollary}(Existence of Extremal Sets for Lebesgue Functions)
\label{cor:existence of extremal sets for Lebesgue functions}

Let $X$ be a metric space, and let $\mu$ be a Borel measure on $X$ such that $0 < \mu(B_r(x)) < \infty$ for every $r \in (0,\infty)$ and every $x \in X$. Suppose $p \in [1,\infty)$ and $f \in L^p(X)$. Then, there exists a set $K \subset X$ with properties 1 and 2 as outlined in Proposition \ref{prop:extremal sets for essential infimum and supremum}. More precisely, the set of Lebesgue points of $f$ possesses these properties.
\end{corollary}

\begin{proof}
Since $f\in L^p(X)$, by the Lebesgue Differentiation Theorem, we know that almost every point in $X$ is a Lebesgue point of $f$ with respect to $\mu$. Let us denote this set by $K$. Therefore, we have property 1: $\mu(X\setminus K)=0$. To establish property 2, let $x_0\in K$ and $\alpha\in (0,1)$. Note that for an arbitrary positive number $\sigma$, there exists $R$ such that

\begin{equation}
\int_{B_R(x_0)}|f(x)-f(x_0)|^pd\mu(x)<\alpha\sigma^p\mu\left(B_R(x_0)\right).
\end{equation}
By Chebyshev's inequality
\begin{equation}
\frac{\mu\left(\Set{x\in B_R(x_0)}[|f(x)-f(x_0)|>\sigma]\right)}{\mu(B_R(x_0))}\leq \frac{1}{\sigma^p}\fint_{B_R(x_0)}|f(x)-f(x_0)|^pd\mu(x)<\alpha.
\end{equation}
Since $f$ is $\mu$-measurable, we obtain
\begin{equation}
\frac{\mu\left(\Set{x\in B_R(x_0)}[|f(x)-f(x_0)|>\sigma]\right)}{\mu(B_R(x_0))}+\frac{\mu\left(\Set{x\in B_R(x_0)}[|f(x)-f(x_0)|\leq\sigma]\right)}{\mu(B_R(x_0))}=1.
\end{equation}
Therefore,
\begin{equation}
\frac{\mu\left(\Set{x\in B_R(x_0)}[|f(x)-f(x_0)|\leq\sigma]\right)}{\mu(B_R(x_0))}\geq 1-\alpha,
\end{equation}
and hence,
\begin{equation}
\mu\left(\Set{x\in X}[|f(x)-f(x_0)|\leq\sigma]\right)\geq (1-\alpha)\mu(B_R(x_0))>0.
\end{equation}
\end{proof}

\subsection{Vector Valued Measures and Variation}
\begin{definition}(Vector Valued Measures and Variation)
\label{def:variation of a measure}

Let $X$ be a set and $\mathcal{E}$ be a $\sigma$-algebra on $X$. Let $\mu:\mathcal{E}\to \R^d$ be a {\it measure}, which means that $\mu(\emptyset)=0$ and for any sequence $\{E_j\}_{j\in\N}\subset\mathcal{E}$ of pairwise disjoint sets we have $\mu\left(\bigcup_{j\in\N}E_j\right)  =\sum_{j\in\N}\mu\left(E_j\right)  $. The {\it variation} of $\mu$ is defined to be
\begin{equation}
\|\mu\|(E):=\sup\bigg\{\sum_{j\in\N}\left|\mu(E_j)\right|:E_j\in\mathcal{E} \text{ pairwise disjoint},E=\bigcup_{j\in\N}E_j\bigg\},\quad E\in\mathcal{E}.
\end{equation}
\end{definition}

\begin{lemma}(Variation of Multiplication of a Vector Valued Function with Positive Measure, Proposition 1.23 in \cite{AFP})
\label{lem: variation of multiplication of vector valued function with positive measure}

Let $\mu$ be a positive measure on the measurable space $(X,\mathcal{E})$, $X$ is a set and $\mathcal{E}$ is a $\sigma$-algebra on $X$. Let $f\in L^1(X,\R^N)$. Then, the variation of the $\R^N$-valued measure
\begin{equation}
f\mu(B):=\int_{B}fd\mu,\quad B\in \mathcal{E}
\end{equation}
satisfies
\begin{equation}
\|f\mu\|(B)=\int_{B}|f|d\mu,\quad B\in \mathcal{E}.
\end{equation}
\end{lemma}

\begin{lemma}(Variation of Multiplication of Scalar Function with Vector Valued Measure)
\label{lem:variation of multiplication of scalar function with vector valued measure}

Let $X$ be a set, $\mathcal{E}$ be a $\sigma$-algebra on $X$ and $\mu:\mathcal{E}\to \R^N$ be a measure. Let $f:X\to \R$ be such that $f\in L^1(X,\|\mu\|)$. Then,
\begin{equation}
\|f\mu\|(E)\leq N^{1/2}|f|\|\mu\|(E),\quad E\in \mathcal{E}.
\end{equation}
\end{lemma}

\begin{proof}
Let us denote $\mu:=(\mu_1,...,\mu_N)$. For every $E\in\mathcal{E}$
\begin{multline}
|f\mu(E)|=\left|(f\mu_1(E),...,f\mu_N(E))\right|=\left(\sum_{i=1}^N(f\mu_i(E))^2\right)^{1/2}
\\
\leq \left(\sum_{i=1}^N(\|f\mu_i\|(E))^2\right)^{1/2}
\leq \left(\sum_{i=1}^N(|f|\|\mu_i\|(E))^2\right)^{1/2}\leq N^{1/2}|f|\|\mu\|(E).
\end{multline}
Therefore,
\begin{multline}
\|f\mu\|(E)=\sup\bigg\{\sum_{j\in\N}|f\mu(E_j)|:E_j\in\mathcal{E} \text{ pairwise disjoint},E=\cup_{j\in\N}E_j\bigg\}\\
\leq N^{1/2}\sup\bigg\{\sum_{j\in\N}|f|\|\mu\|(E_j):E_j\in\mathcal{E} \text{ pairwise disjoint},E=\cup_{j\in\N} E_j\bigg\}
=N^{1/2}|f|\|\mu\|(E).
\end{multline}
\end{proof}

\subsection{{Aspects of Integration on $S^{N-1}$ with respect to $\mathcal{H}^{N-1}$}}
\begin{proposition}
\label{prop:independence of integrals on the sphere}
For every $v_1,v_2\in S^{N-1}$ we have
\begin{equation}
\int_{S^{N-1}}\left|v_1\cdot n\right|d\mathcal{H}^{N-1}(n)=\int_{S^{N-1}}\left|v_2\cdot n\right|d\mathcal{H}^{N-1}(n).
\end{equation}
\end{proposition}

\begin{proof}
Take an isometry $A:\mathbb{R}^N\to \mathbb{R}^N$ such that $A(v_2)=v_1$.
Then,
\begin{equation}
\int_{S^{N-1}}\left|v_1\cdot n\right|d\mathcal{H}^{N-1}(n)=\int_{A^{-1}\left(S^{N-1}\right)}\left|A(v_2)\cdot A(w)\right|d\mathcal{H}^{N-1}(w)=\int_{S^{N-1}}\left|v_2\cdot w\right|d\mathcal{H}^{N-1}(w).
\end{equation}
\end{proof}

\begin{proposition}
\label{prop:calculation of NC_N}
It follows that
\begin{equation}
\label{eq:equation58}
\int_{\R^{N-1}}\frac{2dv}{\big(\sqrt{1+|v|^2}\big)^{N+1}}=\int_{S^{N-1}}|z_1|d\mathcal{H}^{N-1}(z),\quad z=(z_1,...,z_N).
\end{equation}
\end{proposition}
\begin{proof}
Note that \eqref{eq:equation58} holds for $N=1$. So we can assume that $N>1$. Let $B^{N-1}_1(0)$ be the ball of radius $1$ around the origin in $\R^{N-1}$. Define
\begin{align}
g:B^{N-1}_1(0)\to \R^N,\quad  g(z_2,...,z_N):=\left(f(z_2,...,z_N),z_2,...,z_N\right),\quad f(z_2,...,z_N):=\sqrt{1-\sum_{j=2}^{N}z^2_j}.
\end{align}
The image of $g$ is $S^+:=\big\{z=(z_1,...,z_N)\in S^{N-1}:z_1>0\big\}$. Denote $z=(z_1,z'),z':=(z_2,...,z_N)$.
By the area formula
\begin{multline}
\label{eq:equation67}
\int_{S^{N-1}}|z_1|d\mathcal{H}^{N-1}(z)=2\int_{S^+}|z_1|d\mathcal{H}^{N-1}(z)
=2\int_{B^{N-1}_1(0)}\sqrt{1-\sum_{j=2}^{N}z^2_j}\,\sqrt{1+|\nabla f(z')|^2}\,d\mathcal{L}^{N-1}(z')
\\
=2\int_{B^{N-1}_1(0)}\sqrt{1-\sum_{j=2}^{N}z^2_j}\,\sqrt{1+\frac{1}{1-\sum_{j=2}^{N}z^2_j}\sum_{j=2}^{N}z^2_j}\,d\mathcal{L}^{N-1}(z')
=2\Big(\mathcal{L}^{N-1}\big(B^{N-1}_1(0)\big)\Big).
\end{multline}
In addition, by polar coordinates we obtain for $N\geq 2$
\begin{equation}
\label{eq:equation63}
\int_{\R^{N-1}}\frac{dv}{\big(\sqrt{1+|v|^2}\big)^{N+1}}=\Big(\mathcal{H}^{N-2}(S^{N-2})\Big)\int_{0}^{\infty}\frac{r^{N-2}dr}{\big(\sqrt{1+r^2}\big)^{N+1}}.
\end{equation}
Let us denote
\begin{equation}
\label{eq:equation64}
A_N:=\int_{0}^{\infty}\frac{r^{N-2}dr}{\big(\sqrt{1+r^2}\big)^{N+1}}.
\end{equation}
Assume for the moment that $N>3$. Integration by parts gives
\begin{align}
\label{eq:equation59}
A_N&=\int_{0}^{\infty}r^{N-3}\frac{2r}{2\big(1+r^2\big)^{\frac{N+1}{2}}}dr\nonumber
\\
&=\left(r^{N-3}\frac{1}{(1-N)\big(1+r^2\big)^{\frac{N-1}{2}}}\right)\Big|_{r=0}^{r=\infty}-\int_0^\infty (N-3)r^{N-4}\frac{1}{(1-N)\big(1+r^2\big)^{\frac{N-1}{2}}}dr\nonumber
\\
&=\frac{N-3}{N-1}\int_{0}^{\infty}r^{N-4}\frac{1}{\big(1+r^2\big)^{\frac{N-1}{2}}}dr=\frac{N-3}{N-1}A_{N-2}.
\end{align}
We got a recursive sequence. Note for example
\begin{multline}
A_N=\frac{N-3}{N-1}A_{N-2}=\frac{N-3}{N-1}\frac{N-5}{N-3}A_{N-4}=\frac{N-5}{N-1}A_{N-4}=\frac{N-5}{N-1}\frac{N-7}{N-5}A_{N-6}=\frac{N-7}{N-1}A_{N-6}.
\end{multline}
Therefore, we get from \eqref{eq:equation59} for every natural $m>1$
\begin{equation}
\label{eq:equation60}
A_{2m}=\frac{1}{2m-1}A_{2}\quad\text{and}\quad
A_{2m+1}=\frac{2}{2m}A_{3}\,.
\end{equation}
Let us calculate $A_2,A_3$ separately. Note that
\begin{equation}
\label{eq:equation61}
A_3:=\int_{0}^{\infty}\frac{2rdr}{2\big(1+r^2\big)^{2}}=-\frac{1}{2\big(1+r^2\big)}\bigg|_{r=0}^{r=\infty}=\frac{1}{2}.
\end{equation}
Let us prove that
\begin{equation}
\label{eq:equation62}
A_2:=\int_{0}^{\infty}\frac{1}{\big(\sqrt{1+r^2}\big)^{3}}dr=1.
\end{equation}
Changing variables $r=\frac{z}{2}-\frac{1}{2z}$ in the last integral
gives:
\begin{multline}
A_2=\int_{1}^{\infty}\frac{1}{\bigg(\sqrt{1+\Big(\frac{z}{2}-\frac{1}{2z}\Big)^2}\bigg)^{3}}\Big(\frac{1}{2}+\frac{1}{2z^2}\Big)dz
=\int_{1}^{\infty}\frac{1}{\bigg(\sqrt{1+\frac{1}{4z^2}\Big(z^2-1\Big)^2}\bigg)^{3}}\frac{1}{2z^2}\Big(z^2+1\Big)dz
\\
=\int_{1}^{\infty}\frac{1}{\bigg(\sqrt{4z^2+\Big(z^2-1\Big)^2}\bigg)^{3}}4z\Big(z^2+1\Big)dz
=\int_{1}^{\infty}\frac{1}{\bigg(\sqrt{\Big(z^2+1\Big)^2}\bigg)^{3}}4z\Big(z^2+1\Big)dz
\\
=\int_{1}^{\infty}\frac{4z}{(z^2+1)^{2}}dz
=-\frac{2}{z^2+1}\bigg|_{z=1}^{z=\infty}
=1.
\end{multline}
Therefore, by \eqref{eq:equation60}, \eqref{eq:equation61} and \eqref{eq:equation62}
\begin{equation}
A_{2m}=\frac{1}{2m-1}\quad\text{and}\quad A_{2m+1}=\frac{1}{2m}.
\end{equation}
Thus, for every natural $N>1$
\begin{equation}
\label{eq:equation65}
A_{N}=\frac{1}{N-1}\,.
\end{equation}
Therefore, by \eqref{eq:equation63}, \eqref{eq:equation64}, \eqref{eq:equation65} and polar coordinates we get for every $N>1$
\begin{equation}
\label{eq:equation66}
\int_{\R^{N-1}}\frac{dv}{\big(\sqrt{1+|v|^2}\big)^{N+1}}=\Big(\mathcal{H}^{N-2}(S^{N-2})\Big)\frac{1}{N-1}=\mathcal{L}^{N-1}\big(B^{N-1}_1(0)\big).
\end{equation}
Thus, by \eqref{eq:equation67} and \eqref{eq:equation66} we get \eqref{eq:equation58}.
\end{proof}

\begin{proposition}(Polar coordinates, see 3.4.4 in \cite{EG})\\
\label{prop:polar coordinates}
Let $g\in L^1\left(\R^N,\R^d\right)$. Then
\begin{equation}
\int_{\R^N}g(x)dx=\int_{0}^\infty\left(\int_{\partial B_r(0)}g(z)d\mathcal{H}^{N-1}(z)\right)dr=\int_{0}^\infty r^{N-1}\left(\int_{S^{N-1}}g(rz)d\mathcal{H}^{N-1}(z)\right)dr.
\end{equation}
\end{proposition}

\subsection{Sequences of Real Numbers}
\begin{lemma}(Liminfsup Lemma)
\label{lem:liminfsup lemma}

Let $\{a_k\}_{k=1}^\infty,\{b_k\}_{k=1}^\infty\subset\R$ be bounded sequences. Then,
\begin{equation}
\max\left\{|\liminf_{k\to\infty}a_k-\liminf_{k\to\infty}b_k|,|\limsup_{k\to\infty}a_k-\limsup_{k\to\infty}b_k|\right\}\leq \limsup_{k\to\infty}|a_k-b_k|.
\end{equation}
\end{lemma}

\begin{proof}
Recall the general inequalities:
\begin{equation}
\label{eq:upper limit of a sum is less or equal to the sum of upper limits}
\limsup_{k\to\infty}(a_k+b_k)\leq \limsup_{k\to\infty}a_k+\limsup_{k\to\infty}b_k,
\end{equation}
\begin{equation}
\label{eq:lower limit of a sum is less or equal to the sum of upper limit and lower limit}
\liminf_{k\to\infty}(a_k+b_k)\leq \limsup_{k\to\infty}a_k+\liminf_{k\to\infty}b_k.
\end{equation}
By \eqref{eq:upper limit of a sum is less or equal to the sum of upper limits} we get
\begin{equation}
\limsup_{k\to\infty}a_k=\limsup_{k\to\infty}(a_k-b_k+b_k)\leq \limsup_{k\to\infty}(a_k-b_k)+\limsup_{k\to\infty}b_k.
\end{equation}
Changing the roles of $a_k$ and $b_k$, we get
\begin{equation}
\left|\limsup_{k\to\infty}a_k-\limsup_{k\to\infty}b_k\right|\leq \limsup_{k\to\infty}|a_k-b_k|.
\end{equation}
By \eqref{eq:lower limit of a sum is less or equal to the sum of upper limit and lower limit} we get
\begin{equation}
\liminf_{k\to\infty}a_k=\liminf_{k\to\infty}(a_k-b_k+b_k)\leq \limsup_{k\to\infty}(a_k-b_k)+\liminf_{k\to\infty}b_k.
\end{equation}
Changing the roles of $a_k$ and $b_k$, we get
\begin{equation}
\left|\liminf_{k\to\infty}a_k-\liminf_{k\to\infty}b_k\right|\leq \limsup_{k\to\infty}|a_k-b_k|.
\end{equation}
\end{proof}

\subsection{The Truncated Family}
\begin{definition}(Truncated Family)
\label{def:truncated family}

Let $E\subset\R^N$ be a set and let $u:E\to \R^d,u=(u^1,...,u^d)$ be a function. For every $1\leq i\leq d,i\in\N$, $l\in[0,\infty)$ and $x\in E$ we define $u^i_l(x):=l\wedge(-l\vee u^i(x))$, where $a\wedge b:=\min\{a,b\}$, $a\vee b:=\max\{a,b\}$, for $a,b\in\R$; and we define $u_l(x):=(u^1_l(x),...,u^d_l(x))$. We call the family of functions $\{u_l\}_{l\in[0,\infty)}$ the {\it truncated family obtained by $u$}.
\end{definition}

\begin{proposition}(Properties of the Truncated Family)
\label{prop:properties of the truncated family}

Let $E\subset\R^N$ be a set and let $u:E\to \R^d,u=(u^1,...,u^d)$ be a function. Let $\{u_l\}_{l\in[0,\infty)}$ be the truncated family obtained by $u$. Then,
\\
1. $\lim_{l\to\infty}u_l(x)=u(x),\quad \forall x\in E$;
\\
2. For every $x,y\in E$ and $l,m\in[0,\infty),l\leq m,$ we have $|u_l(x)-u_l(y)|\leq |u_{m}(x)-u_{m}(y)|\leq |u(x)-u(y)|$;
\\
3. For every $x,y\in E$, the family $\{|u_l(x)-u_l(y)|\}_{l\in[0,\infty)}$ is monotone increasing to $|u(x)-u(y)|$.
\end{proposition}

\begin{proof}
For $x\in \mathbb{R},l\in[0,\infty)$ we define $x_l:=l\wedge(-l\vee x)$. Notice that for every $x,y\in \mathbb{R}$ and $l,m\in[0,\infty),l<m,$ we have
$|x_l-y_l|\leq |x_{m}-y_{m}|$. For a point $x=(x^1,...,x^d)\in \mathbb{R}^d$, and $l\in[0,\infty)$ we define $x_l:=(x_l^1,...,x_l^d)$. Similarly, we have for every $x,y\in \mathbb{R}^d$ and $l,m\in[0,\infty),l<m,$ the inequality $|x_l-y_l|\leq |x_{m}-y_{m}|$. Notice also that for $x\in \mathbb{R}^d$, the family $\{x_l\}_{l\in[0,\infty)}$ has the property $\lim_{l\to \infty}x_l=x$. In particular, for every $x,y\in \mathbb{R}^d$, the family $\{|x_l-y_l|\}_{l\in[0,\infty)}$ is monotone increasing to $|x-y|$. Therefore, we get items 1,2 and 3 by choosing the points $u(x),u(y)$ in place of the points $x,y$.
\end{proof}

\subsection{Approximate Continuity and Differentiability of $L^1_{\text{loc}}$-functions}
\begin{definition}(Approximate Limit)
\label{def:approximate limit}

Let $\Omega\subset\mathbb{R}^N$ be an
open set and $u\in L^1_{\text{loc}}(\Omega,\mathbb{R}^d)$. We say that $u$
has approximate limit at $x\in \Omega$ if and only if there exists
$z\in \mathbb{R}^d$ such that
\begin{equation}
\label{eq:approximate limit} \lim_{\rho\to
0^+}\fint_{B_\rho(x)}|u(y)-z|dy=0.
\end{equation}
The set $\mathcal{S}_u$ of points where this property does not hold
is called the {\it approximate discontinuity set}. For any $x\in \Omega$
the point $z$, uniquely determined by $\eqref{eq:approximate
limit}$, is called the {\it approximate limit of $u$ at $x$} and denoted by
$\tilde{u}(x)$.
\end{definition}

\begin{definition}(Approximate Jump Points)
\label{def:approximate jump points}

Let
$\Omega\subset\mathbb{R}^N$ be an open set, $u\in
L^1_{\text{loc}}(\Omega,\mathbb{R}^d)$ and $x\in \Omega$. We say that $x$
is an {\it approximate jump point} of $u$ if and only if there
exist different $a,b\in \mathbb{R}^d$ and $\nu\in S^{N-1}$ such
that
\begin{equation}
\label{eq:one-sided approximate limit}
\lim_{\rho\to
0^+}\frac{1}{\rho^N}\left(\int_{B^+_\rho(x,\nu)}|u(z)-a|dz+\int_{B^-_\rho(x,\nu)}|u(z)-b|dz\right)=0,
\end{equation}
where
\begin{equation}
B^+_\rho(x,\nu):=\big\{y\in B_\rho(x):(y-x)\cdot
\nu>0\big\},\quad B^-_\rho(x,\nu):=\big\{y\in
B_\rho(x):(y-x)\cdot \nu<0\big\}.
\end{equation}
The triple $(a,b,\nu)$, uniquely determined by $\eqref{eq:one-sided
approximate limit}$ up to a permutation of $(a,b)$ and the change
of sign of $\nu$, is denoted by $(u^+(x),u^-(x),\nu_u(x))$. The set
of approximate jump points is denoted by $\mathcal{J}_u$. Note that $\mathcal{J}_u\subset\mathcal{S}_u$.
\end{definition}

\begin{definition}(Approximate Differentiability, definition 3.70 in \cite{AFP})
\label{def:approximate differentiability points}

Let $\Omega\subset\R^N$ be an open set and let $u\in L^1_{\text{loc}}(\Omega,\R^d)$. Let $x\in\Omega\setminus \mathcal{S}_u$. We say that $u$ is approximately differentiable at $x$ if there exists a $d\times N$ matrix $L$ such that
\begin{equation}
\label{eq:approximate differentiability}
\lim_{\rho\to 0^+}\fint_{B_\rho(x)}\frac{|u(y)-\tilde{u}(x)-L(y-x)|}{\rho}dy=0.
\end{equation}
If $u$ is approximately differentiable at $x$, the matrix $L$, uniquely determined by \eqref{eq:approximate differentiability}, is called the {\it approximate differential} of $u$ at $x$ and denoted by $\nabla u(x)$. The set of approximate differentiability points of $u$ is denoted by $\mathcal{D}_u$.
\end{definition}

\begin{proposition}(Properties of Approximate Differential, Proposition 3.71 in \cite{AFP})
\label{prop:properties of approximate differential}

Let $\Omega\subset\R^N$ be an open set and let $u\in L^1_{\text{loc}}(\Omega,\R^d)$. Then, $\mathcal{D}_u$ is a Borel set and the map $\nabla u:\mathcal{D}_u\to \R^{dN}$ is a Borel map.
\end{proposition}

\begin{proposition}(Locality Properties of Approximate Differential, Proposition 3.73 in \cite{AFP})

Let $\Omega\subset\R^N$ be an open set, $u,v\in L^1_{\text{loc}}(\Omega,\R^d)$. If $x\in\mathcal{D}_u\cap \mathcal{D}_v$ and the set $\{u=v\}$ has density $1$ at $x$, then $\nabla u(x)=\nabla v(x)$. In particular, $\nabla u(x)=\nabla v(x)$ for $\mathcal{L}^N$-almost every $x\in \{u=v\}\cap \mathcal{D}_u\cap \mathcal{D}_v$.
\end{proposition}

\begin{proposition}(Properties of Approximate Limits, Proposition 3.64 in \cite{AFP})
\label{prop:properties of approximate limits}

Let $\Omega\subset\mathbb{R}^N$ be an open set and $u\in L^1_{\text{loc}}(\Omega,\mathbb{R}^d)$.
\\
$(a)$ $\mathcal{S}_u$ is a Borel set, $\mathcal{L}^N(\mathcal{S}_u)=0$ and $\tilde{u}:\Omega\setminus \mathcal{S}_u\to \mathbb{R}^d$ is a Borel function, coinciding $\mathcal{L}^N-$almost everywhere in $\Omega\setminus \mathcal{S}_u$ with $u$;
\\
$(b)$ if $f:\mathbb{R}^d\to \mathbb{R}^p$ is a Lipschitz map and $v=f\circ u$, then $\mathcal{S}_v\subset \mathcal{S}_u$ and $\tilde{v}(x)=f(\tilde{u}(x))$ for any $x\in \Omega \setminus \mathcal{S}_u$.
\end{proposition}

\begin{proposition}(Properties of One-Sided Approximate Limits, Proposition 3.69 in \cite{AFP})
\label{prop:properties of one-sided approximate limits}

Let $\Omega\subset\mathbb{R}^N$ be an open set and $u\in L^1_{\text{loc}}(\Omega,\mathbb{R}^d)$.
\\
(a) The set $\mathcal{J}_u$ is a Borel subset of $\mathcal{S}_u$ and there exist Borel functions
\begin{equation}
\left(u^+,u^-,\nu_u\right):\mathcal{J}_u\to \R^d\times\R^d\times S^{N-1}
\end{equation}
such that for every $x\in \mathcal{J}_u$ we have
\begin{equation}
\lim_{\rho\to 0^+}\fint_{B^+_\rho(x,\nu_u(x))}|u(y)-u^+(x)|dy=0,\quad \lim_{\rho\to 0^+}\fint_{B^-_\rho(x,\nu_u(x))}|u(y)-u^-(x)|dy=0.
\end{equation}
\\
(b) if $f:\mathbb{R}^d\to \mathbb{R}^p$ is a Lipschitz map, $v=f\circ u$ and $x\in \mathcal{J}_u$, then $x\in \mathcal{J}_v$ if and only if $f(u^+(x))\neq f(u^-(x))$, and in this case
\begin{equation}
\left(v^+(x),v^-(x),\nu_{v}(x)\right)=\left(f(u^+(x)),f(u^-(x)),\nu_{u}(x)\right).
\end{equation}
Otherwise, $x\notin \mathcal{S}_v$ and $\tilde{v}(x)=f(u^+(x))=f(u^-(x))$.
\end{proposition}

\begin{proposition}(Truncation and Jumps)
\label{prop:truncation and jumps}

Let $\Omega\subset\R^N$ be an open set, and let $u\in L^1_{\text{loc}}(\Omega,\R^d)$. For each $l\in[0,\infty)$, let us define the $l$-truncated function by $T_l:\R^d\to \R^d,T_l(x):=x_l$, where $x_l$ is defined as in the proof of Proposition \ref{prop:properties of the truncated family}. Then we have the following assertions:
\\
1. $T_l$ is a Lipschitz map;
\\
2. The jumps set of $u$ can be decomposed in terms of the jump sets of $T_l\circ u$ through the formula:
\begin{equation}
\mathcal{J}_u=\bigcup_{l\in[0,\infty)}\mathcal{J}_{T_l\circ u}\cap \mathcal{J}_u;
\end{equation}
3. For every $l,m\in[0,\infty)$ such that $l\leq m$ we have the following monotonicity property:
\begin{equation}
\mathcal{J}_{T_l\circ u}\cap \mathcal{J}_u\subset \mathcal{J}_{T_m\circ u}\cap \mathcal{J}_u.
\end{equation}
\end{proposition}
\begin{proof}
1. For each $l\in[0,\infty)$, by Proposition $\ref{prop:properties of the truncated family}$ we get that the map $T_l:\mathbb{R}^d\to \mathbb{R}^d,T_l(x):=x_l$, is Lipschitz.
\\
2. For $u\in L^1_{\text{loc}}(\Omega,\R^d)$, where $\Omega\subset\mathbb{R}^N$ is an open set, and $x\in \mathcal{J}_u$, we know by Proposition $\ref{prop:properties of one-sided approximate limits}$ that $x\in \mathcal{J}_{T_l\circ u}$ if and only if $T_l(u^+(x))\neq T_l(u^-(x))$, and in this case
\begin{equation}
\left((T_l\circ u)^+(x),(T_l\circ u)^-(x),\nu_{T_l\circ u}(x)\right)=\left(T_l(u^+(x)),T_l(u^-(x)),\nu_{u}(x)\right);
\end{equation}
and if $T_l(u^+(x))=T_l(u^-(x))$, then $x\notin \mathcal{S}_{T_l\circ u}$. Thus, since for every $x\in \mathcal{J}_u$ there exists a big enough $l\in[0,\infty)$ such that $T_l(u^+(x))=u^+(x)\neq u^-(x)=T_l(u^-(x))$, we have
\begin{equation}
\mathcal{J}_u=\bigcup_{l\in[0,\infty)} \mathcal{J}_{T_l\circ u}\cap \mathcal{J}_u.
\end{equation}
3. we have for every $l,m\in[0,\infty),l\leq m,$ that $\mathcal{J}_{T_l\circ u}\cap \mathcal{J}_u\subset \mathcal{J}_{T_m\circ u}\cap \mathcal{J}_u$: If
$x\in \mathcal{J}_{T_l\circ u}\cap \mathcal{J}_u$, then $T_l(u^+(x))\neq T_l(u^-(x))$ and so $T_{m}(u^+(x))\neq T_{m}(u^-(x))$. If not, then  $T_{m}(u^+(x))=T_{m}(u^-(x))$ and then $x\notin \mathcal{S}_{T_m\circ u}$, and since $T_l\circ \left(T_m\circ u\right)=T_l\circ u$, then, by part (b) of Proposition
$\ref{prop:properties of approximate limits}$ with $T_l$ in place of $f$ and $T_m\circ u$ in place of $u$, we obtain $\mathcal{S}_{T_l\circ u}\subset \mathcal{S}_{T_m\circ u}$ and so
$x\notin \mathcal{S}_{T_l\circ u}$. It is a contradiction since $x\in \mathcal{J}_{T_l\circ u}\subset \mathcal{S}_{T_l\circ u}$. From $T_{m}(u^+(x))\neq T_{m}(u^-(x))$ and $x\in \mathcal{J}_u$ we get $x\in \mathcal{J}_{T_m\circ u}\cap \mathcal{J}_u$.
\end{proof}

\begin{lemma}(Lower Semi-Continuity for Jump-Integral with respect to the Truncated Family)
Let $\Omega\subset\R^N$ be an open set, $u\in L^1_{\text{loc}}(\Omega,\R^d)$, $h:\R^N\to \R$ be a non-negative, $\mathcal{H}^{N-1}$-measurable function and $F:\R\to\R$ be a non-negative continuous function. Then,
\begin{align}
\label{eq:lower semi-continuity for q-jump variation w.r.t truncated family}
\liminf_{l\to \infty}\int_{\mathcal{J}_{u_{l}}}F\left(|(u_{l})^+(x)-(u_{l})^-(x)|\right)h(x)d\Haus^{N-1}(x)\geq \int_{\mathcal{J}_u}F\left(|u^+(x)-u^-(x)|\right)h(x)d\Haus^{N-1}(x),
\end{align}
where $\{u_l\}_{l\in[0,\infty)}$ is the truncated family obtained by $u$.
\end{lemma}

\begin{proof}
By Proposition $\ref{prop:truncation and jumps}$ we obtain
\begin{multline}
\label{eq:lower semi-continuity for q-jump variation w.r.t truncated family1}
\int_{\mathcal{J}_{u_l}}F\left(|(u_l)^+(x)-(u_l)^-(x)|\right)h(x)d\Haus^{N-1}(x)\geq \int_{\mathcal{J}_{u_l}\cap \mathcal{J}_u}F\left(|(u_l)^+(x)-(u_l)^-(x)|\right)h(x)d\Haus^{N-1}(x)
\\
=\int_{\mathcal{J}_{u_l}\cap \mathcal{J}_u}F\left(|(u^+(x))_l-(u^-(x))_l|\right)h(x)d\Haus^{N-1}(x)
\\
=\int_{\mathcal{J}_u}\chi_{\mathcal{J}_{u_l}\cap \mathcal{J}_u}(x)F\left(|(u^+(x))_l-(u^-(x))_l|\right)h(x)d\Haus^{N-1}(x).
\end{multline}
By Proposition
$\ref{prop:truncation and jumps}$  we have
$\lim_{l\to \infty}\chi_{\mathcal{J}_{u_l}\cap \mathcal{J}_u}(x)=\chi_{\mathcal{J}_u}(x),\forall x\in \mathcal{J}_u$, and by Proposition $\ref{prop:properties of the truncated family}$ we have $\lim_{l\to \infty}|(u^+(x))_l-(u^-(x))_l|=|u^+(x)-u^-(x)|,\forall x\in \mathcal{J}_u$.
Taking the lower limit as
$l\to \infty$ on both sides of \eqref{eq:lower semi-continuity for q-jump variation w.r.t truncated family1} and using Fatou's lemma we obtain \eqref{eq:lower semi-continuity for q-jump variation w.r.t truncated family}.
\end{proof}

\subsection{Aspects of $BV$-Functions}
\begin{definition}(Definition of $BV$ Functions)
Let $\Omega\subset\R^N$ be an open set. We say that $u\in BV(\Omega,\R^d)$ if and only if $u\in L^1(\Omega,\R^d)$ and there exists an $d\times N$ matrix valued measure $\mu:\mathcal{B}(\Omega)\footnote{Borel sigma algebra}\to \R^{d\times N}$ such that for every $\varphi \in C^\infty_c(\Omega)$ it follows that
\begin{equation}
\label{eq:integration by parts for BV}
\int_{\Omega}u(x)\nabla \varphi(x)dx=-\int_{\Omega}\varphi(x)d\mu(x).
\end{equation}
In this case we denote $\mu:=Du$. In formula \eqref{eq:integration by parts for BV} we think about $u$ as a column vector $u=(u_1,...,u_d)^T$ and $\nabla \varphi=(\partial_1 \varphi,...,\partial_N \varphi)$.
\end{definition}
\begin{lemma}(Continuity for Jump-Integral with respect to the Truncated Family)
\label{lem:limit of jump variation of truncated family}

Let $\Omega\subset\R^N$ be an open set, $u\in BV_{\text{loc}}(\Omega,\R^d)$. Let $h:\R^N\to \R$ be a non-negative $\mathcal{H}^{N-1}$-measurable function, and $F:\R\to \R$ be a non-negative, monotone increasing function. Let $\{u_l\}_{l\in[0,\infty)}$ be the truncated family obtained by $u$. Then,
\begin{enumerate}
\item
\begin{align}
\label{eq:limit of jump variation of truncated family}
\lim_{l\to \infty}\int_{\mathcal{J}_{u_{l}}}F\left(|(u_{l})^+(x)-(u_{l})^-(x)|\right)h(x)d\Haus^{N-1}(x)= \int_{\mathcal{J}_u}F\left(|u^+(x)-u^-(x)|\right)h(x)d\Haus^{N-1}(x).
\end{align}

\item For every $n\in\R^N$
\begin{multline}
\label{eq:limit of jump variation of truncated family with the jump normal}
\lim_{l\to \infty}\int_{\mathcal{J}_{u_{l}}}F\left(|(u_{l})^+(x)-(u_{l})^-(x)|\right)|\nu_{u_{l}}(x)\cdot n|h(x)d\Haus^{N-1}(x)
\\
=\int_{\mathcal{J}_u}F\left(|u^+(x)-u^-(x)|\right)|\nu_{u}(x)\cdot n|h(x)d\Haus^{N-1}(x).
\end{multline}
\end{enumerate}
\end{lemma}

\begin{proof}
Let us prove assertion 1. Since $u\in BV_{\text{loc}}(\Omega,\R^d)$, then for every $l\in [0,\infty)$ we have by chain rule for $BV$-functions (refer to Theorem \ref{thm:chain rule in BV}) that $u_l\in BV_{\text{loc}}(\Omega,\R^d)$, and by Federer-Vol'pert theorem (refer to Theorem \ref{thm: Federer-Vol'pert}) we have $\mathcal{H}^{N-1}\left(\mathcal{S}_{u}\setminus \mathcal{J}_u\right)=\mathcal{H}^{N-1}\left(\mathcal{S}_{u_l}\setminus \mathcal{J}_{u_l}\right)=0$. Therefore,
$\mathcal{H}^{N-1}\left(\mathcal{J}_{u_l}\setminus \mathcal{J}_u\right)=\mathcal{H}^{N-1}\left(\mathcal{S}_{u_l}\setminus \mathcal{S}_u\right)=0$, because
$\mathcal{S}_{u_l}\subset \mathcal{S}_u$. Therefore, by item $(b)$ of Proposition \ref{prop:properties of one-sided approximate limits} we get
\begin{multline}
\label{eq:expression for jump variation of truncated family}
\int_{\mathcal{J}_{u_l}}F\left(|(u_l)^+(x)-(u_l)^-(x)|\right)h(x)d\Haus^{N-1}(x)=\int_{\mathcal{J}_{u_l}\cap \mathcal{J}_u}F\left(|(u_l)^+(x)-(u_l)^-(x)|\right)h(x)d\Haus^{N-1}(x)
\\
+\int_{\mathcal{J}_{u_l}\setminus \mathcal{J}_u}F\left(|(u_l)^+(x)-(u_l)^-(x)|\right)h(x)d\Haus^{N-1}(x)
=\int_{\mathcal{J}_{u_l}\cap \mathcal{J}_u}F\left(|(u^+(x))_l-(u^-(x))_l|\right)h(x)d\Haus^{N-1}(x)
\\
=\int_{\mathcal{J}_u}\chi_{\mathcal{J}_{u_l}\cap \mathcal{J}_u}(x)F\left(|(u^+(x))_l-(u^-(x))_l|\right)h(x)d\Haus^{N-1}(x).
\end{multline}
By Proposition \ref{prop:properties of the truncated family}, Proposition \ref{prop:truncation and jumps} and monotone convergence theorem we get \eqref{eq:limit of jump variation of truncated family} by taking the limit as $l\to\infty$ on both sides of \eqref{eq:expression for jump variation of truncated family}.
\\
For assertion 2, note that, by item $(b)$ of Proposition \ref{prop:properties of one-sided approximate limits} we get
\begin{multline}
\int_{\mathcal{J}_{u_{l}}}F\left(|(u_{l})^+(x)-(u_{l})^-(x)|\right)|\nu_{u_{l}}(x)\cdot n|h(x)d\Haus^{N-1}(x)
\\
=\int_{\mathcal{J}_{u_{l}}\cap \mathcal{J}_{u}}F\left(|(u_{l})^+(x)-(u_{l})^-(x)|\right)|\nu_{u_{l}}(x)\cdot n|h(x)d\Haus^{N-1}(x)
\\
+\int_{\mathcal{J}_{u_{l}}\setminus \mathcal{J}_{u}}F\left(|(u_{l})^+(x)-(u_{l})^-(x)|\right)|\nu_{u_{l}}(x)\cdot n|h(x)d\Haus^{N-1}(x)
\\
=\int_{\mathcal{J}_{u_{l}}}F\left(|(u_{l})^+(x)-(u_{l})^-(x)|\right)|\nu_{u}(x)\cdot n|\chi_{\mathcal{J}_{u}}(x)h(x)d\Haus^{N-1}(x).
\end{multline}
Using item 1 with $|\nu_{u}(x)\cdot n|\chi_{\mathcal{J}_{u}}(x)h(x)$ in place of $h(x)$, we conclude \eqref{eq:limit of jump variation of truncated family with the jump normal}.
\end{proof}

\begin{theorem}(Calder{\'o}n-Zygmund, Theorem 3.83 in \cite{AFP})
\label{thm:Calderon-Zygmund}

Let $\Omega\subset\R^N$ be an open set. Any function $u\in BV(\Omega,\R^d)$ is approximately differentiable at $\mathcal{L}^N$-almost every point of $\Omega$. Moreover, the approximate differential $\nabla u$ is the density of the absolutely continuous part of $Du$ with respect to $\mathcal{L}^N$, in particular $\nabla u\in L^1(\Omega,\R^{d\times N})$.
\end{theorem}

\begin{theorem}(Federer-Vol'pert Theorem, Theorem 3.78 in \cite{AFP})
\label{thm: Federer-Vol'pert}

Let $\Omega\subset\R^N$ be an open set, and $u\in
BV_{\text{loc}}(\Omega,\R^d)$. Then, the jump set $\mathcal{J}_u$ is
countably $(N-1)-$rectifiable set, oriented with the jump vector
$\nu_u(x)$, and moreover, we have
$\mathcal{H}^{N-1}\big(\mathcal{S}_u\setminus \mathcal{J}_u\big)=0$. In particular, $\mathcal{S}_u$ is $\sigma$-finite with respect to $\mathcal{H}^{N-1}$.
\end{theorem}

\begin{lemma}(Variation Inequality)
\label{lem:variation inequality}

Let $\Omega\subset\R^N$ be an open set and $u\in BV(\Omega,\R^d)$. Let $E\subset\Omega$ be an $\mathcal{L}^N$-measurable set and let $h\in \R^N\setminus \{0\}$. Assume that $\dist(E,\partial \Omega)>|h|$. Then,
\begin{equation}
\label{eq:equation56}
\int_{E}\frac{|u(x+h)-u(x)|}{|h|}dx\leq \|Du\|(\Omega).
\end{equation}
In particular, if $\Omega=\R^N$, then
\begin{equation}
\label{eq:equation57}
\sup_{h\in \R^N\setminus\{0\}}\int_{\R^N}\frac{|u(x+h)-u(x)|}{|h|}dx\leq \|Du\|(\R^N).
\end{equation}
\end{lemma}

\begin{proof}
Let $\{u_k\}_{k=1}^\infty \subset C^1(\Omega,\R^d)$ be a sequence of functions which converges to $u$ $\mathcal{L}^N$-almost everywhere and $\lim_{k\to \infty}\|Du_k\|(\Omega)=\|Du\|(\Omega)$. Then, for every $k\in \N$, by the fundamental theorem of calculus and Fubini's theorem we get
\begin{multline}
\int_{E}\frac{|u_k(x+h)-u_k(x)|}{|h|}dx
=\int_{E}\frac{|\int_0^1\nabla u_k(x+th)\cdot h dt|}{|h|}dx\leq \int_0^1\int_{E}|\nabla u_k(x+th)|dxdt
\\
=\int_0^1\int_{E+th}|\nabla u_k(y)|dydt
\leq \|Du_k\|(\Omega).
\end{multline}
Taking the lower limit as $k\to \infty$  and using Fatou's Lemma we get \eqref{eq:equation56}. To get \eqref{eq:equation57} note that for every $h\in \R^N$, $\dist(\R^N,\emptyset)=\infty>|h|$.
\end{proof}

\subsection{Negligibility of Sets with respect to $\|Du\|$}
\begin{definition}(Measure-theoretic Boundary)
\label{def:measure-theoretic boundary}

Let $E\subset\R^N$ be a set. We write $x\in\partial^*E$ if and only if the following two inequalities hold:
\begin{align}
\limsup_{\e\to 0^+}\frac{\mathcal{L}^N\left(B_\e(x)\cap E\right)}{\mathcal{L}^N\left(B_\e(x)\right)}>0,\quad
\limsup_{\e\to 0^+}\frac{\mathcal{L}^N\left(B_\e(x)\cap \left(\R^N\setminus E\right)\right)}{\mathcal{L}^N\left(B_\e(x)\right)}>0.
\end{align}
Equivalently, $x\in\partial^*E$ if and only if $E$ and its complement $\R^N\setminus E$ do not have density $0$ at $x$; if and only if the set $E$ does not have density neither $0$ nor $1$. In other words, if we denote by $E^0$ the set of points at which $E$ has density $0$ and by $E^1$ the set of points at which $E$ has density $1$, namely
\begin{align}
E^0:=\bigg\{x\in\R^N:\lim_{\e\to 0^+}\frac{\mathcal{L}^N\left(B_\e(x)\cap E\right)}{\mathcal{L}^N\left(B_\e(x)\right)}=0\bigg\},\quad
E^1:=\bigg\{x\in\R^N:\lim_{\e\to 0^+}\frac{\mathcal{L}^N\left(B_\e(x)\cap E\right)}{\mathcal{L}^N\left(B_\e(x)\right)}=1\bigg\},
\end{align}
then $x\in\partial^*E$ if and only if $x\notin E^0\cup E^1$. We call $\partial^*E$ the {\it measure-theoretic boundary of the set $E$}.
\end{definition}

\begin{theorem}(The Co-Area Formula for $BV$-Functions, see equation (3.63) in \cite{AFP})
\label{thm:co-area formula for $BV$-functions}

Let $\Omega\subset\R^N$ be an open set, and let $u\in BV(\Omega)$. Then, for every Borel set $B\subset\Omega$
\begin{equation}
\|Du\|(B)=\int_{\R}\mathcal{H}^{N-1}\left(B\cap \partial^*\{u>t\}\right)d\mathcal{H}^{1}(t).
\end{equation}
\end{theorem}

\begin{proposition}(Variation-Negligibility of Sets with $\mathcal{H}^1$-Negligible Images)
\label{prop:variation-negligibility of sets with H1-negligible images}
Let $\Omega\subset \R^N$ be an open set  and $u\in BV(\Omega,\R^d)$. Let $B\subset\Omega\setminus \mathcal{S}_u$ be a Borel set such that $\mathcal{H}^1\left(\tilde{u}(B)\right)=0$. Then, $\|Du\|(B)=0$.
\end{proposition}

\begin{proof}
Assume first that $d=1$. Let us first prove that for every $t\in \R$ we have
\begin{equation}
\label{eq:image of measure-theoretic boundaries under approximate limits}
\tilde{u}\big(\left(\Omega\setminus \mathcal{S}_u\right)\cap \partial^*\big\{z\in \Omega:u(z)>t\big\}\big)\subset\{t\}.
\end{equation}
It means that the approximate limit $\tilde{u}$ takes the measure-theoretic boundaries of super-level sets $\partial^*\big\{z\in \Omega:u(z)>t\big\}$, which are outside $\mathcal{S}_u$, to the corresponding points $t$. We use the short notation $\{u>t\}:=\big\{z\in \Omega:u(z)>t\big\}$, as well as for similar sets. Assume that $z_0\in \left(\Omega\setminus \mathcal{S}_u\right)\cap \partial^*\{u>t\}$. Therefore, if $\tilde{u}(z_0)<t$, then for every $\e\in(0,\infty)$ we have by Chebyshev's inequality
\begin{multline}
\label{eq:expression for density}
\frac{\mathcal{L}^N\left(B_\e(z_0)\cap \{u>t\}\right)}{\mathcal{L}^N\left(B_\e(z_0)\right)}=\frac{\mathcal{L}^N\left(B_\e(z_0)\cap \{u-\tilde{u}(z_0)>t-\tilde{u}(z_0)\}\right)}{\mathcal{L}^N\left(B_\e(z_0)\right)}
\\
\leq \frac{1}{t-\tilde{u}(z_0)}\fint_{B_\e(z_0)}|u(x)-\tilde{u}(z_0)|dx.
\end{multline}
Since $z_0\in \Omega\setminus \mathcal{S}_u$, then we get from \eqref{eq:expression for density} that the density of $\{u>t\}$ at $z_0$ is zero, which contradicts the assumption that $z_0\in \partial^*\{u>t\}$. Similarly, if $\tilde{u}(z_0)>t$, then for every $\e\in(0,\infty)$ we have by Chebyshev's inequality
\begin{multline}
\label{eq:expression for density1}
\frac{\mathcal{L}^N\left(B_\e(z_0)\cap \{u\leq t\}\right)}{\mathcal{L}^N\left(B_\e(z_0)\right)}=\frac{\mathcal{L}^N\left(B_\e(z_0)\cap \{\tilde{u}(z_0)-u\geq \tilde{u}(z_0)-t\}\right)}{\mathcal{L}^N\left(B_\e(z_0)\right)}
\\
\leq \frac{1}{\tilde{u}(z_0)-t}\fint_{B_\e(z_0)}|u(x)-\tilde{u}(z_0)|dx.
\end{multline}
Since $z_0\in \Omega\setminus \mathcal{S}_u$, then we get from \eqref{eq:expression for density1} that the density of $\{u\leq t\}$ at $z_0$ is zero, which contradicts the assumption that $z_0\in \partial^*\{u>t\}$. We conclude that $\tilde{u}(z_0)=t$, which proves \eqref{eq:image of measure-theoretic boundaries under approximate limits}.

By \eqref{eq:image of measure-theoretic boundaries under approximate limits} we get that, if $t\notin \tilde{u}(B)$, then $B\cap \partial^*\{u>t\}=\emptyset$. We get from the co-area formula (Theorem \ref{thm:co-area formula for $BV$-functions}) and the assumption $\mathcal{H}^1\left(\tilde{u}(B)\right)=0$ that
\begin{equation}
\|Du\|(B)=\int_{\R}\mathcal{H}^{N-1}\left(B\cap \partial^*\{u>t\}\right)d\mathcal{H}^{1}(t)=\int_{\tilde{u}(B)}\mathcal{H}^{N-1}\left(B\cap \partial^*\{u>t\}\right)d\mathcal{H}^{1}(t)=0.
\end{equation}
In the general case, $d\in\N$, let us denote $u=(u_1,...,u_d)$. Notice that for every natural $1\leq j\leq d$ we have $\mathcal{S}_{u_j}\subset \mathcal{S}_{u}$, and for $x\in \mathcal{S}_{u}$ we have by uniqueness of approximate limit $\tilde{(u_j)}(x)=\left(\tilde{u}\right)_j(x)$. Therefore,
\begin{equation}
B\subset \Omega\setminus\mathcal{S}_{u}\subset \Omega\setminus\mathcal{S}_{u_j},\quad \mathcal{H}^1(\tilde{(u_j)}(B))=\mathcal{H}^1(\left(\tilde{u}\right)_j(B))=\mathcal{H}^1(P_j(\tilde{u}(B)))\leq \mathcal{H}^1(\tilde{u}(B))=0.
\end{equation}
Here $P_j:\R^d\to \R$ is the projection on the $j$-th coordinate which is a Lipschitz function. Therefore,
\begin{equation}
\|Du\|(B)\leq \sum_{j=1}^d\|Du_j\|(B)=0.
\end{equation}
\end{proof}

\begin{proposition}(Properties of Cantor Part $D^cu$, Proposition 3.92 in \cite{AFP})
\label{prop:Properties of $D^cu$}

Let $\Omega\subset\R^N$ be an open set, and let $u\in BV(\Omega,\R^d)$. Then, the Cantor part $D^cu$ (see Definition \ref{def:jump and Cantor parts}) of the distributional derivative $Du$ vanishes on sets which are $\sigma$-finite with respect to $\mathcal{H}^{N-1}$ and on sets of the form $\tilde{u}^{-1}(E)$ with $E\subset\R^d$, $\mathcal{H}^1(E)=0$.
\end{proposition}

\begin{remark}(Variation of Cantor Part Vanishes on $\mathcal{H}^{N-1}$ $\sigma$-Finite Sets)
\label{rem:variation of Cantor part vanishes on sigma finite sets w.r.t Hausdorff measure}

Since $D^cu$ vanishes on sets which are $\sigma$-finite with respect to $\mathcal{H}^{N-1}$, and any subset of such a set is also $\sigma$-finite with respect to $\mathcal{H}^{N-1}$, then the variation $\|D^cu\|$ vanishes on sets which are $\sigma$-finite with respect to $\mathcal{H}^{N-1}$ (recall that a variation of a vector valued measure $\mu$ vanishes on a set if and only if $\mu$ vanishes on every subset of the set).
\end{remark}

\subsection{Decomposition of $Du$ and the Chain Rule for $BV$-Functions}
\begin{definition}(Jump and Cantor Parts)
\label{def:jump and Cantor parts}

Let $\Omega\subset\R^N$ be an open set, and let $u\in BV(\Omega,\R^d)$. Let $Du=D^au+D^su$ be the decomposition of the distributional derivative $Du$ of $u$ into the absolutely continuous and  singular parts with respect to $\mathcal{L}^N$. We define the jump part and the Cantor part of $Du$, respectively, to be the following measures:
\begin{equation}
D^ju:=D^su\llcorner\mathcal{J}_u,\quad D^cu:=D^su\llcorner\left(\Omega\setminus \mathcal{S}_u\right).
\end{equation}
\end{definition}

\begin{theorem}(Decomposition of $Du$ into the Absolutely Continuous, Jump and Cantor Parts)
\label{thm:decomposition of Du into three parts}

Let $\Omega\subset\R^N$ be an open set, and let $u\in BV(\Omega,\R^d)$. Then,
\begin{equation}
Du=D^au+D^ju+D^cu,
\end{equation}
where $D^au,D^ju,D^cu$ are defined in Definition \ref{def:jump and Cantor parts}. They have the following properties:
\\
1. $D^au,D^ju,D^cu$ are finite Radon measures in $\Omega$ (it means that they are measures from $\mathcal{B}(\Omega)$, the Borel $\sigma$-algebra, into $\R^{d\times N}$, the set of all matrices of size $d\times N$ with entries from $\R$);
\\
2. They are orthogonal to each other;
\\
3. It follows that:
\begin{equation}
D^au=\nabla u\mathcal{L}^N,\quad D^ju=\left(u^+-u^-\right)\otimes\nu_u\mathcal{H}^{N-1}\llcorner\mathcal{J}_u,
\end{equation}
where for points $a=(a_1,...,a_d)\in\R^d,b=(b_1,...,b_N)\in\R^N$ we define $a\otimes b$ to be the $d\times N$ matrix given by $(a\otimes b)_{ij}:=a_ib_j$.
\\
4. We have
\begin{equation}
\|Du\|=|\nabla u|\mathcal{L}^N+\left|u^+-u^-\right|\mathcal{H}^{N-1}\llcorner\mathcal{J}_u+\|D^cu\|.
\end{equation}
\end{theorem}
One can find proofs for the assertions of Theorem \ref{thm:decomposition of Du into three parts} in section 3.9 in \cite{AFP}.

\begin{theorem}(Chain Rule in $BV$, Theorem 3.99 in \cite{AFP})
\label{thm:chain rule in BV}

Let $\Omega\subset\R^N$ be an open set. Let $u\in BV(\Omega,\R)$ and let $f:\R\to\R$ be a Lipschitz function satisfying $f(0)=0$ if $\mathcal{L}^N(\Omega)=\infty$. Then, $v:=f\circ u$ belongs to $BV(\Omega,\R)$ and
\begin{equation}
Dv=f'(u)\nabla u\mathcal{L}^N+\left(f(u^+)-f(u^-)\right)\nu_u\mathcal{H}^{N-1}\llcorner\mathcal{J}_u+f'(\tilde{u})D^c u.
\end{equation}
\end{theorem}

\begin{remark}(Well-Definedness of Compositions in Chain Rule for $BV$-Functions)
\label{rem:Well-definedness of compositions in Chain rule for $BV$-functions}
In this remark we would like to explain why $f'\circ u\in L^1(\Omega,|\nabla u|\mathcal{L}^N)$ and $f'\circ\tilde{u}\in L^1(\Omega,\|D^cu\|)$.
Let $\Omega\subset\R^N$ be an open set, and let $u\in BV(\Omega,\R)$. Let $f:\R\to \R$ be a Lipschitz function.
\\
1. By Rademacher's theorem there exists a Borel set $\Theta\subset\R$ such that $f$ is differentiable at every $x\in\R\setminus\Theta$ and $\mathcal{H}^1(\Theta)=0$.
\\
2. Since the approximate limit $\tilde{u}:\Omega\setminus\mathcal{S}_u\to \R$ is a Borel function, then $\tilde{u}^{-1}(\Theta)\subset \Omega\setminus\mathcal{S}_u$ is a Borel set.
\\
3. Therefore, we get by Proposition \ref{prop:variation-negligibility of sets with H1-negligible images} that $\|Du\|(\tilde{u}^{-1}(\Theta))=0$.
\\
4. Since $f$ is Lipschitz, then its $\mathcal{L}^1$-almost everywhere derivative $f':\R\setminus\Theta\to \R$ is a Borel function. Therefore, the composition $f'\circ\tilde{u}: \Omega\setminus\left(\mathcal{S}_u\cup \tilde{u}^{-1}(\Theta)\right) \to \R$ is a Borel function.
\\
5. Since by Remark \ref{rem:variation of Cantor part vanishes on sigma finite sets w.r.t Hausdorff measure} we have $\|D^cu\|(\mathcal{S}_u)=0$, then $f'\circ\tilde{u}$ is defined almost everywhere in $\Omega$ with respect to the measure $\|D^cu\|$. Since $\|D^cu\|$ is a Borel measure, then $f'\circ\tilde{u}$ is a measurable function with respect to the measure $\|D^cu\|$. Since $f$ is Lipschitz and $u\in BV(\Omega,\R)$, then
\begin{equation}
\int_{\Omega}|f'(\tilde{u}(x))|d\|D^cu\|(x)\leq \|f'\|_{L^\infty(\R)}\|D^cu\|(\Omega)\leq \|f'\|_{L^\infty(\R)}\|Du\|(\Omega)<\infty.
\end{equation}
Therefore, $f'\circ\tilde{u}\in L^1(\Omega,\|D^cu\|)$.
\\
6. Without loss of generality assume that the $\mathcal{L}^N$-almost everywhere defined function $u$ is defined on all of $\Omega$ and it is a Borel function. Let us denote by
\begin{equation}
E:=\mathcal{S}_u\cup \big\{x\in\Omega\setminus\mathcal{S}_u:u(x)\neq\tilde{u}(x)\big\}\cup \tilde{u}^{-1}(\Theta)\cup \left(\Omega\setminus \mathcal{D}_u\right).
\end{equation}
The function $f'\circ u:\Omega\setminus E\to\R$ is a Borel function because it is a restriction of the Borel function $f'\circ\tilde{u}$ to the Borel set $\Omega\setminus E$. Hence, it is a measurable function with respect to the Borel measure $|\nabla u|\mathcal{L}^N$. The function $f'\circ u$ is defined almost everywhere in $\Omega$ with respect to the measure $|\nabla u|\mathcal{L}^N$: by item $(a)$ of Proposition \ref{prop:properties of approximate limits} and Theorem \ref{thm:Calderon-Zygmund} we get that $|\nabla u|\mathcal{L}^N(E)=0$. Since $f$ is Lipschitz and $u\in BV(\Omega,\R)$, then
\begin{equation}
\int_{\Omega}|f'(u(x))|d|\nabla u|\mathcal{L}^N(x)\leq \|f'\|_{L^\infty(\R)}|\nabla u|\mathcal{L}^N(\Omega)\leq \|f'\|_{L^\infty(\R)}\|Du\|(\Omega)<\infty.
\end{equation}
Therefore, $f'\circ u\in L^1(\Omega,|\nabla u|\mathcal{L}^N)$.
\end{remark}

\subsection{Convergence of the Truncated Family in the Space $BV$}
\begin{lemma}(Convergence of the Truncated Family in Lebesgue Spaces)
\label{lem:convergence of the truncated in Lebesgue space}

Let $p\in(0,\infty)$, $E\subset\R^N$ be an $\mathcal{L}^N$-measurable set and $u\in L^p(E,\R^d)$. Then,
\begin{equation}
\label{eq:convergence of the truncated in Lebesgue space}
\lim_{l\to\infty}\int_{E}|u(z)-u_l(z)|^pd\mathcal{L}^N(z)=0,
\end{equation}
where $\{u_l\}_{l\in[0,\infty)}$ is the truncated family obtained by $u$.
\end{lemma}

\begin{proof}
Since for $\mathcal{L}^N$-almost every $z\in E$ $\lim_{l\to\infty}|u(z)-u_l(z)|=0$, $|u(z)-u_l(z)|\leq 2|u(z)|$, and $u\in L^p(E,\R^d)$, then we get \eqref{eq:convergence of the truncated in Lebesgue space} from Dominated Convergence Theorem.
\end{proof}

\begin{lemma}(Convergence of the Truncated Family in $BV$)
\label{lem:Convergence of the truncated family in BV}

Let $\Omega\subset\R^N$ be an open set and $u\in BV(\Omega,\R^d)$. Let $\{u_l\}_{l\in [0,\infty)}$ be the truncated family obtained by $u$. Then, for every $l\in [0,\infty)$ we have $u_l\in BV(\Omega,\R^d)$, and
\begin{equation}
\label{eq:convergence of truncated family in variation}
\lim_{l\to \infty}\|D(u-u_l)\|(\Omega)=0.
\end{equation}
In particular, $u_l$ converges to $u$ as $l\to \infty$ in the norm of the space $BV(\Omega,\R^d)$, which means that $\lim_{l\to \infty}\left(\|D(u-u_l)\|(\Omega)+\|u-u_l\|_{L^1(\Omega,\R^d)}\right)=0$.
\end{lemma}

\begin{proof}
Assume first that $u\in BV(\Omega,\R)$. By Theorem \ref{thm:decomposition of Du into three parts}, we can decompose the distributional derivative $Du$ into the sum of the absolutely continuous part, the jump part and the Cantor part:
\begin{equation}
\label{eq:decomposition of Du}
Du=\nabla u \mathcal{L}^N+(u^+-u^-)\nu_u\mathcal{H}^{N-1}\llcorner \mathcal{J}_u+D^c u.
\end{equation}
For each $l\in[0,\infty)$ let us define a function $f_l:\R\to \R,f_l(z):=l\wedge\left(-l\vee z\right)$. By Proposition \ref{prop:properties of the truncated family} we have that $f_l$ is a Lipschitz function and by the definition of the truncated family (Definition \ref{def:truncated family}) we have $u_l=f_l\circ u$. By the chain rule for $BV$-functions (refer to Theorem \ref{thm:chain rule in BV}) we have $u_l\in BV(\Omega,\R)$ and
\begin{equation}
\label{eq:decomposition of Dul}
Du_l=f'_l(u)\nabla u \mathcal{L}^N+(f_l(u^+)-f_l(u^-))\nu_u\mathcal{H}^{N-1}\llcorner \mathcal{J}_u+f'_l(\tilde{u})D^c u.
\end{equation}
By \eqref{eq:decomposition of Du}, \eqref{eq:decomposition of Dul} and Remark \ref{rem:Well-definedness of compositions in Chain rule for $BV$-functions} we have
\begin{multline}
\label{eq:difference of Dul-Du}
D(u-u_l)=\left(1-f'_l(u)\right)\nabla u\mathcal{L}^N
\\
+\left((u^+-u^-)-(f_l(u^+)-f_l(u^-))\right)\nu_u\mathcal{H}^{N-1}\llcorner \mathcal{J}_u
+(1-f'_l(\tilde{u}))D^c u.
\end{multline}
By Lemma \ref{lem: variation of multiplication of vector valued function with positive measure} we get
\begin{equation}
\label{eq:variation of the absolutely continuous part of D(u-ul)}
\|\left(1-f'_l(u)\right)\nabla u\mathcal{L}^N\|(\Omega)=\int_{\Omega}\left|1-f'_l(u(x))\right|\left|\nabla u(x)\right|d\mathcal{L}^N(x)
\end{equation}
and
\begin{multline}
\label{eq:variation of the jump part of D(u-ul)}
\|\left((u^+-u^-)-(f_l(u^+)-f_l(u^-))\right)\nu_u\mathcal{H}^{N-1}\llcorner \mathcal{J}_u\|(\Omega)
\\
=\int_{\mathcal{J}_u}\left|(u^+(x)-u^-(x))-(f_l(u^+(x))-f_l(u^-(x)))\right|d\mathcal{H}^{N-1}(x).
\end{multline}
Note that for getting \eqref{eq:variation of the jump part of D(u-ul)} we use that $|\nu_u|=1$ (refer to Definition \ref{def:approximate jump points}).
By Lemma \ref{lem:variation of multiplication of scalar function with vector valued measure} we get
\begin{align}
\label{eq:variation of the Cantor part of D(u-ul)}
\|(1-f'_l(\tilde{u}))D^c u\|(\Omega)\leq N^{1/2}\int_{\Omega}\left|1-f'_l(\tilde{u}(x))\right|d\|D^cu\|(x),
\end{align}
where $\|\cdot\|$ stands for the variation (refer to Definition \ref{def:variation of a measure}). Therefore, we get by \eqref{eq:difference of Dul-Du}, the triangle inequality of the variation, \eqref{eq:variation of the absolutely continuous part of D(u-ul)},\eqref{eq:variation of the jump part of D(u-ul)} and \eqref{eq:variation of the Cantor part of D(u-ul)} that
\begin{multline}
\label{eq:decomposition of variation of u-ul}
\|D(u-u_l)\|(\Omega)\leq \int_{\Omega}\left|1-f'_l(u(x))\right|\left|\nabla u(x)\right|d\mathcal{L}^N(x)
\\
+\int_{\mathcal{J}_u}\left|(u^+(x)-u^-(x))-(f_l(u^+(x))-f_l(u^-(x)))\right|d\mathcal{H}^{N-1}(x)
\\
+N^{1/2}\int_{\Omega}\left|1-f'_l(\tilde{u}(x))\right|d\|D^cu\|(x).
\end{multline}
For every $l\in[0,\infty)$, from item $(a)$ of Proposition \ref{prop:properties of approximate limits} we get
\begin{equation}
\label{eq:negligibility of level sets w.r.t absolutely continuous part}
|\nabla u|\mathcal{L}^N\left(\big\{x\in\Omega:|u(x)|=l\big\}\right)=|\nabla u|\mathcal{L}^N\left(\big\{x\in\Omega\setminus\mathcal{S}_u:|\tilde{u}(x)|=l\big\}\right),
\end{equation}
and from Proposition \ref{prop:variation-negligibility of sets with H1-negligible images} we have
\begin{equation}
\label{eq:negligibility of level sets w.r.t Cantor part}
\|D^cu\|\left(\big\{x\in\Omega\setminus \mathcal{S}_u:|\tilde{u}(x)|=l\big\}\right)=|\nabla u|\mathcal{L}^N\left(\big\{x\in\Omega\setminus\mathcal{S}_u:|\tilde{u}(x)|=l\big\}\right)=0.
\end{equation}
Note that for getting \eqref{eq:negligibility of level sets w.r.t Cantor part} we use the assumption that $u$ is a scalar function in order to get that
\begin{equation}
\mathcal{H}^1\left(\tilde{u}\left(E_l\right)\right)\leq \mathcal{H}^1(\{l,-l\})=0, \quad E_l:=\big\{x\in\Omega\setminus \mathcal{S}_u:|\tilde{u}(x)|=l\big\}.
\end{equation}
For every $l\in(0,\infty)$ we have
\begin{equation}
\label{eq:calculation of derivative of fl}
f_l'(z):=
\begin{cases}
1,& if\quad  |z|<l\\
0,& if\quad |z|>l\\
\end{cases}.
\end{equation}
By \eqref{eq:negligibility of level sets w.r.t absolutely continuous part} and \eqref{eq:calculation of derivative of fl} we get for every $l\in(0,\infty)$ that
\begin{multline}
\label{eq:calculation for the varation absolutely continuous part}
\int_{\Omega}\left|1-f'_l(u(x))\right|\left|\nabla u(x)\right|d\mathcal{L}^N(x)=
\int_{\Omega}\left|1-f'_l(u(x))\right|d\left|\nabla u\right|\mathcal{L}^N(x)
\\
=\int_{\big\{x\in\Omega:|u(x)|>l\big\}}\left|1-f'_l(u(x))\right|d\left|\nabla u\right|\mathcal{L}^N(x)
+\int_{\big\{x\in\Omega:|u(x)|=l\big\}}\left|1-f'_l(u(x))\right|d\left|\nabla u\right|\mathcal{L}^N(x)
\\
+\int_{\big\{x\in\Omega:|u(x)|<l\big\}}\left|1-f'_l(u(x))\right|d\left|\nabla u\right|\mathcal{L}^N(x)
=\int_{\big\{x\in\Omega:|u(x)|>l\big\}}d\left|\nabla u\right|\mathcal{L}^N(x).
\end{multline}
By Calder{\'o}n-Zygmund theorem (refer to Theorem \ref{thm:Calderon-Zygmund}), we have $\nabla u\in L^1(\Omega,\R^N)$. Therefore, we get by \eqref{eq:calculation for the varation absolutely continuous part} and the decreasing monotonicity of the measure $\left|\nabla u\right|\mathcal{L}^N$ that
\begin{equation}
\label{eq:abloutely continuous part converges to zero}
\lim_{l\to\infty}\int_{\Omega}\left|1-f'_l(u(x))\right|\left|\nabla u(x)\right|d\mathcal{L}^N(x)=0.
\end{equation}
Since $u\in BV(\Omega,\R)$, then we get from Federer-Vol'pert Theorem (refer to Theorem \ref{thm: Federer-Vol'pert}) that $\mathcal{S}_u$ is $\sigma$-finite with respect to $\mathcal{H}^{N-1}$. Thus, by Proposition \ref{prop:Properties of $D^cu$} and Remark \ref{rem:variation of Cantor part vanishes on sigma finite sets w.r.t Hausdorff measure} we have $\|D^cu\|(\mathcal{S}_u)=0$. Therefore, by \eqref{eq:negligibility of level sets w.r.t Cantor part} and \eqref{eq:calculation of derivative of fl} we obtain
\begin{multline}
\label{eq:calculation for the varation Cantor part}
\int_{\Omega}\left|1-f'_l(\tilde{u}(x))\right|d\|D^cu\|(x)=\int_{\Omega\setminus\mathcal{S}_u}\left|1-f'_l(\tilde{u}(x))\right|d\|D^cu\|(x)
\\
=\int_{\big\{x\in\Omega\setminus\mathcal{S}_u:|\tilde{u}(x)|>l\big\}}\left|1-f'_l(\tilde{u}(x))\right|d\|D^cu\|(x)
+\int_{\big\{x\in\Omega\setminus\mathcal{S}_u:|\tilde{u}(x)|=l\big\}}\left|1-f'_l(\tilde{u}(x))\right|d\|D^cu\|(x)
\\
+\int_{\big\{x\in\Omega\setminus\mathcal{S}_u:|\tilde{u}(x)|<l\big\}}\left|1-f'_l(\tilde{u}(x))\right|d\|D^cu\|(x)=\int_{\big\{x\in\Omega\setminus\mathcal{S}_u:|\tilde{u}(x)|>l\big\}}1d\|D^cu\|(x).
\end{multline}
Note that since $\mathcal{S}_u$ is a Borel set in $\Omega$ and $\tilde{u}:\Omega\setminus \mathcal{S}_u\to \R$ is a Borel function (refer to Proposition \ref{prop:properties of approximate limits}), then the sets $\big\{x\in\Omega\setminus\mathcal{S}_u:|\tilde{u}(x)|>l\big\}$, $\big\{x\in\Omega\setminus\mathcal{S}_u:|\tilde{u}(x)|<l\big\}$ and $\big\{x\in\Omega\setminus\mathcal{S}_u:|\tilde{u}(x)|=l\big\}$ are Borel sets in $\Omega$, so they are measurable with respect to the measure $\|D^cu\|$, because $\|D^cu\|$ is a Borel measure (refer to Theorem \ref{thm:decomposition of Du into three parts}). Since $\|D^cu\|$ is a finite Borel measure in $\Omega$, then we get by \eqref{eq:calculation for the varation Cantor part} and the decreasing  monotonicity of the measure $\|D^cu\|$ that
\begin{multline}
\label{eq:Cantor part converges to zero}
\lim_{l\to \infty}\int_{\Omega}\left|1-f'_l(\tilde{u}(x))\right|d\|D^cu\|(x)=\lim_{l\to\infty}\|D^cu\|\left(\big\{x\in\Omega\setminus\mathcal{S}_u:|\tilde{u}(x)|>l\big\}\right)
\\
=\|D^cu\|\left(\bigcap_{l\in\N}\big\{x\in\Omega\setminus\mathcal{S}_u:|\tilde{u}(x)|>l\big\}\right)=\|D^cu\|(\emptyset)=0.
\end{multline}
At last, by Proposition \ref{prop:properties of the truncated family} we get
\begin{align}
&\lim_{l\to\infty}\left|(u^+(x)-u^-(x))-(f_l(u^+(x))-f_l(u^-(x)))\right|=0,\quad x\in\mathcal{J}_u;
\\
&\left|(u^+(x)-u^-(x))-(f_l(u^+(x))-f_l(u^-(x)))\right|\leq 2\left|u^+(x)-u^-(x)\right|,\quad x\in\mathcal{J}_u.
\end{align}
By Theorem \ref{thm:decomposition of Du into three parts} we get
\begin{equation}
\left|u^+-u^-\right|\in L^1(\mathcal{J}_u,\mathcal{H}^{N-1}).
\end{equation}
Therefore, Dominated Convergence Theorem gives
\begin{equation}
\label{eq:jump part converges to zero}
\lim_{l\to\infty}\int_{\mathcal{J}_u}\left|(u^+(x)-u^-(x))-(f_l(u^+(x))-f_l(u^-(x)))\right|d\mathcal{H}^{N-1}(x)=0.
\end{equation}
Equation \eqref{eq:convergence of truncated family in variation} follows from \eqref{eq:decomposition of variation of u-ul}, \eqref{eq:abloutely continuous part converges to zero}, \eqref{eq:Cantor part converges to zero} and \eqref{eq:jump part converges to zero} in case $u\in BV(\Omega,\R)$. The general case, $u\in BV(\Omega,\R^d)$,  follows from the inequality
\begin{equation}
\label{eq:inequality for coordinate functions of BV}
\max_{1\leq i\leq d,i\in\N}\|Du^i\|(\Omega)\leq \|Du\|(\Omega)\leq \sum_{i=1}^d\|Du^i\|(\Omega),
\end{equation}
where $u=(u^1,...,u^d)$. Indeed, note that by the definition of the truncated family, Definition \ref{def:truncated family}, it follows that $(u_l)^i=(u^i)_l$ for every natural $1\leq i\leq d$ and $l\in[0,\infty)$. Therefore, we get
\begin{equation}
\label{eq:estimate for variation of u-ul}
\|D(u-u_l)\|(\Omega)\leq \sum_{i=1}^d\|D(u-u_l)^i\|(\Omega)=\sum_{i=1}^d\|D(u^i-(u^i)_l)\|(\Omega).
\end{equation}
Therefore, since for every natural $1\leq i\leq d$ we have that $u^i\in L^1(\Omega,\R)$, then we get by \eqref{eq:inequality for coordinate functions of BV} that $u^i\in BV(\Omega,\R)$. Therefore, we obtain \eqref{eq:convergence of truncated family in variation} from \eqref{eq:estimate for variation of u-ul} taking the limit as $l$ goes to infinity. The convergence of $u_l$ to $u$ as $l\to\infty$ in the norm of the space $BV(\Omega,\R^d)$ follows from Lemma \ref{lem:convergence of the truncated in Lebesgue space} with $p=1$ and \eqref{eq:convergence of truncated family in variation}.
\end{proof}


\vskip 0.3cm
\begin{thebibliography}{99}

\bibitem{AFP} L.~Ambrosio, N.~Fusco, D.~Pallara, {\it Functions of Bounded Variation and Free Discontinuity Problems},
Oxford Mathematical Monographs. The Clarendon Press, Oxford
University Press, New York (2000).

\bibitem{BBM} J.~Bourgain, H.~Brezis, P.~ Mironescu. {\it Another look at Sobolev spaces}, Optimal Control and Partial Differential Equations, IOS Press ISBN 1 58603 096 5, (2001): 439-455.

\bibitem{Br}
J.~Brasseur, {\it A Bourgain--Brezis--Mironescu characterization of
higher order Besov-Nikol'skii spaces}, Annales de l'Institut
Fourier. Vol. 68. No. 4. 2018.

\bibitem{Davila} J. D{\'a}vila, {\it On an open question about functions of bounded variation}, Calculus of Variations and Partial Differential Equations {\bf 15} (2002): 519-527.


\bibitem{FSF}
E.~Di Nezza, G.~Palatucci and E.~Valdinoci, {\it Hitchhiker's guide
to the fractional Sobolev spaces}, Bulletin des sciences
math\'{e}matiques {\bf 136} (2012): 521-573.

\bibitem{EG}
L.~C.~Evans, R.~F.~Gariepy, {\it Measure theory and fine properties
of functions}, CRC Press (2015).

\bibitem{F} H.~Federer, {\it Geometric measure theory}, Sp\-rin\-ger Verlag, Berlin (1969).


\bibitem{AFDJ} A. Figalli and D. Jerison, {\it How to recognize convexity of a
set from its marginals}, Journal of Functional Analysis, {\bf 266}
(3), 1685-1701 (2014).


\bibitem{PA}
P.~Hashash, A~Poliakovsky, {\it Jumps in Besov spaces and fine
properties of Besov and fractional Sobolev functions}, Calculus of
Variations and Partial Differential Equations,  {\bf 63}, Issue 2,
Article number 28 (2024), https://doi.org/10.1007/s00526-023-02630-3

\bibitem{FHHalp} F.~Hern\'{a}ndez, {\it Some Properties of a Hilbertian Norm for
Perimeter}, Pure and Applied Functional Analysis, {\bf 4}, no. 3,
559-572, (2019).

\bibitem{Giovanni}
G.~Leoni, {\it A first course in Sobolev spaces}, American
Mathematical Soc. (2017).

\bibitem{PAsymptotic}
A.~Poliakovsky, {\it Asymptotic behavior of $W^{1/q,q}$-norm of
mollified $BV$ function and its application to singular perturbation
problems}, ESAIM Control Optimisation and Calculus of Variations,
{\bf 26} (2020), Paper No. 77, 20 pp.


\bibitem{P} A.~Poliakovsky, {\it Jump detection in Besov spaces via a new BBM formula. Applications to Aviles-Giga type functionals}, Commum. Contemp. Math. \textbf{20} (2018), no. 7, 1750096, 36 pp.


\bibitem{Poliakovsky}
A.~Poliakovsky, {\it Some remarks on a formula for Sobolev norms due
to Brezis, Van Schaftingen and Yung},  J. Funct. Anal. {\bf 282}
(2022), no. 3, Paper No. 109312, 47 pp.
\end{thebibliography}
\end{document}